\documentclass[sn-mathphys,Numbered]{sn-jnl}

\usepackage{graphicx}%
\usepackage{multirow}%
\usepackage{amsmath,amssymb,amsfonts}%
\usepackage{amsthm}%
\usepackage{mathrsfs}%
\usepackage[title]{appendix}%
\usepackage{xcolor}%
\usepackage{textcomp}%
\usepackage{manyfoot}%
\usepackage{booktabs}%
\usepackage{algorithm}%
\usepackage{algorithmicx}%
\usepackage{algpseudocode}%
\usepackage{listings}%
\usepackage[utf8]{inputenc}
\usepackage{amsmath}
\usepackage{amsfonts}
\usepackage{amssymb}
\usepackage{latexsym}
\usepackage{centernot}
\usepackage{mathrsfs}
\usepackage{mathtools}
\usepackage{stmaryrd}
\usepackage{graphicx} 
\usepackage[rightcaption]{sidecap}
\usepackage{wrapfig}
\usepackage{subfig}
\usepackage{comment}
\usepackage{mathdots}
\usepackage{amsthm}
\usepackage{hyperref}

\usepackage{tikz-cd}
\usepackage{tikz}

\usepackage{a4wide}
\usepackage{faktor}

\DeclareMathOperator{\diag}{diag}
\DeclareMathOperator{\Res}{Res}
\DeclareMathOperator{\Sym}{Sym}
\DeclareMathOperator\Spec{Spec}

\usepackage[english]{babel}

\theoremstyle{plain}%
\newtheorem{theorem}{Theorem}
\newtheorem*{theorem*}{Theorem}

\newtheorem{conjecture}[theorem]{Conjecture}
\newtheorem{corollary}[theorem]{Corollary}

\newtheorem{definition}[theorem]{Definition}
\newtheorem*{definition*}{Definition}
\newtheorem*{example*}{Example}

\newtheorem*{exercise*}{Exercise}
\newtheorem{formula}[theorem]{Formula}
\newtheorem{lemma}[theorem]{Lemma}

\newtheorem*{problem*}{Problem}
\newtheorem{proposition}[theorem]{Proposition}

\newtheorem{remark}[theorem]{Remark}
\newtheorem*{recall*}{Recall}

\newtheorem{fact}[theorem]{Fact}
\begin{document}

\title[On the Periods of Twisted Moments of the Kloosterman Connection]{On the Periods of Twisted Moments of the Kloosterman Connection}


\author*[1]{\fnm{Ping-Hsun} \sur{Chuang}}\email{f09221011@ntu.edu.tw}

\author[1]{\fnm{Jeng-Daw} \sur{Yu}}\email{jdyu@ntu.edu.tw}
\equalcont{The author contributed to the appendix.}

\affil*[1]{\orgdiv{Department of Mathematics}, \orgname{National Taiwan University}, \orgaddress{\city{Taipei}, \postcode{10617},  \country{Taiwan}}}


\abstract{This paper aims to study the Betti homology and de Rham cohomology of twisted symmetric powers of the Kloosterman connection of rank two on the torus. We compute the period pairing and, with respect to certain bases, interpret these associated period numbers in terms of the Bessel moments. Via the rational structures on Betti homology and de Rham cohomology, we prove the $\mathbb{Q}$-linear and quadratic relations among these Bessel moments.}

\keywords{Periods, Bessel moments, Kloosterman connection}


\pacs[MSC Classification]{32G20, 33C10, 34M35}

\maketitle


\section*{Statements and Declarations}

This research was supported by the grant 108-2115-M-002-003-MY2 of the Ministry of Science and Technology. The authors have no relevant financial or non-financial interests to disclose.


\section{Introduction}
Let $\mathbb{G}_{\mathrm{m},z}=\Spec\,(\mathbb{Q}[z,z^{-1}])$
be the algebraic torus over $\mathbb{Q}$ with variable $z$,
and similarly for the torus
$\mathbb{G}_{\mathrm{m},t}$ with variable $t$.
Let ${\mathrm{Kl}}_{2}$ be the Kloosterman connection (of rank two) on
$\mathbb{G}_{\mathrm{m},z}$ corresponding to the differential operator $(z\partial_{z})^{2}-z$. (For details, see section \ref{section_recall_and_pairing}.) In \cite{MR4432013}, in order to study the
Hodge aspects of the symmetric powers $\Sym^{k}{\mathrm{Kl}}_{2}$, Fres\'{a}n, Sabbah, and Yu consider the following settings. Let $\left[2\right]:\mathbb{G}_{\mathrm{m},t}\rightarrow \mathbb{G}_{\mathrm{m},z}$ be the double cover induced by the ring homomorphism $\mathbb{Q}[z,z^{-1}]\rightarrow\mathbb{Q}[t,t^{-1}]$, given by $z\mapsto t^{2}$. One obtains the pullback connection
\[
\widetilde{{\mathrm{Kl}}}_{2}=\left[2\right]^{+}{\mathrm{Kl}}_{2}.
\]
The structure of $\widetilde{{\mathrm{Kl}}}_{2}$ is much simpler since it is the restriction to $\mathbb{G}_{\mathrm{m}}$ of the Fourier transform of a regular holonomic module on the affine line. In addition, the symmetric power $\Sym^{k}{\mathrm{Kl}}_{2}$ appears in the pushforward $\left[2\right]_{+}\Sym^{k}{\widetilde{\mathrm{Kl}}}_{2}$ naturally in the decomposition \cite[p. 1662]{MR4432013}
\[
\left[2\right]_{+}\Sym^{k}{\widetilde{\mathrm{Kl}}}_{2}\cong \Sym^{k}{\mathrm{Kl}}_{2}\oplus \sqrt{z}\Sym^{k}{\mathrm{Kl}}_{2},
\]
where $\sqrt{z}\Sym^{k}{\mathrm{Kl}}_{2}=\left(\mathcal{O}_{\mathbb{G}_{\mathrm{m}}},{\mathrm{d}}+\frac{{\mathrm{d}}z}{2z}\right)\otimes \Sym^{k}{\mathrm{Kl}}_{2}$. In \cite{MR4578001}, Fres\'{a}n, Sabbah, and Yu compute the de Rham cohomology and Betti homology for $\Sym^{k}{\mathrm{Kl}}_{2}$. In this paper, we study the analogues for $\sqrt{z}\Sym^{k}{\mathrm{Kl}}_{2}$.

\subsection{Historical results and our results}
Let $I_{0}(t)$ and $K_{0}(t)$ be modified Bessel functions. Define the Bessel moments
\begin{align}\label{Bessel_moments}
{\mathrm{IKM}}_{k}(a,b)=\int_{0}^{\infty}I_{0}(t)^{a}K_{0}(t)^{k-a}t^{b}{\mathrm{d}}t
\end{align}
provided that $0\leq a\leq k$ are non-negative integers, $b\in\mathbb{Z}$, and the convergence of this integral. The particular Bessel moments of the form
${\mathrm{IKM}}_{a+b}(a,2c-1)$
appear in two-dimensional quantum field theory as Feynman integrals \cite{MR3573666,MR2450513,Broadhurst:2016hbq}. From a mathematical point of view, these moments are realized as period integrals of $\Sym^k\mathrm{Kl}_{2}$ and $\sqrt{z}\Sym^{k}\mathrm{Kl}_{2}$. For the details, we refer to \cite{MR4578001}. In that paper, Fres\'{a}n, Sabbah, and Yu developed the Hodge theory on
symmetric powers of the generalized Kloosterman connection $\mathrm{Kl}_{n+1}$ of rank $(n+1)$. 

\subsubsection*{Sum rule identities} In \cite[(220)]{MR2450513}, the authors provide the following conjecture on the $\mathbb{Q}(\pi)$ linear relation of Bessel moments which is called
the ``sum rule'' in their paper.
\begin{conjecture}
For each pair of integers $\left(n,k\right)$ with $n\geq 2k\geq 2$, the following combination of Bessel moments vanish
\[
\sum_{m=0}^{\lfloor n/2\rfloor}(-1)^{m}\binom{n}{2m}\pi^{n-2m}{\mathrm{IKM}}_{2n}(n-2m,n-2k)=0.
\]
\end{conjecture}

Later in \cite[(1.5)]{MR3902501}, Zhou uses the Hilbert transformation to prove this conjecture. Moreover, he also proves a ``sum rule'':
\begin{formula}
For each pair of integers $\left(n,k\right)$ with $n-1\geq 2k\geq 2$, the following conbination of Bessel moments vanish
\[
\sum_{m=1}^{\lfloor (n+1)/2\rfloor}(-1)^{m}\binom{n}{2m-1}\pi^{n-2m+1}{\mathrm{IKM}}_{2n}(n-2m+1,n-2k-1)=0.
\]
\end{formula}

When the involved exponents of $t$ in \eqref{Bessel_moments} are odd, these two identities are both reproved by Fres\'{a}n, Sabbah, and Yu \cite{MR4578001}. The proof is to study the connection $\Sym^{k}\mathrm{Kl}_{2}$ on $\mathbb{G}_{\mathrm{m}}$ whose period integrals are those Bessel moments $\mathrm{IKM}_{k}(a,b)$ with odd $b$. In this paper, by studying the connection $\sqrt{z}\Sym^{k}\mathrm{Kl}_{2}$ on $\mathbb{G}_{\mathrm{m}}$, we provide proofs of these two identities involving even powers of $t$ using a similar approach in section \ref{section_period_pairing}. The key point to consider the twisted connection $\sqrt{z}\Sym^{k}\mathrm{Kl}_{2}$ is that the period integrals of $\sqrt{z}\Sym^{k}\mathrm{Kl}_{2}$ are those Bessel moments $\mathrm{IKM}_{k}(a,b)$ with even $b$. For example, we have the following result:

\begin{formula}[Corollary \ref{linear_relation_Bessel_moments}]\label{formula:Q-linear_relation}
For $k=4r+4$, a multiple of $4$, one has
\[
\sum_{j=0}^{r}\binom{k/2}{2j}(-1)^{j}\pi^{2j}{\mathrm{IKM}}_{k}(2j,2i)
= \begin{cases}
(-1)^r\pi^{2r+2}{\mathrm{IKM}}_{k}(2r+2,2i) & \text{if $0\leq i\leq r$}, \\
(-1)^r\pi^{2r+2}{\mathrm{IKM}}_{k}^{\mathrm{reg}}(2r+2,2i)
	& \text{if $r+1\leq i\leq \lfloor\frac{k-1}{2}\rfloor$}.
\end{cases}
\]
\end{formula}
The notation $\mathrm{IKM}^{\mathrm{reg}}_{k}(2r+2,2i)$ above denotes the regularized Bessel moments (see Lemma \ref{regularized_lemma}). Roughly speaking, the regularized Bessel moments are obtained from those Bessel moments $\mathrm{IKM}_{k}(a,b)$ with parameters $k,a,b$ that makes $\mathrm{IKM}_{k}(a,b)$ diverge but minus their divergent asymptotic. Therefore, our sum rule generalizes the sum rules in \cite{MR3902501, MR2450513}.

\subsubsection*{$\mathbb{Q}$-dimension of Bessel moments} In \cite{MR4437425}, Zhou considers the $\mathbb{Q}$-vector subspace spanned by the Bessel moments in $\mathbb{C}$. This vector subspace is finite-dimensional due to the sum rule. Similarly, we have the following upper bound of the dimension.
\begin{theorem}[Corollary \ref{cor_bessel_dim_upper_bound}]\label{theorem:Q-dimension}
For any $k$ and any $0\leq a\leq \lfloor (k-1)/2\rfloor$,
the dimension of the $\mathbb{Q}$-vector space generated by the Bessel moments has an upper bound:
\[
\dim {\mathrm{span}}_{\mathbb{Q}}\left\{{\mathrm{IKM}}_{k}(a,2j)\mid j\in\left\{0\right\}\cup \mathbb{N}\right\}\leq \lfloor (k+1)/2 \rfloor.
\]
For $k$ even, the dimension of the $\mathbb{Q}$-vector space generated by the regularized Bessel moments has an upper bound:
\[
\dim {\mathrm{span}}_{\mathbb{Q}}\left\{{\mathrm{IKM}}^{\mathrm{reg}}_{k}\left(k/2,2j\right)\mid j\in\left\{0\right\}\cup \mathbb{N}\right\}\leq \lfloor (k+1)/2\rfloor.
\]
\end{theorem}
Note that our statement involves the regularized Bessel moments. This conclusion is a more general result than the one given by Zhou.

\subsubsection*{Quadratic relations of Bessel moments}
In \cite{MR4601771}, the authors prove a general result of quadratic relations between periods given by a self-dual connection. We apply this result and obtain the quadratic relation between the Bessel moments. Under certain bases of cohomologies, let $B$ be the topological pairing matrix on Betti homology, $D$ be the Poincar\'{e} pairing matrix on de Rham cohomology, and $P,P_{\mathrm{c}}$ be the period pairing matrices between these two homology and cohomology. Then, the quadratic relations on these periods (Bessel moments) are given by
\begin{align}\label{quadratic_relation_formula}
PD^{-1}P_{\mathrm{c}}=(-1)^{k}(2\pi\sqrt{-1})^{k+1}B.
\end{align}
The entries of the matrices $P$ and $P_{\mathrm{c}}$ consist of $\mathbb{Q}$-linear combinations of Bessel moments and regularized Bessel moments,
which are obtained in Section \ref{section_period_pairing}.
Moreover, due to the rational structure of Betti homology and de Rham cohomology, the corresponding pairing matrices $D, B$ consist of rational numbers.

Note that the quadratic relations among Bessel moments does not depend on the choices of the bases. The effect of changing the bases of these homologies and cohomologies is just the conjugation of the matrices $P,D,P_{\mathrm{c}},B$ and thus on \eqref{quadratic_relation_formula}.

\subsubsection*{Determinants of Bessel moment matrix} Another interesting result is to compute the determinants of certain matrices consisting of Bessel moments. In \cite[Conjectures 4 and 7]{MR3573666}, Broadhurst conjectures closed formulae of the determinants of the following two $r\times r$ matrices $\mathbf{M}_{r}$ and $\mathbf{N}_{r}$ involving the Bessel moments:
\[ \mathbf{M}_{r} =\big({\mathrm{IKM}}_{2r+1}(a,2b-1)\big)_{1\leq a,b\leq r},
\quad
\mathbf{N}_{r} =\big({\mathrm{IKM}}_{2r+2}(a,2b-1)\big)_{1\leq a,b\leq r}. \]
Later, in \cite{MR3824579}, Zhou uses an analytic method to prove these two determinant formulae. Using a similar method as Zhou, we give explicit determinant formulae:
\begin{formula}[Corollary \ref{period_determinant}]\label{formula:determinant_formula}
For $r\geq 1$, we have
\begin{align*}
\det \big({\mathrm{IKM}}_{2r-1}(i-1,2j-2)\big)_{1\leq i,j\leq r}
&=\sqrt{\pi}^{r(r+1)}\sqrt{2}^{r(r-3)}\prod_{a=1}^{r-1}\dfrac{a^{r-a}}{\sqrt{2a+1}^{2a+1}},\\
\det \big({\mathrm{IKM}}_{2r}(i-1,2j-2)\big)_{1\leq i,j\leq r}
&=\dfrac{\sqrt{\pi}^{(r+1)^{2}}}{\Gamma\left(\frac{r+1}{2}\right)}\dfrac{1}{\sqrt{2}^{r(r+3)}}\prod_{a=1}^{r-1}\dfrac{(2a+1)^{r-a}}{(a+1)^{a+1}}.
\end{align*}
\end{formula}

\subsection{Approach}

In \cite{MR2143558}, Bloch and Esnault study irregular connections on curves and provide the associated homology theory. Due to their results, we study the de Rham cohomology and Betti homology of $\sqrt{z}\Sym^{k}{\mathrm{Kl}}_{2}$ on $\mathbb{G}_{\mathrm{m}}$
and provide explicit bases in order to find the periods.

In section \ref{section_recall_and_pairing}, we introduce the twisted $k$-th symmetric power of the Kloosterman connection $\sqrt{z}\Sym^{k}{\mathrm{Kl}}_{2}$, which is the main object in this paper. We discuss the rational structures on the de Rham cohomology and Betti homology of the connection. Moreover, since the connection is self-dual, we introduce its algebraic and topological self-pairings. These pairings will play an important role in our computations.

In section \ref{section_de_Rham_side}, we study the de Rham cohomology and the de Rham cohomology with compact support of $\sqrt{z}\Sym^{k}{\mathrm{Kl}}_{2}$ and write down certain elements in these two cohomologies. Next, we introduce the Poincar\'{e} pairing between them and compute the pairing with respect to the elements we have constructed. Using the dimension result of de Rham cohomology, along with the non-vanishing determinant of the Poincar\'{e} pairing, in Corollary \ref{de_rham_basis}, we conclude that the explicit elements in de Rham cohomology form bases.

We study parallelly the Betti homology of $\sqrt{z}\Sym^{k}{\mathrm{Kl}}_{2}$ in section \ref{section_Betti_side}. Since our ambient space is a non-compact space $\mathbb{C}^{\times}$, we need to modify our Betti homology theory by allowing the chain to go to $0$ or $\infty$. By controlling the growth behaviors of the horizontal sections, we study the moderate decay Betti homology and rapid decay Betti homology on $\sqrt{z}\Sym^{k}{\mathrm{Kl}}_{2}$. Similarly, We first write down some elements in the moderate decay homology and rapid decay homology and compute their topological pairing explicitly. Moreover, by the duality of de Rham cohomology and Betti homology, the dimension of Betti homology is the same as the de Rham cohomology. Together with the topological pairing, we conclude that they are bases in Corollary \ref{Betti_basis_assume_dim}.

Finally, in section \ref{section_period_pairing}, we compute the period pairing between the de Rham cohomologies and the Betti homologies and interpret them in terms of the Bessel moments. Note that our variety $\mathbb{G}_{\mathrm{m}}=\Spec\mathbb{Q}[z,z^{-1}]$ and the connection $\sqrt{z}\Sym^{k}{\mathrm{Kl}}_{2}$ are defined over $\mathbb{Q}$ and therefore, the de Rham cohomology and Betti homology are naturally endowed with a $\mathbb{Q}$-vector space structure. From the dimension constraint of homologies, after computing the period pairing, we obtain the $\mathbb{Q}$-linear relation of Bessel moments (Formula \ref{formula:Q-linear_relation}) and an upper bound of $\mathbb{Q}$-dimension of space spanned by the Bessel moments (Theorem \ref{theorem:Q-dimension}). In addition, the self duality of $\sqrt{z}\Sym^{k}{\mathrm{Kl}}_{2}$ gives quadratic relations between these Bessel moments \eqref{quadratic_relation_formula}.

In appendix \ref{section_symmetric_power_of_differential_operator}, we provide an accurate analysis of the symmetric powers of the modified Bessel differential operator. The first usage belongs to section \ref{section_de_Rham_side}, which enables us to determine the dimension of the de Rham cohomology $H^{1}_{\mathrm{dR}}\big(\mathbb{G}_{\mathrm{m}},\sqrt{z}\Sym^{k}{\mathrm{Kl}}_{2}\big)$. The second usage belongs to appendix \ref{section_period_determinant}, which allows us to analyze the leading term of the Vanhove operator. This helps us to obtain the determinant formula (Formula \ref{formula:determinant_formula}). 

\section{The Kloosterman connection and its twisted symmetric powers}\label{section_recall_and_pairing}

In this section, we recall the definition and basic properties of Kloosterman connection and its symmetric powers from \cite{MR4432013,MR4578001}. Besides, we recall the twisted connection on $\mathbb{G}_{\mathrm{m}}$ obtained from the decomposition of the pushforward of trivial connection under the cyclic cover of $\mathbb{G}_{\mathrm{m}}$. Combining these connections, we obtain the twisted symmetric powers of the Kloosterman connection. Moreover, since these connections are all self-dual, the duality induces the algebraic pairings on them and the topological pairings on the sheaves of horizontal sections. 

\subsection{Self-duality and pairing on $\mathrm{Kl}_{2}$}\label{section_recall_dual_and_pairing_Kl2}

The connection ${\mathrm{Kl}}_{2}=(\mathcal{O}_{\mathbb{G}_{\mathrm{m},z}}v_{0}\oplus\mathcal{O}_{\mathbb{G}_{\mathrm{m},z}}v_{1},\nabla)$
consists of a rank $2$ free sheaf on $\mathbb{G}_{\mathrm{m},z}=\Spec\mathbb{Q}[z,z^{-1}]$ with basis of sections $v_{0}$ and $v_{1}$ and the connection $\nabla$ on it given by
\begin{align*}
z\nabla\left(v_{0},v_{1}\right) = \left(v_{0},v_{1}\right)\begin{pmatrix}
0 & z\\
1 & 0
\end{pmatrix}{\mathrm{d}}z.
\end{align*}
That is, $z\nabla v_{0} = v_{1}\mathrm{d}z$ and $\nabla v_{1} = v_{0}\mathrm{d}z$. The connection $\mathrm{Kl}_{2}$ is self-dual in the sense that there exists an algebraic horizontal pairing $\langle\ ,\, \rangle_{\text{alg}}$ on it:
\[
\left(\left\langle v_{i},v_{j}\right\rangle_{\mathrm{alg}}\right)_{0\leq i,j\leq 1}=
\begin{pmatrix}
0 & 1\\ 
-1 &0
\end{pmatrix}
\]
such that $\lambda :\mathrm{Kl}_{2}\rightarrow\mathrm{Kl}_{2}^{\vee}$ by $\left(v_{0},v_{1}\right)\mapsto \left(v_{1}^{\vee},-v_{0}^{\vee}\right)$ makes the following diagram commute.
\[
\begin{tikzcd}[ampersand replacement=\&]
{\mathrm{Kl}}_{2}\times {\mathrm{Kl}}_{2} \arrow[rr, "{\langle\ ,\, \rangle_{\text{alg}}}"] \arrow[d, "\lambda\times 1"'] \& \& (\mathcal{O}_{\mathbb{G}_{\mathrm{m}}},\mathrm{d}) \\
{\mathrm{Kl}}_{2}^{\vee}\times {\mathrm{Kl}}_{2} \arrow[rru, "{\text{natural}}"']                                                  \&       \&    
\end{tikzcd}
\]
Here ${\mathrm{Kl}}_{2}^{\vee}$ denotes the dual connection
with the dual basis $\{v_0^{\vee},v_1^{\vee}\}$.

Recall that the modified Bessel functions $I_{0}(t)$ and $K_{0}(t)$ satisfy the differential equation $((t\partial_{t})^{2}-t^{2})y=0$ and the Wronskian relation
\begin{equation}\label{eq:Wronskian_IK}
I_0(t)K_0'(t) - I_0'(t)K_0(t) = \frac{-1}{t}.
\end{equation} 
Under the change of variable $z=\frac{t^{2}}{4}$, the differential equation $((t\partial_{t})^{2}-t^{2})y=0$ becomes $4((z\partial_{z})^{2}-z)y=0$. Define $A_{0},B_{0}$ be the fundamental solutions to the differential equation $((z\partial_{z})^{2}-z)y=0$ by rescaling the modified Bessel functions. In addition, define $A_{1},B_{1}$ by $z\partial_{z}$ differential of $A_{0},B_{0}$:  
\begin{align*}
A_{0}(z)=-2I_{0}(2\sqrt{z}),&\qquad 
A_{1}(z)=z\partial_{z}A_{0}(z);\\
B_{0}(z)=2K_{0}(2\sqrt{z}), &\qquad
B_{1}(z)=z\partial_{z}B_{0}(z).
\end{align*}
Here, the function $\sqrt{z}$ is taken to be the principal branch on the range $|\arg z\,|<\pi$. For other $z$, these functions are defined via the analytic continuation. Throughout this paper, the multivalued functions such as $z^{k/2}$ or $z^{-1/4}$ are all treated in this way without a mention. The functions $A_{0}(z)$ and $B_{0}(z)$ are annihilated by the operator $\left(z\partial_{z}\right)^{2}-z$ and real-valued on the ray $\mathbb{R}_{>0}$. This gives 
\[ \partial_{z}A_{1}(z)=A_{0}(z),\qquad\partial_{z}B_{1}(z)=B_{0}(z). \]
Together with the Wronskian relation
$A_{0}B_{1} - A_{1}B_{0} = 2$ from \eqref{eq:Wronskian_IK}, we obtain a basis of horizontal sections of $\nabla$ on ${\mathrm{Kl}}_{2}$ 
\begin{align}\label{horizontal_sections_Kl2}
e_{0}=\frac{1}{2}(A_{0}v_{1}-A_{1}v_{0}),\qquad e_{1}=\frac{1}{2\pi\sqrt{-1}}(B_{0}v_{1}-B_{1}v_{0}).
\end{align}
Denote $\mathrm{Kl}_{2}^{\nabla}$ the local system of $\mathbb{Q}$-vector space generated by $e_{0},e_{1}$. There exists a topological pairing $\langle\ ,\, \rangle_{\mathrm{top}}=2\pi\sqrt{-1}\langle\ ,\, \rangle_{\mathrm{alg}}$ on $\mathrm{Kl}_{2}^{\nabla}$:
\[
\left(\left\langle e_{i},e_{j}\right\rangle_{\mathrm{top}}\right)_{0\leq i,j\leq 1}=
\begin{pmatrix}
0 & 1\\ 
-1 &0
\end{pmatrix}.
\]

\subsection{Rational structures and pairings on $(\mathcal{O}_{\mathbb{G}_{\mathrm{m}}},\mathrm{d}+\frac{1}{2}\frac{\mathrm{d}z}{z})$.} 
Consider the double cover $\left[2\right]:\mathbb{G}_{\mathrm{m},t}\rightarrow\mathbb{G}_{\mathrm{m},z}$ induced by the ring homomorphism $\mathbb{Q}[z,z^{-1}]\rightarrow\mathbb{Q}[t,t^{-1}]$, $z\mapsto t^{2}$. Let $T=(\mathcal{O}_{\mathbb{G}_{\mathrm{m},t}},{\mathrm{d}})$ be the trivial connection on $\mathbb{G}_{\mathrm{m},t}$.
Via the ring homomorphism $\mathbb{Q}[z,z^{-1}]\rightarrow\mathbb{Q}[t,t^{-1}]$, we view $\mathbb{Q}[t,t^{-1}]$ as the $\mathbb{Q}[z,z^{-1}]$-module. Then, from the decomposition of $\mathbb{Q}[z,z^{-1}]$-modules
\[
\mathbb{Q}[t,t^{-1}] = \mathbb{Q}[z,z^{-1}]\oplus t\cdot \mathbb{Q}[z,z^{-1}],
\]
the pushforward connection $\left[2\right]_{+}T$ decomposes into
the direct sum
\[
(\mathcal{O}_{\mathbb{G}_{\mathrm{m},z}},{\mathrm{d}})\oplus \left(t\cdot \mathcal{O}_{\mathbb{G}_{\mathrm{m},z}},{\mathrm{d}}\right).
\]
The second component $\left(t\cdot \mathcal{O}_{\mathbb{G}_{\mathrm{m},z}},{\mathrm{d}}\right)$ is isomorphic to $(\mathcal{O}_{\mathbb{G}_{\mathrm{m},z}},\mathrm{d}+\frac{1}{2}\frac{\mathrm{d}z}{z})$ via the following diagram
\begin{equation}\label{iso_of_connection_sqrt(z)}
\begin{tikzcd}
{\mathcal{O}_{\mathbb{G}_{\mathrm{m}},z}} \arrow[rr, "\mathrm{d}+\frac{\mathrm{d}t}{t}=\mathrm{d}+\frac{1}{2}\frac{\mathrm{d}z}{z}"] \arrow[d, "t"', "\wr"] && {\mathcal{O}_{\mathbb{G}_{\mathrm{m}},z}\otimes \Omega_{\mathbb{G}_{\mathrm{m}},z}^{1}} \arrow[d, "t", "\wr"'] \\
{t\cdot \mathcal{O}_{\mathbb{G}_{\mathrm{m}},z}} \arrow[rr, "\mathrm{d}"]                                       && {t\cdot \mathcal{O}_{\mathbb{G}_{\mathrm{m}},z}\otimes \Omega_{\mathbb{G}_{\mathrm{m}},z}^{1}}               
\end{tikzcd}
\end{equation}

The dual connection of $(\mathcal{O}_{\mathbb{G}_{\mathrm{m},z}},\mathrm{d}+\frac{1}{2}\frac{\mathrm{d}z}{z})$ is given by $(\mathcal{O}_{\mathbb{G}_{\mathrm{m},z}},\mathrm{d}-\frac{1}{2}\frac{\mathrm{d}z}{z})$,
and the two are isomorphic via multiplication by $z$, that is, the following diagram commutes.
\[
\begin{tikzcd}
{\mathcal{O}_{\mathbb{G}_{\mathrm{m}},z}} \arrow[r, " \mathrm{d}+\frac{1}{2}\frac{\mathrm{d}z}{z}"] \arrow[d, "z"', "\wr"] & {\mathcal{O}_{\mathbb{G}_{\mathrm{m}},z}\otimes \Omega_{\mathbb{G}_{\mathrm{m}},z}^{1}} \arrow[d, "z", "\wr"'] \\
{\mathcal{O}_{\mathbb{G}_{\mathrm{m}},z}} \arrow[r, "\mathrm{d}-\frac{1}{2}\frac{\mathrm{d}z}{z}"]                                       & {\mathcal{O}_{\mathbb{G}_{\mathrm{m}},z}\otimes \Omega_{\mathbb{G}_{\mathrm{m}},z}^{1}}               
\end{tikzcd}
\]
This induces a perfect algebraic horizontal pairing $\langle\ ,\, \rangle_{\mathrm{alg}}$ on $(\mathcal{O}_{\mathbb{G}_{\mathrm{m},z}},\mathrm{d}+\frac{1}{2}\frac{\mathrm{d}z}{z})$ given by
\[
\langle 1,1\rangle_{\mathrm{alg}} = z.
\]

On the other hand,
the rational structure of the local system of horizontal sections of $(\mathcal{O}_{\mathbb{G}_{\mathrm{m},t}},\mathrm{d})$ is generated by $1$. Under the isomorphism \eqref{iso_of_connection_sqrt(z)}, the rational structure of local system of horizontal sections of $(\mathcal{O}_{\mathbb{G}_{\mathrm{m},z}},\mathrm{d}+\frac{1}{2}\frac{\mathrm{d}z}{z})$ is generated by
$\frac{1}{t}=\frac{1}{\sqrt{z}}$. Its dual connection $(\mathcal{O}_{\mathbb{G}_{\mathrm{m},z}},\mathrm{d}-\frac{1}{2}\frac{\mathrm{d}z}{z})$ has local system of horizontal sections generated by $\sqrt{z}$. This induces a rational topological pairing $\langle\ ,\, \rangle_{\mathrm{top}}$ on $(\mathcal{O}_{\mathbb{G}_{\mathrm{m},z}},\mathrm{d}+\frac{1}{2}\frac{\mathrm{d}z}{z})^{\nabla}$
\[
\left\langle\frac{1}{\sqrt{z}},\frac{1}{\sqrt{z}}\right\rangle_{\mathrm{top}} = 1.
\]

\subsection{Algebraic and topological pairings on $\sqrt{z}\Sym^{k}{\mathrm{Kl}}_{2}$}\label{algebraic_topological_pairing}
The $k$-th symmetric product of $\mathrm{Kl}_{2}$, $\Sym^{k}\mathrm{Kl}_{2}$, is a rank $k+1$ free sheaf over $\mathcal{O}_{\mathbb{G}_{\mathrm{m}}}$ with basis of sections 
\[
v_{0}^{a}v_{1}^{k-a} = \frac{1}{|\mathfrak{S}_{k}|}\sum_{\sigma\in\mathfrak{S}_{k}}\sigma\left(v_{0}^{\otimes a}\otimes v_{1}^{\otimes k-a}\right)\qquad a=0,1,\ldots,k,
\]
where $\mathfrak{S}_{k}$ is the symmetric group on $k$ elements. It is endowed with the induced connection from $\left(\mathrm{Kl}_{2},\nabla\right)$. After twisting with the connection $\left(\mathcal{O}_{\mathbb{G}_{\mathrm{m}}},\mathrm{d}+\frac{1}{2}\frac{{\mathrm{d}z}}{z}\right)$, we define
\[
\sqrt{z}\Sym^{k}{\mathrm{Kl}}_{2}=\left(\mathcal{O}_{\mathbb{G}_{\mathrm{m}}},{\mathrm{d}}+\frac{1}{2}\frac{{\mathrm{d}}z}{z}\right)\otimes \Sym^{k}{\mathrm{Kl}}_{2}.
\]
The induced connection $\nabla$ on $\sqrt{z}\Sym^{k}{\mathrm{Kl}}_{2}$ is given by
\begin{align}\label{connection_on_twisted_symmetric_Kl2}
\nabla v_{0}^{a}v_{1}^{k-a}&=(k-a)v_{0}^{a+1}v_{1}^{k-a-1}{\mathrm{d}}z+\dfrac{a}{z}v_{0}^{a-1}v_{1}^{k-a+1}{\mathrm{d}}z+\dfrac{1}{2z}v_{0}^{a}v_{1}^{k-a}{\mathrm{d}}z.
\end{align}
Note that $\sqrt{z}\Sym^{k}{\mathrm{Kl}}_{2}$ is the same sheaf as $\Sym^{k}{\mathrm{Kl}}_{2}$ but endowed with a different connection.

Via the self-duality on ${\mathrm{Kl}}_{2}$ and on $\left(\mathcal{O}_{\mathbb{G}_{\mathrm{m}}},\mathrm{d}+\frac{1}{2}\frac{\mathrm{d}z}{z}\right)$, we have the perfect algebraic pairing $\left\langle\ ,\ \right\rangle_{\mathrm{alg}}$ on $\sqrt{z}\Sym^{k}{\mathrm{Kl}}_{2}$:
\[
\begin{tikzcd}
\sqrt{z}\Sym^{k}{\mathrm{Kl}}_{2}\times \sqrt{z}\Sym^{k}{\mathrm{Kl}}_{2} \arrow[rr, "{\left\langle\ ,\ \right\rangle_{\mathrm{alg}}}"] &  & {\left(\mathcal{O}_{\mathbb{G}_{\mathrm{m}}},{\mathrm{d}}\right)}
\end{tikzcd}
\]
given by
\[
\left\langle v_{0}^{k-a}v_{1}^{a},v_{0}^{k-b}v_{1}^{b}\right\rangle_{\mathrm{alg}}=z\delta_{k,a+b}(-1)^{a}\dfrac{a!b!}{k!}=(2\pi\sqrt{-1})^{k}\left\langle e_{0}^{k-a}e_{1}^{a},e_{0}^{k-b}e_{1}^{b}\right\rangle_{\mathrm{alg}}.
\]
Indeed, 
\begin{align*}
\left\langle v_{0}^{k-a}v_{1}^{a},v_{0}^{k-b}v_{1}^{b}\right\rangle_{\mathrm{alg}} & = \left\langle \frac{1}{|\mathfrak{S}_{k}|}\sum_{\sigma\in\mathfrak{S}_{k}}\sigma\left(v_{0}^{\otimes k-a}\otimes v_{1}^{\otimes a}\right),\frac{1}{|\mathfrak{S}_{k}|}\sum_{\tau\in\mathfrak{S}_{k}}\tau\left(v_{0}^{\otimes k-b}\otimes v_{1}^{\otimes b}\right)\right\rangle_{\mathrm{alg}}\\
&=\frac{1}{(k!)^{2}}\sum_{\sigma,\tau\in\mathfrak{S}_{k}}\langle \sigma(v_{0}^{\otimes k-a}\otimes v_{1}^{\otimes a}),\tau(v_{0}^{\otimes k-b}\otimes v_{1}^{\otimes b})\rangle_{\mathrm{alg},\mathrm{Kl}_{2}}\cdot \langle 1,1\rangle_{\mathrm{alg},(\mathcal{O}_{\mathbb{G}_{\mathrm{m}}},\mathrm{d}+\frac{1}{2}\frac{\mathrm{d}z}{z})}\\
& = \frac{1}{(k!)^{2}}(\delta_{k-a,b}k!a!b!(-1)^{a})\cdot z = z\delta_{k,a+b}(-1)^{a}\frac{a!b!}{k!}.
\end{align*}
By the definition of $e_{0},e_{1}$ in \eqref{horizontal_sections_Kl2} and Wronskian relation $A_{0}B_{1}-A_{1}B_{0}=2$, similar computation shows the formula for the algebraic pairing $\langle e_{0}^{k-a}e_{1}^{a},e_{0}^{k-b}e_{1}^{b}\rangle_{\mathrm{alg}}$.

The local system $(\sqrt{z}\Sym^{k}\mathrm{Kl}_{2})^{\nabla}$ is a $\mathbb{Q}$-vector space generated by the horizontal sections
\begin{equation}\label{eq:horizontal_basis}
\dfrac{1}{\sqrt{z}}e_{0}^{a}e_{1}^{k-a} = \dfrac{1}{\sqrt{z}}\sum_{\sigma\in\mathfrak{S}_{k}}\sigma\big(e_{0}^{\otimes a}\otimes e_{1}^{\otimes k-a}\big),\quad a=0,1,\cdots,k,
\end{equation}
which are the products of the horizontal sections of the connections $\left(\mathcal{O}_{\mathbb{G}_{\mathrm{m}}},\mathrm{d}+\frac{1}{2}\frac{{\mathrm{d}}z}{z}\right)$ and $\Sym^{k}\mathrm{Kl}_{2}$. The topological pairing $\langle\ ,\, \rangle$ on $\mathrm{Kl}_{2}^{\nabla}$ induces a topological pairing on $(\Sym^{k}\mathrm{Kl}_{2})^{\nabla}$ and thus on $(\sqrt{z}\Sym^{k}\mathrm{Kl}_{2})^{\nabla}$:
\[
\begin{tikzcd}
(\sqrt{z}\Sym^{k}{\mathrm{Kl}}_{2})^{\nabla}\times (\sqrt{z}\Sym^{k}{\mathrm{Kl}}_{2})^{\nabla} \arrow[rr, "{\left\langle\ ,\ \right\rangle_{\mathrm{top}}}"] &  & \mathbb{Q},
\end{tikzcd}
\]
where $\mathbb{Q}$ on the right hand side is the constant sheaf associated with the field $\mathbb{Q}$ on $\mathbb{G}_{\mathrm{m}}$.
This pairing reads
\[
\left<\frac{1}{\sqrt{z}}e_{0}^{a}e_{1}^{k-a},\frac{1}{\sqrt{z}}e_{0}^{b}e_{1}^{k-b}\right>_{\mathrm{top}}=\delta_{k,a+b}(-1)^{k-a}\dfrac{a!b!}{k!}
\]
by the similar computation as above.

\section{The de Rham cohomology}\label{section_de_Rham_side}
In this section, we study the de Rham cohomology of the twisted Kloosterman connection $H^{1}_{\mathrm{dR}}(\mathbb{G}_{\mathrm{m}},\sqrt{z}\Sym^{k}\mathrm{Kl}_{2})$ and its dual, compact support de Rham cohomology $H^{1}_{\mathrm{dR,c}}(\mathbb{G}_{\mathrm{m}},\sqrt{z}\Sym^{k}\mathrm{Kl}_{2})$. We will write down certain elements in these cohomologies explicitly and compute the Poincar\'{e} pairing between these elements. Finally, we conclude the bases of these two cohomologies in the end of this section.

\subsection{Dimension of $H^{1}_{\mathrm{dR}}\big(\mathbb{G}_{\mathrm{m}},\sqrt{z}\Sym^{k}{\mathrm{Kl}}_{2}\big)$}

\begin{proposition}\label{dimthm}
For the connection $\sqrt{z}\Sym^{k}{\mathrm{Kl}}_{2}$ on $\mathbb{G}_{\mathrm{m}}$, we have
\[
\dim H^{1}_{\mathrm{dR}}\Big(\mathbb{G}_{\mathrm{m}},\sqrt{z}\Sym^{k}{\mathrm{Kl}}_{2}\Big)=\left\lfloor\dfrac{k+1}{2}\right\rfloor.
\]
\end{proposition}

\begin{proof}
In \cite[Lemma 2.9.13]{MR1081536}, we have the following formula.
\begin{lemma}\label{dimension_lemma}
On $\mathbb{G}_{\mathrm{m}}$ with parameter $z$, let $\mathcal{D}=\mathcal{O}_{\mathbb{G}_{\mathrm{m}}}[\partial_{z}]$ be the ring of all differential operators on $\mathbb{G}_{\mathrm{m}}$. Write $\theta_{z}=z\partial_{z}$. For a non-zero element $L\in\mathcal{D}$, write $L$ into a finite sum of the form $\sum_{i} z^{i}P_{i}(\theta_{z})$, where $P_{i}(x)\in\mathbb{Q}[x]$. Define integers $a,b$ by
\[
a_{L}:=\max\left\{i\mid P_{i}\neq 0\right\};\ \ \ \ b_{L}:=\min\left\{i\mid P_{i}\neq 0\right\}.
\]
Then the Euler characteristic of the $\mathcal{D}$-module $\mathcal{D}/\mathcal{D}L$ is given by $\chi\left(\mathbb{G}_{\mathrm{m}},\mathcal{D}/\mathcal{D}L\right)=-\left(a_{L}-b_{L}\right)$.
\end{lemma}

In this proof, we will follow the notations as in this lemma. Now, the differential operator on $\mathbb{G}_{\mathrm{m}}$ associated with the connection ${\mathrm{Kl}}_{2}$ is given by $\theta_{z}^{2}-z$ which annihilates $v_{0}$ and has fundamental solutions $A_{0}(z)$ and $B_{0}(z)$. Then, the differential operator for $\Sym^{k}{\mathrm{Kl}}_{2}$ is given by the $k$-th symmetric power of $\theta_{z}^{2}-z$, i.e., the differential operator annihilates $v_{0}^{k}$ and has fundamental solutions $A_{0}^{i}B_{0}^{k-i}$ for $i=0,\ldots,k$. Denote this operator by $\widetilde{L}_{k+1}\in\mathcal{D}$. For $\sqrt{z}\Sym^{k}{\mathrm{Kl}}_{2}$, the corresponding differential operator reads $\frac{1}{\sqrt{z}}\widetilde{L}_{k+1}\sqrt{z}=:L$ since the solution is now given by $\frac{1}{\sqrt{z}}A_{0}^{i}B_{0}^{k-i}$ for $i=0,1,\ldots,k$.

Recall in subsection \ref{section_recall_dual_and_pairing_Kl2}, $L_{2}=(t\partial_{t})^{2}-t^{2}$ is the differential operator annihilates $I_{0}(t)$ and $K_{0}(t)$. Write $L_{k+1}$ to be the $k$-th symmetric power of $(t\partial_{t})^{2}-t^{2}$. That is, $L_{k+1}$ annihilates $I_{0}^{a}(t)K_{0}^{k-a}(t)$ for $a=0,\ldots,k$. As discussed in subsection \ref{section_recall_dual_and_pairing_Kl2}, the change of variable $z=\frac{t^{2}}{4}$ sends $L_{k+1}$ to $\widetilde{L}_{k+1}$. By Proposition \ref{dimennsion_parameter}, we have that $a_{L_{k+1}}=2\lfloor\frac{k+1}{2}\rfloor$, $b_{L_{k+1}}=0$. Therefore, by the degree $2$ change of variable $z=\frac{t^{2}}{4}$, we conclude $a_{\widetilde{L}_{k+1}}=\lfloor\frac{k+1}{2}\rfloor$, $b_{\widetilde{L}_{k+1}}=0$.

Using the fact that $\frac{1}{\sqrt{z}}\theta_{z}\sqrt{z}=\theta_{z}+\frac{1}{2}$, we have $\frac{1}{\sqrt{z}}\widetilde{L}_{k+1}\sqrt{z}=\sum z^{i}P_{i}\left(\theta_{z}+\frac{1}{2}\right)$ whenever $\widetilde{L}_{k+1}=\sum z^{i}P_{i}(\theta_{z})$. This shows $a_{L}=a_{\widetilde{L}_{k+1}}$ and $b_{L} = b_{\widetilde{L}_{k+1}}$. Therefore, by Lemma \ref{dimension_lemma}, we have 

\[
\chi\left(\mathbb{G}_{\mathrm{m}},\sqrt{z}\Sym^{k}{\mathrm{Kl}}_{2}\right)=\chi\left(\mathbb{G}_{\mathrm{m}},\mathcal{D}/\mathcal{D}L\right) =-\left\lfloor\frac{k+1}{2}\right\rfloor.
\]

Similar to the behavior of $I_{0}$ and $K_{0}$ \cite[\S 10.30(i)]{NIST:DLMF}, $A_{0}$ is holomorphic at $0$ and has exponential growth near infinity, and $B_{0}$ has a log pole at $0$. These imply all of the solutions $\frac{1}{\sqrt{z}}A_{0}^{i}B_{0}^{k-i}$ are not algebraic solutions, and thus $H^{0}_{\mathrm{dR}}\big(\mathbb{G}_{\mathrm{m}},\sqrt{z}\Sym^{k}{\mathrm{Kl}}_{2}\big)=0$. Hence, combining the fact that $H^{2}_{\mathrm{dR}}\big(\mathbb{G}_{\mathrm{m}},\Sym^{k}{\mathrm{Kl}}_{2}\big)=0$ by Artin vanishing theorem, we conclude that $H^{1}_{\mathrm{dR}}\big(\mathbb{G}_{\mathrm{m}},\sqrt{z}\Sym^{k}{\mathrm{Kl}}_{2}\big)$ has dimension $\left\lfloor\frac{k+1}{2}\right\rfloor$.
\end{proof}

\begin{remark}
In \cite[Lemma 2.9.13]{MR1081536}, Katz provides the proof of lemma \ref{dimension_lemma} only in the case $\mathbb{G}_{\mathrm{m},\mathbb{C}}$ which is over $\mathbb{C}$. Yet, the same proof still works in our situation $\mathbb{G}_{\mathrm{m},\mathbb{Q}}$ which is over $\mathbb{Q}$.
\end{remark}

\subsection{Compactly supported de Rham cohomology}
Write $k'=\lfloor\frac{k-1}{2}\rfloor$. Consider the $k^{\prime}+1$ elements $\left\{v_{0}^{k}z^{i}\frac{{\mathrm{d}z}}{z}\right\}_{i=0}^{k^{\prime}}$ in $H^{1}_{\mathrm{dR}}\big(\mathbb{G}_{\mathrm{m}},\sqrt{z}\Sym^{k}{\mathrm{Kl}}_{2}\big)$. We will prove these elements form a $\mathbb{Q}$-basis. (See Corollary \ref{de_rham_basis}.) The Poincar\'{e} dual of the de Rham cohomology is the de Rham cohomology with compact support.
An element in the de Rham cohomology with compact support $H^{1}_{\mathrm{dR,c}}\big(\mathbb{G}_{\mathrm{m}},\sqrt{z}\Sym^{k}{\mathrm{Kl}}_{2}\big)$ is represented by a triple $(\xi,\eta,\omega)$, where $\omega\in H^{1}_{\mathrm{dR}}\big(\mathbb{G}_{\mathrm{m}},\sqrt{z}\Sym^{k}{\mathrm{Kl}}_{2}\big)$ and $\xi,\eta$ are formal solutions to $\nabla\xi = \nabla\eta =\omega$ at $0$ and $\infty$ respectively (See \cite[Corollary 3.5]{MR4601771}). The solutions are provided by the following lemma. 

\begin{lemma}\label{lem_formal_sol}
Suppose that $k\equiv 0,1,3\bmod{4}$. For $0\leq i\leq k^{\prime}$, there exists $(\xi_{i},\eta_{i})\in\big(\sqrt{z}\Sym^{k}{\mathrm{Kl}}_{2}\big)_{\widehat{0}}\oplus \big(\sqrt{z}\Sym^{k}{\mathrm{Kl}}_{2}\big)_{\widehat{\infty}}$ such that $\nabla\xi_{i}=\nabla\eta_{i}=v_{0}^{k}z^{i}\frac{\mathrm{d}z}{z}$. 

On the other hand, let $k\equiv 2\bmod{4}$, say $k=4r+2$. For $0\leq i\leq k^{\prime}$ with $i\neq r$, there exists $(\xi_{i},\eta_{i})\in\big(\sqrt{z}\Sym^{k}{\mathrm{Kl}}_{2}\big)_{\widehat{0}}\oplus \big(\sqrt{z}\Sym^{k}{\mathrm{Kl}}_{2}\big)_{\widehat{\infty}}$ such that
\[
\nabla\xi_{i}=\nabla\eta_{i}=v_{0}^{k}z^{i}\frac{{\mathrm{d}}z}{z}-\gamma_{k,i-r}v_{0}^{k}z^{r}\frac{{\mathrm{d}}z}{z},
\]
where $\gamma_{k,n}\in\mathbb{Q}$ are the coefficients in the asymptotic expansion of $(-A_{0}\left(z\right)B_{0}\left(z\right))^{k/2}$
given by \eqref{expansion_of_(-AB)^half} below.
\end{lemma}

\begin{proof}
Near $0$, we want to find
\[ \xi_{i}=\sum_{a=0}^{k}\xi_{i,a}(z)v_{0}^{a}v_{1}^{k-a}\in\bigoplus_{a=0}^{k}\mathbb{Q}\llbracket z\rrbracket v_{0}^{a}v_{1}^{k-a} \]
such that $\nabla\xi_{i}=v_{0}^{k}z^{i}\frac{{\mathrm{d}}z}{z}$.
Using the connection formula \eqref{connection_on_twisted_symmetric_Kl2} on $\sqrt{z}\Sym^{k}{\mathrm{Kl}}_{2}$, we need to solve:
\begin{align*}
\frac{{\mathrm{d}}}{\mathrm{d}z}\xi_{i,k}(z)+(k-1)\xi_{i,k-1}(z)+\frac{1}{2z}\xi_{i,k}(z)&=z^{i-1},\\
\frac{{\mathrm{d}}}{\mathrm{d}z}\xi_{i,a}(z)+(k-a+1)\xi_{i,a-1}(z)+\frac{a+1}{z}\xi_{i,a+1}(z)+\frac{1}{2z}\xi_{i,a}(z)&=0 {\text{ for }}a=1,2,\cdots,k-1,\\
\frac{{\mathrm{d}}}{\mathrm{d}z}\xi_{i,0}(z)+\frac{1}{z}\xi_{i,1}(z)+\frac{1}{2z}\xi_{i,0}(z)&=0.
\end{align*}

Write $\xi_{i,a}=\sum_{n=0}^{\infty}\xi_{i,a,n}z^{n}$. We solve $\xi_{i,a,n}$ recursively on $n$. Suppose that we have solved $\xi_{i,a,j}$ for $j<n$. Compare the coefficient of $z^{n-1}$ of the above system of equations and get
\[
\left(
\begin{array}{ccccc}
&&& &n+\frac{1}{2}\\ 
&&&n+\frac{1}{2}&k\\
&&n+\frac{1}{2}&k-1&\\
&\iddots&\iddots&&\\
n+\frac{1}{2}&1&&&
\end{array}\right)
\left(\begin{array}{c}\xi_{i,0,n}\\ \xi_{i,1,n}\\\xi_{i,2,n}\\\vdots\\\xi_{i,k,n}\end{array}\right)=
{\text{lower order combinations}}.
\]
Since the first square matrix is invertible, $\xi_{i,a,n}$ is determined uniquely. Thus, we find $\xi_{i}\in\displaystyle\bigoplus_{a=0}^{k}\mathbb{Q}\llbracket z\rrbracket v_{0}^{a}v_{1}^{k-a}$ such that $\nabla\xi_{i}=v_{0}^{k}z^{i}\frac{{\mathrm{d}}z}{z}$. In $k\equiv 2\bmod{4}$ case, we only need to replace $\xi_{i}$ by $\xi_{i}-\gamma_{k,i-r}\xi_{r}$.

Next, we turn to investigate the formal solutions at $\infty$ using horizontal frames. We have the modified Bessel functions have the asymptotic expansions at $\frac{1}{t}$ \cite[\S 7.23]{MR1349110}
\begin{align}
I_0(t)
&\sim e^t\frac{1}{\sqrt{2\pi t}}
	\sum_{n=0}^\infty \frac{((2n-1)!!)^2}{2^{3n} n!} \frac{1}{t^n},
&& |\arg t\,| < \frac{1}{2}\pi \label{expansion_I_0_at_inf}\\
K_0(t)
&\sim e^{-t} \sqrt{\frac{\pi}{2t}}
	\sum_{n=0}^\infty (-1)^n\frac{((2n-1)!!)^2}{2^{3n} n!} \frac{1}{t^n},
&& |\arg t\,| < \frac{3}{2}\pi \label{expansion_K_0_at_inf}\\
I_0(t)K_0(t)
&\sim \frac{1}{2t}
	\sum_{n=0}^\infty \frac{((2n-1)!!)^3}{2^{3n}n!}\frac{1}{t^{2n}}.\label{expansion_I_0K_0_at_inf}
\end{align}
Here, the notation $n!!$ is the double factorial of a positive integer $n$ defined by 
\[
n!!=\prod_{k=0}^{\lceil\frac{n}{2}\rceil -1}(n-2k).
\]
Let $w=\frac{1}{z}$ be the local coordinate at $z=\infty$. For $k$ even, by the last asymptotic expansion, we have
\begin{align}\label{expansion_of_(-AB)^half}
(-A_{0}(z)B_{0}(z))^{k/2}\sim w^{1/4}\sum_{n=0}^{\infty}\gamma_{k,n}w^{n},
\end{align}
where $\gamma_{k,0}=1$ and $\gamma_{k,n}>0$ for all $n>0$. For convenience, we set $\gamma_{k,j}=0$ for all $j<0$.

Following the notation in \eqref{horizontal_sections_Kl2}, let us set $\overline{e}_{1} = \pi\sqrt{-1}e_{1}$.
Then $\frac{1}{\sqrt{z}}e_{0}^{a}\overline{e_{1}}^{k-a}$ are horizontal sections. Using the Wronskian relation $A_{0}B_{1}-A_{1}B_{0}=2$, we have $v_{0}=B_{0}e_{0}-A_{0}\overline{e}_{1}$. Then, we obtain
\begin{align*}
v_{0}^{k}z^{i}\frac{\mathrm{d}z}{z}&=-\sum_{a=0}^{k}\binom{k}{a}\frac{(-A_{0})^{a}B_{0}^{k-a}}{w^{i+3/2}}\frac{1}{\sqrt{z}}e_{0}^{k-a}\overline{e}_{1}^{a}\mathrm{d}w.
\end{align*}

To solve the formal solution $\eta_{i}$ of $\nabla\eta_{i} = v_{0}^{k}z^{i}\frac{\mathrm{d}z}{z}$, We first solve $\eta_{i,a}$ for each $i,a$ such that
\begin{align*}
\mathrm{d}\eta_{i,a}=\frac{(-A_{0})^{a}B_{0}^{k-a}}{w^{i+3/2}}.
\end{align*}
Then, $\eta_{i} = -\sum_{a=0}^{k}\binom{k}{a}\eta_{i,a}\frac{1}{\sqrt{z}}e_{0}^{k-a}\overline{e}_{1}^{a}$ is the desired solution. Moreover, since the function $\eta_{i,a}$ have $w^{k/4}$ and the exponential factor (see \eqref{expansion_eta_i_a} below), we need to justify
that $\eta_{i}$ lies in $\bigoplus_{a=0}^{k}\mathbb{Q}\llbracket w\rrbracket v_{0}^{a}v_{1}^{k-a}$ (not just in $\bigoplus_{a=0}^{k}\mathbb{Q}\llbracket w^{1/4},e^{-1/\sqrt{w}}\rrbracket v_{0}^{a}v_{1}^{k-a}$).

Near $\infty$, we have the expansion
\[
\frac{(-A_{0})^{a}B_{0}^{k-a}}{w^{i+3/2}}=
\begin{cases}
\sqrt{\pi}^{k-2a}e^{-2(k-2a)/\sqrt{w}}w^{k/4-i-3/2}\cdot F_{i,a}, & a\neq\frac{k}{2}\\
w^{k/4-i-3/2}\left(\sum_{n=0}^{\infty}\frac{\left(\left(2n-1\right)!!\right)^{3}}{2^{5n}n!}w^{n}\right)^{k/2}, & a=\frac{k}{2}
\end{cases}
\]
where $F_{i,a}\in 1+\sqrt{w}\mathbb{Q}\llbracket\sqrt{w}\rrbracket$.
When $a\neq \frac{k}{2}$, we can find an antiderivative $\eta_{i,a}$ of $\frac{(-A_{0})^{a}B_{0}^{k-a}}{w^{i+3/2}}$ with the expansion
\begin{align}\label{expansion_eta_i_a}
\eta_{i,a}=\frac{\sqrt{\pi}^{k-2a}}{k-2a}e^{-2(k-2a)/\sqrt{w}}w^{k/4-i}\cdot G_{i,a}
\end{align}
for some $G_{i,a}\in 1+\sqrt{w}\mathbb{Q}\llbracket\sqrt{w}\rrbracket$. We analyze $\eta_{i,a}\frac{1}{\sqrt{z}}e_{0}^{k-a}\overline{e}_{1}^{a}$.
Write $e_{0}^{k-a}\overline{e}_{1}^{a}$ back to the expression in basis $v_{0}^{b}v_{1}^{k-b}$:
\begin{align*}
e_{0}^{k-a}\overline{e}_{1}^{a}&=2^{-k}(A_{0}v_{1}-A_{1}v_{0})^{k-a}\cdot(B_{0}v_{1}-B_{1}v_{0})^{a}\\
&=2^{-k}e^{2(k-a)\sqrt{z}}\frac{1}{\sqrt{\pi}^{k-a}}(F_{1}v_{0}-F_{2}v_{1})^{k-a}\cdot e^{-2a\sqrt{z}}\sqrt{\pi}^{a}(G_{1}v_{1}-G_{2}v_{0})^{a}\\
&=2^{-k}e^{2(k-2a)\sqrt{z}}\sqrt{\pi}^{2a-k}(F_{1}v_{0}-F_{2}v_{1})^{k-a}(G_{1}v_{1}-G_{2}v_{0})^{a},
\end{align*}
where $F_{1},F_{2},G_{1},G_{2}\in z^{1/4}\mathbb{Q}\llbracket z^{-1/4}\rrbracket$. Thus,

\begin{align*}
\eta_{i,a}\frac{1}{\sqrt{z}}e_{0}^{k-a}\overline{e}_{1}^{a}&= \frac{2^{-k}}{k-2a}w^{k/4-i+1/2}G_{i,a}(F_{1}v_{0}-F_{2}v_{1})^{k-a}(G_{1}v_{1}-G_{2}v_{0})^{a},
\end{align*}
where $F_{1},F_{2},G_{1},G_{2}\in z^{1/4}\mathbb{Q}\llbracket z^{-1/4}\rrbracket$.
We conclude that the desired $\eta_{i}=-\sum_{a=0}^{k}\binom{k}{a}\eta_{i,a}\frac{1}{\sqrt{z}}e_{0}^{k-a}\overline{e}_{1}^{a}$ has no exponential factor as a combination of monomials $v_{0}^{k-b}v_{1}^{b}$, that is, $\eta_{i}$ lies in $\bigoplus_{a=0}^{k}\mathbb{Q}\llbracket w^{1/4}\rrbracket v_{0}^{a}v_{1}^{k-a}$.

Next, we will prove $\eta_{i}$ lies in $\bigoplus_{a=0}^{k}\mathbb{Q}\llbracket w\rrbracket v_{0}^{a}v_{1}^{k-a}$ by showing $\eta_{i}$ is invariant under the Galois group action. Let $\sigma :w^{1/4}\mapsto \sqrt{-1}w^{1/4}$ be the generator of the Galois group of the extension $\mathbb{C}(w^{1/4})$ of $\mathbb{C}(w)$. From the monodromy action \cite[10.34.5]{NIST:DLMF} of $I_{0},K_{0}$, the $\sigma$ action on $A_{i},B_{i}$ is given by
\begin{align*}
\sigma\left(A_{j},B_{j}\right)&=\left(\frac{1}{\pi\sqrt{-1}}B_{j},-\pi\sqrt{-1}A_{j}\right) {\text{ for }}j=0,1,
\end{align*}
and thus on $e_{0},\overline{e}_{1}$ by
\begin{align*}
\sigma\left(e_{0},\overline{e}_{1}\right)=\left(\frac{1}{\pi\sqrt{-1}} \overline{e}_{1},-\pi\sqrt{-1}e_{0}\right);&\quad
\sigma\left(e_{0}^{k-a}\overline{e}_{1}^{a}\right)=(\sqrt{-1})^{-k}\pi^{2a-k}e_{0}^{a}\overline{e}_{1}^{k-a}.
\end{align*}
Moreover, we have
\[
\sigma\left(\eta_{i,a}\frac{1}{\sqrt{z}}e_{0}^{k-a}\overline{e}_{1}^{a}\right)=\eta_{i,k-a}\frac{1}{\sqrt{z}}e_{0}^{a}\overline{e}_{1}^{k-a}.
\]
Hence, when $k\equiv 1,3\bmod{4}$, the element $\eta_{i}$ is fixed by $\sigma$ and
\[
\displaystyle\eta_{i}=-\sum_{a=0}^{k}\binom{k}{a}\eta_{i,a}\frac{1}{\sqrt{z}}e_{0}^{k-a}\overline{e}_{1}^{a} \in\bigoplus_{a=0}^{k}\mathbb{Q}\llbracket w\rrbracket v_{0}^{a}v_{1}^{k-a}.
\]
This gives $\nabla\eta_{i}=v_{0}^{k}z^{i}\frac{{\mathrm{d}}z}{z}.$

When $k=4r+4$ and $a=2r+2$, the exponents of $w$ of the expansion of $\frac{(-A_{0})^{a}B_{0}^{k-a}}{w^{i+3/2}}$ are in $\frac{1}{2}+\mathbb{Z}$ and one takes \[\eta_{i,2r+2}\sim\frac{w^{r-i+1/2}}{r-i+1/2}G_{i}\]
where $G_{i}\in 1+w\mathbb{Q}\llbracket w\rrbracket$. More precisely, we have
\[
G_{i}= 1+\sum_{n=1}^{\infty}\frac{r-i+1/2}{r-i+1/2+n}\gamma_{k,n}w^{n}.
\]
Moreover,
$\eta_{i,2r+2}\frac{1}{\sqrt{z}}(e_{0}\overline{e}_{1})^{2r+2}$ has no exponential factor as a combination of monomials $v_{0}^{k-b}v_{1}^{b}$ and is invariant under $\sigma$. Hence, when $k\equiv 0\bmod{4}$, we take an element 
\[
\eta_{i}=-\sum_{a=0}^{k}\binom{k}{a}\eta_{i,a}\frac{1}{\sqrt{z}}e_{0}^{k-a}\overline{e}_{1}^{a}\in\displaystyle\bigoplus_{a=0}^{k}\mathbb{Q}\llbracket z\rrbracket v_{0}^{a}v_{1}^{k-a}
\]
This gives $\nabla\eta_{i}=v_{0}^{k}z^{i}\frac{{\mathrm{d}}z}{z}$.

Now, suppose that $k=4r+2$, a positive integer congruent to $2$ modulo $4$, and $a=2r+1$.
Using the expansion \eqref{expansion_of_(-AB)^half}, we have the residue:
\[
\Res_{w}\frac{(-A_{0})^{a}B_{0}^{k-a}}{w^{i+3/2}}=\gamma_{k,i-r},
\]
which vanishes if and only if $i\leq r-1$. Therefore, for $i\geq r$, there exists 
\[
\eta_{i,2r+1}\sim \frac{1}{r-i} w^{r-i}\cdot H_{i}
\]
such that
\[
\mathrm{d}\eta_{i,2r+1}=\left(w^{-i-3/2}-\gamma_{k,i-r}w^{-r-3/2}\right)(-A_{0}B_{0})^{2r+1}\mathrm{d}w
\]
where $H_{i}\in 1+w\mathbb{Q}\llbracket w\rrbracket$.
Also,
$\eta_{i,2r+1}\frac{1}{\sqrt{z}}(e_{0}\overline{e}_{1})^{2r+1}$
is invariant under $\sigma$.
Moreover,
$\eta_{i,2r+1}\frac{1}{\sqrt{z}}(e_{0}\overline{e}_{1})^{2r+1}$
has no exponential factor as a combination of monomials $v_{0}^{k-b}v_{1}^{b}$. Thus, we have
\begin{align*}
v_{0}^{k}z^{i}\frac{{\mathrm{d}}z}{z}-\gamma_{k,i-r}v_{0}^{k}z^{r}\frac{\mathrm{d}z}{z}&=\nabla\left(-\sum_{\substack{a=0\\a\neq k/2}}^{k}\binom{k}{a}\frac{\eta_{i,a}-\gamma_{k,i-r}\eta_{r,a}}{\sqrt{z}}e_{0}^{k-a}\overline{e}_{1}^{a}
-\binom{k}{k/2}\frac{\eta_{i,2r+1}}{\sqrt{z}}e_{0}^{2r+1}\overline{e}_{1}^{2r+1}\right)
\end{align*}
and hence we find an element $\eta_{i}$ in $\displaystyle\bigoplus_{a=0}^{k}\mathbb{Q}\llbracket z\rrbracket v_{0}^{a}v_{1}^{k-a}$ such that $\nabla\eta_{i}=v_{0}^{k}z^{i}\frac{{\mathrm{d}}z}{z}-\gamma_{k,i-r}v_{0}^{k}z^{r}\frac{{\mathrm{d}}z}{z}$.
\end{proof}

Now, we define some elements in the de Rham cohomology and the de Rham cohomology with compact support. In next subsection, we will prove that these elements form bases of the corresponding cohomology spaces (see Corollary \ref{de_rham_basis}).

\begin{definition}\label{de_Rham_basis_elements}
In the de Rham cohomology $H^{1}_{\mathrm{dR}}\big(\mathbb{G}_{\mathrm{m}},\sqrt{z}\Sym^{k}{\mathrm{Kl}}_{2}\big)$,
the classes $\omega_{k,i}$ are given as follows.
\begin{enumerate}
\item When $k\equiv 0,1,3\bmod{4}$, define the $k^{\prime}+1$ elements:
\[
\omega_{k,i}=v_{0}^{k}z^{i}\frac{\mathrm{d}z}{z} {\text{ for }}i=0,1,2,\cdots,k^{\prime}.
\]
\item When $k\equiv 2\bmod{4}$, write $k=4r+2$ and define the $k^{\prime}$ elements:
\[
\omega_{k,i}=
\begin{cases}
v_{0}^{k}z^{i}\frac{{\mathrm{d}}z}{z}, & 0\leq i\leq r-1;\\
v_{0}^{k}z^{i}\frac{{\mathrm{d}}z}{z}-\gamma_{k,i-r}v_{0}^{k}z^{r}\frac{{\mathrm{d}}z}{z}, & r+1\leq i\leq 2r,
\end{cases}
\]
where $\gamma_{k,n}\in\mathbb{Q}$ are the coefficients in the asymptotic expansion of $(-A_{0}\left(z\right)B_{0}\left(z\right))^{k/2}$
given by \eqref{expansion_of_(-AB)^half} above.
\end{enumerate}
\end{definition}

From the Lemma \ref{lem_formal_sol}, we define the elements in the compactly supported de Rham cohomology.

\begin{definition}\label{compact_support_de_rham_bisis_elements}
We define certain elements
in the compactly supported de Rham cohomology $H^{1}_{\mathrm{dR,c}}
\big(\mathbb{G}_{\mathrm{m}},\sqrt{z}\Sym^{k}{\mathrm{Kl}}_{2}\big)$
as follows.
\begin{enumerate}
\item When $k\equiv 0,1,3\bmod{4}$, define $k^{\prime}+1$ elements
\[
\widetilde{\omega}_{k,i}=(\xi_{i},\eta_{i},\omega_{k,i}) {\text{ for }}0\leq i\leq k^{\prime},
\]
where $\nabla \xi_{i}=\nabla\eta_{i}=\omega_{k,i}$.
\item When $k\equiv 2\bmod{4}$, write $k=4r+2$ and define $k^{\prime}$ elements
\[
\widetilde{\omega}_{k,i}=(\xi_{i},\eta_{i},\omega_{k,i}) {\text{ for }} 0\leq i\leq r-1 {\text{ and }} r+1\leq i\leq k^{\prime}, 
\]
where $\nabla \xi_{i}=\nabla\eta_{i}=\omega_{k,i}$.
\item In the case that $k\equiv 2\bmod{4}$, write $k=4r+2$ and further define
\[
\widehat{m}_{2r+1}=\left(0,2^{k}\frac{1}{\sqrt{z}}(e_{0}\overline{e}_{1})^{2r+1},0\right)\in H^{1}_{\mathrm{dR,c}}\Big(\mathbb{G}_{\mathrm{m}},\sqrt{z}\Sym^{k}{\mathrm{Kl}}_{2}\Big).
\]
Here, $\overline{e}_{1} := \pi\sqrt{-1}e_{1}$ and $e_{0},e_{1}$ are horizontal sections of $\mathrm{Kl}_{2}$ defined in \eqref{horizontal_sections_Kl2}.
\end{enumerate}
\end{definition}

\begin{remark}
The pair of the formal solutions $(\xi_{i},\eta_{i})$ is unique
except in the case that there are solutions $(\xi,\eta)$
to $\nabla\xi = \nabla\eta = 0$.
The latter happens only when $k\equiv 2\bmod{4}$.
In this circumstance,
we fix the choice of $(\xi_{i},\eta_{i})$
to be the one constructed in the proof of Lemma \ref{lem_formal_sol}.
These expressions will be used in the computations of Poincar\'{e} pairing and period pairing in the rest of this paper.
\end{remark}

Further, we define the middle part de Rham cohomology, $H^{1}_{\mathrm{mid}}\big(\mathbb{G}_{\mathrm{m}},\sqrt{z}\Sym^{k}{\mathrm{Kl}}_{2}\big)$, to be the image of the projection $H^{1}_{\mathrm{dR,c}}\big(\mathbb{G}_{\mathrm{m}},\sqrt{z}\Sym^{k}{\mathrm{Kl}}_{2}\big)\rightarrow H^{1}_{\mathrm{dR}}\big(\mathbb{G}_{\mathrm{m}},\sqrt{z}\Sym^{k}{\mathrm{Kl}}_{2}\big)$, $(\xi,\eta,\omega)\mapsto\omega$. We therefore have
\begin{align*}
\omega_{k,i}\in H^{1}_{\mathrm{mid}}\Big(\mathbb{G}_{\mathrm{m}},\sqrt{z}\Sym^{k}{\mathrm{Kl}}_{2}\Big) & {\text{ for }} 0\leq i\leq k^{\prime} {\text{ when }} k\equiv 0,1,3\bmod{4};\\
\omega_{k,i}\in H^{1}_{\mathrm{mid}}\Big(\mathbb{G}_{\mathrm{m}},\sqrt{z}\Sym^{k}{\mathrm{Kl}}_{2}\Big) & {\text{ for }} 0\leq i\leq k^{\prime},\, i\neq r {\text{ when }} k=4r+2.
\end{align*}
We may regard $H^{1}_{\mathrm{mid}}\big(\mathbb{G}_{\mathrm{m}},\sqrt{z}\Sym^{k}{\mathrm{Kl}}_{2}\big)$ as a quotient of $H^{1}_{\mathrm{dR,c}}\big(\mathbb{G}_{\mathrm{m}},\sqrt{z}\Sym^{k}{\mathrm{Kl}}_{2}\big)$ containing the class of elements $\widetilde{\omega}_{k,i}$.

\subsection{Poincar\'{e} pairing} We have the following Poincar\'{e} pairing between the de Rham cohomology and the compactly supported de Rham cohomology. Recall the algebraic pairing $\left\langle\ ,\ \right\rangle_{\mathrm{alg}}$ is introduced in section \ref{algebraic_topological_pairing}.
\[
\begin{tikzcd}[row sep=tiny]
{H^{1}_{\text{dR},c}\Big(\mathbb{G}_{\mathrm{m}},\sqrt{z}\Sym^{k}{\text{Kl}}_{2}\Big)\otimes H^{1}_{\text{dR}}\Big(\mathbb{G}_{\mathrm{m}},\sqrt{z}\Sym^{k}{\text{Kl}}_{2}\Big)} \arrow[r, "{\left\langle\ ,\ \right\rangle_{\text{Poin}}}"] & \mathbb{Q}(-k-1)                                                                                                                    \\
{(\widehat{m}_{0},\widehat{m}_{\infty},\omega)\otimes\eta} \arrow[r, mapsto]                                                                                                                                                                                                                 & {\Res_{z}\left\langle\widehat{m}_{0},\eta\right\rangle_{{\mathrm{alg}}}+\Res_{w}\left\langle\widehat{m}_{\infty},\eta\right\rangle_{{\mathrm{alg}}}}
\end{tikzcd}
\]
Here, a one-form $\eta$ occurs in $\langle \widehat{m},\eta\rangle_{\mathrm{alg}}$. This algebraic pairing means $\langle \widehat{m},f\rangle_{\mathrm{alg}} \mathrm{d}z$ whenever $\eta = f\mathrm{d}z$. The notation $\mathbb{Q}(-k-1)$ is the $(k+1)$-time tensor product of the Tate structures $\mathbb{Q}(-1)$. As a vector space, $\mathbb{Q}(-k-1)$ is nothing but $\mathbb{Q}$. Here, in consideration of Hodge filtrations, we use $\mathbb{Q}(-k-1)$ instead of $\mathbb{Q}$ to indicate the Hodge filtrations on both sides respect the Poincar\'{e} pairing. Note that the Poincar\'{e} pairing induces on the middle part de Rham cohomology which we still call it $\langle\ ,\ \rangle_{\mathrm{Poin}}$:
\[
\begin{tikzcd}[row sep=tiny]
{H^{1}_{\mathrm{mid}}\Big(\mathbb{G}_{\mathrm{m}},\sqrt{z}\Sym^{k}{\text{Kl}}_{2}\Big)\otimes H^{1}_{\mathrm{mid}}\Big(\mathbb{G}_{\mathrm{m}},\sqrt{z}\Sym^{k}{\text{Kl}}_{2}\Big)} \arrow[rr, "{\left\langle\ ,\ \right\rangle_{\text{Poin}}}"] &  & \mathbb{Q}(-k-1).
\end{tikzcd}
\]

\begin{proposition}\label{de_rham_paiting_matrix}
Under the notation as in Definition \ref{compact_support_de_rham_bisis_elements}, for $j\geq 0$, we have the Poincar\'{e} pairing
\begin{align*}
\left\langle \widetilde{\omega}_{k,i}, v_{0}^{k}z^{j}\frac{{\mathrm{d}}z}{z} \right\rangle_{\mathrm{Poin}}&=
\begin{cases}
0 & {\text{if }}i+j\leq k^{\prime}-1,\ k:{\text{arbitrary,}}\\
(-2)^{k^{\prime}}\frac{k^{\prime}!}{k!!}& {\text{if }}i+j= k^{\prime},\ k\equiv 1,3\bmod{4},\\
\frac{\binom{k}{k/2}}{2^{k}(r-i+1/2)}& {\text{if }}i+j= k^{\prime},\ k\equiv 0\bmod{4},\ k=4r+4,\\
-\frac{\binom{k}{k/2}}{2^{k}(r-i)}& {\text{if }}i+j= k^{\prime},\ k\equiv 2\bmod{4},\ k=4r+2.
\end{cases}
\end{align*}
Moreover, if $k=4r+2$, we have
\[
\left\langle\widehat{m}_{2r+1},v_{0}^{k}z^{j}\frac{{\mathrm{d}}z}{z}\right\rangle_{\mathrm{Poin}}=\begin{cases}
0 & {\text{if }}j<r,\\
\gamma_{k,j-r} & {\text{if }}j\geq r.
\end{cases}
\]
In particular, the Poincar\'{e} pairing matrix between $k'+1$ elements in Definition \ref{compact_support_de_rham_bisis_elements} and $k'+1$ elements $\left\{v_{0}^{k}z^{i}\frac{{\mathrm{d}z}}{z}\right\}_{i=0}^{k^{\prime}}$ in $H^{1}_{\mathrm{dR}}\big(\mathbb{G}_{\mathrm{m}},\sqrt{z}\Sym^{k}{\mathrm{Kl}}_{2}\big)$ is non-degenerate.
\end{proposition}

\begin{proof}
In this proof, we will follow the notations as in the proof of Lemma \ref{lem_formal_sol}. We first discuss the residue at $z=0$. For any $k$ and $0\leq i,j\leq k^{\prime}$, we compute
\begin{align*}
\Res_{z}\left\langle \xi_{i},v_{0}^{k}z^{j}\frac{{\mathrm{d}}z}{z}\right\rangle_{\mathrm{alg}}&=\sum_{a=0}^{k}\Res_{z}\left\langle \xi_{i,a}v_{0}^{a}v_{1}^{k-a},v_{0}^{k}z^{j}\frac{{\mathrm{d}}z}{z}\right\rangle_{\mathrm{alg}}\\
&=\Res_{z}\left\langle \xi_{i,0}v_{1}^{k},v_{0}^{k}z^{j}\frac{{\mathrm{d}}z}{z}\right\rangle_{\mathrm{alg}}\\
&=\Res_{z}\left((-1)^{k}\xi_{i,0}z^{j}\right)=0
\end{align*}
where $\xi_{i,0}\in\mathbb{Q}\llbracket z\rrbracket$.

Next, we discuss the residue at $z=\infty$. When $k\equiv 1,3\bmod{4}$ and for any $0\leq i,j\leq k^{\prime}$, we compute
\begin{align*}
\Res_{w}\left\langle\eta_{i},v_{0}^{k}z^{j}\frac{{\mathrm{d}}z}{z}\right\rangle_{\mathrm{alg}} &=-\sum_{a=0}^{k}\binom{k}{a}\Res_{w}\left\langle\eta_{i,a}\frac{e_{0}^{k-a}\overline{e}_{1}^{a}}{\sqrt{z}},v_{0}^{k}z^{j}\frac{{\mathrm{d}}z}{z}\right\rangle_{\mathrm{alg}}\\
&=\sum_{a,b=0}^{k}\binom{k}{a}\binom{k}{b}\Res_{w}\left\langle\eta_{i,a}\frac{e_{0}^{k-a}\overline{e}_{1}^{a}}{\sqrt{z}},\frac{(-A_{0})^{b}B_{0}^{k-b}}{w^{j+3/2}}\frac{e_{0}^{k-b}\overline{e}_{1}^{b}}{\sqrt{z}}\mathrm{d}w\right\rangle_{\mathrm{alg}}\\
&=\frac{1}{2^{k}}\sum_{a=0}^{k}(-1)^{a}\binom{k}{a}\Res_{w}\left(\eta_{i,a}\frac{1}{w^{j+3/2}}(-A_{0})^{k-a}B_{0}^{a}{\mathrm{d}}w\right)\\
&=\frac{1}{2^{k}}\sum_{a=0}^{k}\frac{(-1)^{a}\binom{k}{a}}{k-2a}\Res_{w}\left(w^{(k-1)/2-i-j-1}G_{i,a}F_{j,k-a}{\mathrm{d}}w\right)\\
&=
\begin{cases}
0 & {\text{ if }}i+j\leq k^{\prime}-1,\\
(-2)^{k^{\prime}}\frac{k^{\prime}!}{k!!} & {\text{ if }} i+j=k^{\prime}.
\end{cases}
\end{align*}
where $G_{i,a},F_{j,k-a}\in 1+\sqrt{w}\mathbb{Q}\llbracket\sqrt{w}\rrbracket$ and the last equality follows from \cite[lemma 3.18]{MR4578001}.

When $k\equiv 0\bmod{4}$, write $k=4r+4$. For any $0\leq i,j\leq k^{\prime}$, we compute
\begin{align*}
\Res_{w}\left\langle\eta_{i},v_{0}^{k}z^{j}\frac{{\mathrm{d}}z}{z}\right\rangle_{\mathrm{alg}}&=\frac{1}{2^{k}}\sum_{\substack{a=0\\a\neq\frac{k}{2}}}^{k}\frac{(-1)^{a}\binom{k}{a}}{k-2a}\Res_{w}\left(w^{(k-1)/2-i-j-1}G_{i,a}F_{j,k-a}{\mathrm{d}}w\right)\\
&\qquad+\frac{(-1)^{k/2}}{2^{k}}\binom{k}{k/2}\Res_{w}\left(\eta_{i,2r+2}\frac{1}{w^{j+3/2}}(-A_{0}B_{0})^{2r+2}{\mathrm{d}}w\right)\\
&=\frac{1}{2^{k}}\sum_{\substack{a=0\\a\neq\frac{k}{2}}}^{k}\frac{(-1)^{a}\binom{k}{a}}{k-2a}\Res_{w}\left(w^{(k-1)/2-i-j-1}G_{i,a}F_{j,k-a}{\mathrm{d}}w\right)\\
&\qquad+\frac{\binom{k}{k/2}}{2^{k}(r-i+1/2)}\Res_{w}\left(w^{k^{\prime}-i-j-1}\cdot G_{i}F_{2r+2}{\mathrm{d}}w\right)\\
&=
\begin{cases}
0 & {\text{ if }}i+j\leq k^{\prime}-1,\\
\frac{\binom{k}{k/2}}{2^{k}(r-i+1/2)} & {\text{ if }}i+j=k^{\prime}.
\end{cases}
\end{align*}
where $G_{i,a},F_{j,k-a}\in 1+\sqrt{w}\mathbb{Q}\llbracket\sqrt{w}\rrbracket$, $G_{i}\in 1+w\mathbb{Q}\llbracket w\rrbracket$ and $F_{2r+2}=\left(\sum_{n=0}^{\infty}\frac{\left(\left(2n-1\right)!!\right)^{3}}{2^{5n}n!}w^{n}\right)^{2r+2}$.

When $k\equiv 2\bmod{4}$, the computation is similar to the case $k\equiv 0\bmod{4}$.

Finally, we compute
\begin{align*}
\Res_{w}\left\langle\frac{2^{k}(e_{0}\overline{e}_{1})^{2r+1}}{\sqrt{z}},v_{0}^{k}z^{j}\frac{{\mathrm{d}}z}{z}\right\rangle_{\mathrm{alg}}&=-\sum_{b=0}^{k}\binom{k}{b}\Res_{w}\left\langle\frac{2^{k}(e_{0}\overline{e}_{1})^{2r+1}}{\sqrt{z}},\frac{(-A_{0})^{b}B_{0}^{k-b}}{w^{j+3/2}}\frac{e_{0}^{k-b}\overline{e}_{1}^{b}}{\sqrt{z}}{\mathrm{d}}w\right\rangle_{\mathrm{alg}}\\
&=-\binom{k}{k/2}\Res_{w}\left\langle\frac{2^{k}(e_{0}\overline{e}_{1})^{2r+1}}{\sqrt{z}},\frac{(-A_{0}B_{0})^{2r+1}}{w^{j+3/2}}\frac{(e_{0}\overline{e}_{1})^{2r+1}}{\sqrt{z}}{\mathrm{d}}w\right\rangle_{\mathrm{alg}}\\
&=(-1)^{2r+2}\Res_{w}\left(\frac{(-A_{0}B_{0})^{2r+1}}{w^{j+3/2}}{\mathrm{d}}w\right)\\
&=\Res_{w}\left(w^{k/4-j-3/2}\sum_{n=0}^{\infty}\gamma_{k,n}w^{n}{\mathrm{d}}w\right)\\
&=\begin{cases}
0 & {\text{if }}j< r,\\
\gamma_{k,j-r} & {\text{if }}j\geq r.
\end{cases}
\end{align*}
Combining these residues, we obtain this proposition. 
\end{proof}

\begin{corollary}[Bases in de Rham side]\label{de_rham_basis}
Let $k$ be a positive integer.
\begin{enumerate}
\item $H^{1}_{\mathrm{dR}}\big(\mathbb{G}_{\mathrm{m}},\sqrt{z}\Sym^{k}{\mathrm{Kl}}_{2}\big)$ has basis
$\left\{v_{0}^{k}z^{j}\frac{{\mathrm{d}}z}{z}\right\}_{j=0}^{k^{\prime}}$.
\item $H^{1}_{\mathrm{dR,c}}\big(\mathbb{G}_{\mathrm{m}},\sqrt{z}\Sym^{k}{\mathrm{Kl}}_{2}\big)$ has basis
\[
\begin{cases}
\left\{\widetilde{\omega}_{k,j}\right\}_{j=0}^{k^{\prime}} & {\text{ if }}k\equiv 0,1,3\bmod{4},\\
\left\{\widetilde{\omega}_{k,j}\right\}_{j=0}^{r-1}\cup \left\{\widetilde{\omega}_{k,j}\right\}_{j=r+1}^{k^{\prime}}\cup \left\{\widehat{m}_{2r+1}\right\} & {\text{ if }}k\equiv 2\bmod{4} {\text{ with }} k=4r+2.
\end{cases}
\]
\item $H_{{\mathrm{dR,mid}}}^{1}\big(\mathbb{G}_{\mathrm{m}},\sqrt{z}\Sym^{k}{\mathrm{Kl}}_{2}\big)$ has basis
\[
\begin{cases}
\left\{\omega_{k,i}\right\}_{i=0}^{k^{\prime}} & {\text{ if }}k\equiv 0,1,3\bmod{4},\\
\left\{\omega_{k,i}\right\}_{i=0}^{r-1}\cup \left\{\omega_{k,i}\right\}_{i=r+1}^{k^{\prime}} & {\text{ if }} k\equiv 2\bmod{4} {\text{ with }} k=4r+2.
\end{cases}
\]
\end{enumerate}
\end{corollary}
\begin{proof}
Putting the dimension result in Proposition \ref{dimthm},
the non-vanishing determinant of the Poincar\'{e} pairing matrix in Proposition \ref{de_rham_paiting_matrix} together,
we obtain this corollary,
thanks to the following simple observation in linear algebra.
\end{proof}

\begin{fact}\label{fact_linear_algebar_fact}
Let $V$ and $W$ be two $n$-dimensional vector spaces over a field $F$. Suppose that $\langle\ ,\ \rangle:V\times W\rightarrow F$ is a bilinear pairing. If $\{v_{1},\ldots,v_{n}\}\subseteq V$ and $\{w_{1},\ldots,w_{n}\}\subseteq W$ are subsets of vectors such that the matrix
\[
\big(\langle v_{i},w_{j}\rangle\big)_{i,j=1\ldots n}\in M_{n}(F)
\]
is invertible,
then $\{v_{1},\ldots,v_{n}\}$ is a basis of $V$ and $\{w_{1},\ldots,w_{n}\}$ is a basis of $W$.
\end{fact}

\section{The local system and the associated homology}\label{section_Betti_side}
In this section, we study the rapid decay homology and moderate decay homology of the local system $(\sqrt{z}\Sym^{k}\mathrm{Kl}_{2})^{\nabla}$. We write down the explicit cycles in these homologies and compute their Betti intersection pairing. In the end, we finish this section by concluding the bases of these two homologies.

In order to write down the cycles in the homology, we need to understand the monodromy action of the horizontal sections of $\sqrt{z}\Sym^{k}\mathrm{Kl}_{2}$. Recall $\{e_0,e_1\}$ is the basis of the local system ${\mathrm{Kl}}_{2}^{\nabla}$ defined in \eqref{horizontal_sections_Kl2}. From \cite[10.25(ii)]{NIST:DLMF}, the modified Bessel function $I_0(t)$ is entire. On the other hand, $K_0(t)$ extends analytically to a multivalued function on $\mathbb{C}^\times$ satisfying the monodromy $K_0(e^{\pi\sqrt{-1}}t) = K_0(t) - \pi\sqrt{-1} I_0(t)$ from \cite[10.34]{NIST:DLMF}. This implies $e_{0},e_{1}$ undergo the monodromy action $T:\left(e_{0},e_{1}\right)\mapsto\left(e_{0},e_{1}+e_{0}\right)$ near $0$.
Then the basis in \eqref{eq:horizontal_basis} of the local system $(\sqrt{z}\Sym^{k}{\mathrm{Kl}}_{2})^{\nabla}$
satisfies $T:\frac{1}{\sqrt{z}}e_{0}^{a}e_{1}^{k-a}\mapsto \frac{-1}{\sqrt{z}}e_{0}^{a}(e_{1}+e_{0})^{k-a}$ near $0$.

\subsection{Rapid decay cycles}
Write $k^{\prime}=\left\lfloor\frac{k-1}{2}\right\rfloor$. Denote the chains on $\mathbb{C}^{\times}$:
\begin{align*}
\sigma_{0}&={\text{the unit circle}}, {\text{ starting at }}1 {\text{ and oriented counterclockwise}};\\
\sigma_{+}&=\text{the interval $\left[1,\infty\right)$}, {\text{ starting at }}1 {\text{ toward }}+\infty.
\end{align*}
By the asymptotic expansion \eqref{expansion_I_0_at_inf}, \eqref{expansion_K_0_at_inf}, the horizontal sections
$\frac{1}{\sqrt{z}}e_{0}^{a}e_{1}^{k-a}$ decay exponentially along $\sigma_{+}$ for $a=0,1,\ldots,k^{\prime}$.
We have the following lemma describing some elements in the rapid decay homology. 

\begin{lemma}
For $0\leq b\leq k'$, the elements
\begin{align}\label{rapid_decay_cycle_delta_b}
\delta_{b}&=\sigma_{+}\otimes\frac{1}{\sqrt{z}}e_{0}^{b}e_{1}^{k-b}-\dfrac{1}{2}\sigma_{0}\otimes\frac{1}{\sqrt{z}}e_{0}^{b}e_{1}^{k-b}+\sum_{n=1}^{k-b}d_{k-b}(n)\sigma_{0}^{2n}\otimes\frac{1}{\sqrt{z}}e_{0}^{b}e_{1}^{k-b},
\end{align}
are rapid decay cycles in $H^{\mathrm{rd}}_{1}\big(\mathbb{G}_{\mathrm{m}},(\sqrt{z}\Sym^{k}{\mathrm{Kl}}_{2})^{\nabla}\big)$, where $d_{n}(i)$ are real numbers satisfying
\begin{align*}
\sum_{i=1}^{n}d_{n}(i)\left(2i\right)^{m}&=-\dfrac{1}{2}, {\text{ for }}m=1,2,\cdots,n.
\end{align*}
In fact, by Cramer's rule, one can write $d_{n}(i)=\frac{(-1)^{i}}{n!2^{n+1}}\binom{n}{i}\frac{\left(2n-1\right)!!}{2i-1}$ uniquely.
\end{lemma}
\begin{proof}
We need to prove that $d_{n}(i)$ makes $\delta_{b}$ into a cycle, that is, $\partial\delta_{b}=0$. The boundaries of chains $\sigma_{+}$ and $\sigma_{0}$ in $\delta_{b}$ support at the point $1\in\mathbb{C}^{\times}$. It suffices to check that the coefficient of $1\in \mathbb{C}^{\times}$ in $\partial\delta_{b}$ is $0$. Indeed, considering the monodromy action $T$ describe above, a direct computation shows the coefficient of $1\in\mathbb{C}^{\times}$ in $\partial\delta_{b}$ is
\begin{align*}
&\frac{1}{\sqrt{z}}e_{0}^{b}e_{1}^{k-b}-\dfrac{1}{2}\left(\frac{1}{\sqrt{z}}e_{0}^{b}e_{1}^{k-b}+\frac{1}{\sqrt{z}}e_{0}^{b}(e_{1}+e_{0})^{k-b}\right)\\
&\hspace{40pt}+\sum_{n=1}^{k-b}d_{k-b}(n)\left(\frac{1}{\sqrt{z}}e_{0}^{b}e_{1}^{k-b}-\frac{1}{\sqrt{z}}e_{0}^{b}(e_{1}+2ne_{0})^{k-b}\right)\\
=&\frac{1}{\sqrt{z}}e_{0}^{b}e_{1}^{k-b}-\dfrac{1}{2}\left(\frac{1}{\sqrt{z}}e_{0}^{b}e_{1}^{k-b}+\sum_{j=0}^{k-b}\binom{k-b}{j}\frac{1}{\sqrt{z}}e_{0}^{b+j}e_{1}^{k-b-j}\right)\\
&\hspace{40pt}+\sum_{n=1}^{k-b}d_{k-b}(n)\left(\frac{1}{\sqrt{z}}e_{0}^{b}e_{1}^{k-b}-\sum_{m=0}^{k-b}\binom{k-b}{m}(2n)^{m}\frac{1}{\sqrt{z}}e_{0}^{b+m}e_{1}^{k-b-m}\right)\\
=&-\dfrac{1}{2}\sum_{j=1}^{k-b}\binom{k-b}{j}\frac{1}{\sqrt{z}}e_{0}^{b+j}e_{1}^{k-b-j}-\sum_{m=1}^{k-b}\binom{k-b}{m}\sum_{n=1}^{k-b}d_{k-b}(n)(2n)^{m}\frac{1}{\sqrt{z}}e_{0}^{b+m}e_{1}^{k-b-m}\\
=&\sum_{j=1}^{k-b}\binom{k-b}{j}\left(-\dfrac{1}{2}-\sum_{n=1}^{k-b}d_{k-b}(n)(2n)^{j}\right)\frac{1}{\sqrt{z}}e_{0}^{b+j}e_{1}^{k-b-j}=0,
\end{align*}
where the last equality is the assumption on real numbers $d_{n}(i)$.
\end{proof}

From this lemma, we have $k^{\prime}+1$ elements $\{\delta_{b}\}_{b=0}^{k'}$ in the rapid decay homology $H^{\mathrm{rd}}_{1}\big(\mathbb{G}_{\mathrm{m}},(\sqrt{z}\Sym^{k}{\mathrm{Kl}}_{2})^{\nabla}\big)$. At the end of this section, we will prove these elements form a basis (see Corollary \ref{Betti_basis_assume_dim}).

\subsection{Moderate decay cycles} 
Define one more chain
\begin{align*}
\mathbb{R}_{+}&= \text{the half line $\left[0,\infty\right)$, starting at $0$ toward $+\infty$}.
\end{align*}
By \cite[\S 10.30(i)]{NIST:DLMF}, the modified Bessel function $K_{0}\left(t\right)$ has log pole at $0$, so the horizontal sections
$\frac{1}{\sqrt{z}}e_{0}^{a}e_{1}^{k-a}$
decay moderately along $\mathbb{R}_{+}$ near $0$ for $a=0,1,\ldots,\lfloor\frac{k}{2}\rfloor$. Moreover, by the expression \eqref{expansion_I_0K_0_at_inf}, $(I_{0}K_{0})^{a}$ decay polynomially along $\mathbb{R}_{+}$ near $\infty$. Then, we define the moderate decay cycles in $H^{\mathrm{mod}}_{1}\big(\mathbb{G}_{\mathrm{m}},(\sqrt{z}\Sym^{k}{\mathrm{Kl}}_{2})^{\nabla}\big)$
\begin{align}\label{moderate_decay_cycle_gamma_a}
\gamma_{a}=\mathbb{R}_{+}\otimes \frac{1}{\sqrt{z}}e_{0}^{a}e_{1}^{k-a}, {\text{ for }}a=0,1,2,\cdots,\left\lfloor\dfrac{k}{2}\right\rfloor.
\end{align}
They are indeed a cycle. The proof is the same as the above lemma by taking the homotopy as the radius of $\sigma_{0}$ tends to $0$ and $\sigma_{+}$ tends to $\mathbb{R}_{+}$. Since a rapid decay cycle is a moderate decay cycle as well, we have the natural map
\[
\begin{tikzcd}[row sep=tiny]
{H^{\mathrm{rd}}_{1}\big(\mathbb{G}_{\mathrm{m}},(\sqrt{z}\Sym^{k}{\mathrm{Kl}}_2)^{\nabla}\big)} \arrow[rr] \arrow[rr] &  & {H^{\mathrm{mod}}_{1}\big(\mathbb{G}_{\mathrm{m}},(\sqrt{z}{\Sym^{k}}{\mathrm{Kl}}_{2})^{\nabla}\big)}.
\end{tikzcd}
\]
This natural map sends $\delta_{b}$ to $\gamma_{b}$ for $b=0,1,\cdots,k^{\prime}$ by the homotopy argument. The following lemma shows when $k\equiv 2\bmod{4}$, $\displaystyle\sum_{j=0}^{\left(k-2\right)/4}\binom{k/2}{2j}\delta_{2j}$ belongs to the kernel of this map.

\begin{lemma}\label{kernel_cycle}
In $H^{\mathrm{mod}}_{1}\big(\mathbb{G}_{\mathrm{m}},(\sqrt{z}\Sym^{k}{\mathrm{Kl}}_{2})^{\nabla}\big)$, one has
\begin{align*}
\sum_{j=0}^{k/4}\binom{k/2}{2j}\gamma_{2j}=0 {\text{ if }}k\equiv 0\bmod{4};\\
\sum_{j=0}^{\left(k-2\right)/4}\binom{k/2}{2j}\gamma_{2j}=0 {\text{ if }}k\equiv 2\bmod{4}.
\end{align*}
\end{lemma}
\begin{proof}
Let $\rho :\left\{\left(x,y\right)\in\mathbb{R}^{2}\mid 0<x,y,x+y<1\right\}\rightarrow\mathbb{C}$ be the open simplicial $2$-chain
\[
\rho\left(x,y\right)=\tan\dfrac{\pi\left(x+y\right)}{2}\exp\left(4\sqrt{-1}\tan^{-1}\dfrac{y}{x}\right)
\]
that covers $\mathbb{C}$ once. If $k$ is even, by the asymptotic expansion \eqref{expansion_I_0K_0_at_inf}, the singular chain
\[
\Delta=\rho\otimes\left(\frac{1}{\sqrt{z}}(e_{1}-e_{0})^{k/2}e_{1}^{k/2}\right)
\]
has moderate growth. The boundary of $\rho$ consists of two positive real lines $\mathbb{R}_{+}$. From the monodromy action $T:(e_{0},e_{1})\mapsto (e_{0},e_{1}+e_{0})$, one computes $\partial\Delta$:
\begin{align*}
\partial\Delta &=\mathbb{R}_{+}\otimes\left(\frac{1}{\sqrt{z}}(e_{1}-e_{0})^{k/2}e_{1}^{k/2}\right)+\mathbb{R}_{+}\otimes\left(\frac{1}{\sqrt{z}}e_{1}^{k/2}(e_{0}+e_{1})^{k/2}\right)\\
&=\sum_{i=0}^{k/2}(-1)^{i}\binom{k/2}{i}\mathbb{R}_{+}\otimes \frac{1}{\sqrt{z}}e_{0}^{i}e_{1}^{k-i}+\sum_{i=0}^{k/2}\binom{k/2}{i}\mathbb{R}_{+}\otimes\frac{1}{\sqrt{z}}e_{0}^{i}e_{1}^{k-i}\\
&=\sum_{i=0}^{k/2} (1+(-1)^{i})\binom{k/2}{i}\gamma_{i}.
\end{align*}
When $k\equiv 0\bmod{4}$, this reads 
\[
\frac{1}{2}\partial\Delta = \sum_{j=0}^{k/4}\binom{k/2}{2j}\gamma_{2j}.
\]
Thus, $\sum_{j=0}^{k/4}\binom{k/2}{2j}\gamma_{2j}$ is homologous to zero in $H_{1}^{\mathrm{mod}}(\mathbb{G}_{\mathrm{m}},(\sqrt{z}\Sym^{k}\mathrm{Kl}_{2})^{\nabla})$. The case when $k\equiv 2\bmod{4}$ is similar.
\end{proof}

Here, we have written down the $1+\lfloor\frac{k}{2}\rfloor$ elements $\{\gamma_{a}\}_{a=0}^{\lfloor k/2\rfloor}$ in the moderate decay homology $H^{\mathrm{mod}}_{1}\big(\mathbb{G}_{\mathrm{m}},(\sqrt{z}\Sym^{k}{\mathrm{Kl}}_{2})^{\nabla}\big)$. At the end of this section, we will prove that these elements form a basis modulo the linear relation given in the above lemma (see Corollary \ref{Betti_basis_assume_dim}).

Similar to the middle part de Rham cohomology in the previous section, we define the middle part Betti homology $H_{1}^{{\mathrm{mid}}}\big(\mathbb{G}_{\mathrm{m}},(\sqrt{z}\Sym^{k}{\mathrm{Kl}}_{2})^{\nabla}\big)$ to be the image of $H_{1}^{\mathrm{rd}}\big(\mathbb{G}_{\mathrm{m}},(\sqrt{z}\Sym^{k}{\mathrm{Kl}}_{2})^{\nabla}\big)$ in $H_{1}^{\mathrm{mod}}\big(\mathbb{G}_{\mathrm{m}},(\sqrt{z}\Sym^{k}{\mathrm{Kl}}_{2})^{\nabla}\big)$. More precisely, we have 
\begin{align*}
\gamma_{i}\in H_{1}^{{\mathrm{mid}}}\big(\mathbb{G}_{\mathrm{m}},(\sqrt{z}\Sym^{k}{\mathrm{Kl}}_{2})^{\nabla}\big) & {\text{ for }}0\leq i\leq k^{\prime} {\text{ when }}k\equiv 0,1,3\bmod{4};\\
\gamma_{i}\in H_{1}^{{\mathrm{mid}}}\big(\mathbb{G}_{\mathrm{m}},(\sqrt{z}\Sym^{k}{\mathrm{Kl}}_{2})^{\nabla}\big) & {\text{ for }}1\leq i\leq k^{\prime} {\text{ when }}k\equiv 2\bmod{4}.
\end{align*}
Also, we may regard $H_{1}^{{\mathrm{mid}}}\big(\mathbb{G}_{\mathrm{m}},(\sqrt{z}\Sym^{k}{\mathrm{Kl}}_{2})^{\nabla}\big)$ as the quotient of $H^{\mathrm{rd}}_{1}\big(\mathbb{G}_{\mathrm{m}},(\sqrt{z}\Sym^{k}{\mathrm{Kl}}_{2})^{\nabla}\big)$ containing the class of elements $\delta_{b}$. At the end of this section, we will prove these elements form a basis (see Corollary \ref{Betti_basis_assume_dim}).

\subsection{Betti intersection pairing}
We use the topological pairing $\left\langle\ ,\ \right\rangle_{\mathrm{top}}$ introduced in section \ref{algebraic_topological_pairing} to define the Betti intersection pairing
\[
\begin{tikzcd}[row sep=tiny]
{H^{\mathrm{rd}}_{1}\big(\mathbb{G}_{\mathrm{m}},(\sqrt{z}\Sym^{k}{\mathrm{Kl}}_{2})^{\nabla}\big)\times H^{\mathrm{mod}}_{1}\big(\mathbb{G}_{\mathrm{m}},(\sqrt{z}\Sym^{k}{\mathrm{Kl}}_{2})^{\nabla}\big)} \arrow[rr, "{\left\langle\ ,\ \right\rangle_{\mathrm{Betti}}}"] &  & \mathbb{Q}                                                                                    \\
{\big(\delta=\sum_{i}\sigma_{i}\otimes s_{\sigma_{i}},\gamma =\sum_{j}\tau_{j}\otimes s_{\tau_{j}}\big)} \arrow[rr,maps to]                                                                                                                                                                             &  & {\displaystyle\sum_{i,j}\sum_{\sigma_{i}\cap\tau_{j}}\left\langle s_{\sigma_{i}},s_{\tau_{j}}\right\rangle_{\mathrm{top}}}
\end{tikzcd}
\]
Here, we need to find representatives of $\delta=\sum\sigma_{i}\otimes s_{\sigma_{i}}$ and $\gamma=\sum\tau_{j}\otimes s_{\tau_{j}}$ in their homology classes respectively such that any two chains $\sigma_{i}$ and $\tau_{j}$ intersect transversally for all $i,j$. Then, for each pair $(i,j)$, $\sigma_{i}\cap\tau_{j}$ consists of only finitely many topological intersection points. The sum over $\sigma_{i}\cap\tau_{j}$ is then the sum of the topological pairings of the corresponding sections at each intersection point. Note that the Betti intersection pairing induces on the middle part Betti homology which we still call it $\langle\ ,\ \rangle_{\mathrm{Betti}}$:
\[
\begin{tikzcd}[row sep=tiny]
{H_{1}^{\mathrm{mid}}(\mathbb{G}_{\mathrm{m}},(\sqrt{z}\Sym^{k}{\text{Kl}}_{2})^{\nabla})\times H_{1}^{\mathrm{mid}}(\mathbb{G}_{\mathrm{m}},(\sqrt{z}\Sym^{k}{\text{Kl}}_{2})^{\nabla})} \arrow[rr, "{\left\langle\ ,\ \right\rangle_{\text{Betti}}}"] &  & \mathbb{Q}                                                                                                                 
\end{tikzcd}
\]

To compute the topological pairing with respect to the elements we had written down, we need to introduce the Euler numbers and Euler polynomials. The Euler polynomials $E_{n}(x)$ are given by the following power series, and we define the numbers $E_{n}$ for $n\geq 0$ as in \cite{kim2005note},
\[
\sum_{n=0}^{\infty}E_{n}(x)\frac{z^{n}}{n!}=\frac{2}{e^{z}+1}e^{xz},\ \ E_{n}=E_{n}(0).
\]
The first few $E_{n}$ are
\[
\begin{tabular}{c|c|c|c|c|c|c}
$E_{0}$ & $E_{1}$ & $E_{2}$ & $E_{3}$ & $E_{4}$ & $E_{5}$ & $E_{6}$ \\ \hline
$1$     & $-1/2$  & $0$     & $1/4$   & $0$     & $-1/2$  & $0$    
\end{tabular}
\]
We have the inversion formula for Euler polynomials,
\[
x^{n}=E_{n}(x)+\dfrac{1}{2}\sum_{k=0}^{n-1}\binom{n}{k}E_{k}(x).
\]
Evaluating at $x=0$, we get
\begin{align}\label{Euler_number_relation}
\sum_{k=0}^{n-1}\binom{n}{k}E_{k}=-2E_{n}.
\end{align}

\begin{proposition}\label{Betti_intersection_pairing}
We have the Betti intersection pairing
\[
\left<\delta_{b},\gamma_{a}\right>_{\mathrm{Betti}}=(-1)^{a}\frac{\binom{k-b}{a}}{\binom{k}{a}}
\frac{-1}{2}E_{k-a-b}=\frac{(-1)^{a+1}}{2}\frac{(k-a)!(k-b)!}{k!}\frac{E_{k-a-b}}{(k-a-b)!}
\]
for $b=0,\cdots,k^{\prime}$ and $a=0,\cdots,\left\lfloor\frac{k}{2}\right\rfloor$.


\end{proposition}

\begin{proof}
Fix some $-\pi<\theta_{0}<0$ and let $x_{0}=\exp(\sqrt{-1}\theta_{0})$. To compute the pairing $\left<\delta_{b},\gamma_{a}\right>_{\mathrm{Betti}}$, we move the ray $\sigma_{+}$ by adding the scalar $(x_{0}-1)$ and let the circle $\sigma_{0}$ start at $x_{0}$. Then the component $\sigma_{0}^{j}\otimes \frac{1}{\sqrt{z}}e_{0}^{b}e_{1}^{k-b}$ in the deformed $\delta_{b}$ meets $\gamma_{a}$ topologically $j$ times at the same point $+1\in\mathbb{C}^{\times}$. At the $i$-th intersection, the factor $\frac{1}{\sqrt{z}}e_{0}^{b}e_{1}^{k-b}$ becomes $(-1)^{i-1}\frac{1}{\sqrt{z}}e_{0}^{b}(e_{1}+(i-1)e_{0})^{k-b}$ and we have
\[
\left\langle(-1)^{i-1}\frac{1}{\sqrt{z}}e_{0}^{b}(e_{1}+(i-1)e_{0})^{k-b},\frac{1}{\sqrt{z}}e_{0}^{a}e_{1}^{k-a}\right\rangle_{\mathrm{top}}=(-1)^{i-1}(i-1)^{k-a-b}(-1)^{a}\dfrac{\binom{k-b}{a}}{\binom{k}{a}}.
\]
By adding these contributions, we obtain
\[
\left<\delta_{b},\gamma_{a}\right>_{\mathrm{Betti}}=\dfrac{\binom{k-b}{a}}{\binom{k}{a}}(-1)^{a}\sum_{n=1}^{k-b}d_{k-b}(n)T_{k-a-b}(2n),
\]
where
\[
T_{n}(k)=\sum_{\ell =0}^{k-1}(-1)^{\ell}\ell^{n}=-1+2^{n}-\cdots +(-1)^{k-1}(k-1)^{n}.
\]

Kim \cite{kim2005note} gave the following relation for $T_{n}(k)$:
\[
T_{n}(k)=\dfrac{(-1)^{k+1}}{2}\sum_{\ell =0}^{n-1}\binom{n}{\ell}E_{\ell}k^{n-\ell}+\dfrac{E_{n}}{2}\left(1+(-1)^{k+1}\right).
\]
Now, we have the following computation
\begin{align*}
\sum_{n=1}^{k-b}d_{k-b}(n)T_{k-a-b}(2n)&=\sum_{n=1}^{k-b}d_{k-b}(n)\left[\dfrac{-1}{2}\sum_{\ell =0}^{k-a-b-1}\binom{k-a-b}{\ell}E_{\ell}(2n)^{k-a-b-\ell}\right]\\
&=\dfrac{-1}{2}\sum_{\ell =0}^{k-a-b-1}\binom{k-a-b}{\ell}E_{\ell}\sum_{n=1}^{k-b}d_{k-b}(n)(2n)^{k-a-b-\ell}\\
&=\dfrac{1}{4}\displaystyle\sum_{\ell =0}^{k-a-b-1}\binom{k-a-b}{\ell}E_{\ell} \\
&= \dfrac{-1}{2}E_{k-a-b}
\end{align*}
where the last equality follows from \eqref{Euler_number_relation}.
\end{proof}

Consider the $(k^{\prime}+1)\times (k^{\prime}+1)$ pairing matrix
\[
B_{k}=
\begin{cases}
\left(\left<\delta_{b},\gamma_{a}\right>_{\mathrm{Betti}}\right)_{\ 0\leq b\leq k^{\prime},\ 0\leq a\leq\lfloor \frac{k}{2}\rfloor} & {\text{ if }}k {\text{ is odd,}}\\
\left(\left<\delta_{b},\gamma_{a}\right>_{\mathrm{Betti}}\right)_{0\leq b\leq k^{\prime},\ 1\leq a\leq\frac{k}{2}} & {\text{ if }}k {\text{ is even.}}
\end{cases}
\]
By Proposition \ref{Betti_intersection_pairing},
when $k$ is even, we have
\[
B_{k}=
\left(\begin{array}{ccc}
\frac{(-1)^{2}}{2}\frac{(k-1)!k!}{k!}\frac{E_{k-1}}{(k-1)!} & \cdots & \frac{(-1)^{k/2+1}}{2}\frac{(k/2)!k!}{k!}\frac{E_{k/2}}{(k/2)!}\\ 
\vdots & \ddots &\vdots \\
\frac{(-1)^{2}}{2}\frac{(k-1)!(k/2+1)!}{k!}\frac{E_{k/2}}{(k/2)!} & \cdots &\frac{(-1)^{k/2+1}}{2}\frac{(k/2)!(k/2+1)!}{k!}\frac{E_{1}}{\left(1\right)!}
\end{array}\right)
\]
and that
\[
B_{k-1}=
\left(\begin{array}{ccc}
\frac{(-1)}{2}\frac{(k-1)!(k-1)!}{(k-1)!}\frac{E_{k-1}}{(k-1)!} & \cdots & \frac{(-1)^{k/2}}{2}\frac{(k/2)!(k-1)!}{(k-1)!}\frac{E_{k/2}}{(k/2)!}\\ 
\vdots & \ddots &\vdots \\
\frac{(-1)}{2}\frac{(k-1)!(k/2)!}{(k-1)!}\frac{E_{k/2}}{(k/2)!} & \cdots &\frac{(-1)^{k/2}}{2}\frac{(k/2)!(k/2)!}{(k-1)!}\frac{E_{1}}{\left(1\right)!}
\end{array}\right).
\]
Then we obtain the relation
\begin{equation}\label{Betti_relation}
B_{k}=-\frac{1}{k} \diag(k,k-1,\cdots,k/2+1)\cdot B_{k-1}.
\end{equation}

Thus, $B_{k}$ and $B_{k-1}$ have the same rank whenever $k$ is even. Moreover, we may compute the determinant of $B_{k}$ explicitly as given in the following proposition.
\begin{proposition}\label{Betti_determinant}
The determinant of $B_{k}$ is given by the following.
\begin{enumerate} 
\item When $k$ is odd, we have
\begin{align*}
\det B_{k}&=2^{-k-1}\prod_{a=1}^{k^{\prime}}a^{k^{\prime}+1-2a}(2a+1)^{k^{\prime}-2a}.
\end{align*}

\item When $k$ is even, we have
\begin{align*}
\det B_{k}&=(-1)^{(k^{\prime}+1)(k^{\prime}+3)}2^{-k}\prod_{a=1}^{k^{\prime}}(a+1)^{k^{\prime}-2a-1}(2a+1)^{k^{\prime}+1-2a}.
\end{align*}
\end{enumerate}
In particular, they are all non-vanishing.
\end{proposition}
\begin{proof}
Set $\mathcal{E}_{2n-1}=(-1)^{n}2^{2n-1}E_{2n-1}$. Apply the result \cite[Eq. H12]{MR4105523} in the following computations.

When $k=2k'+1$ is odd, we have
\begin{align*}
\det B_{k} &=\frac{(-1)^{(k'+1)(k'+2)/2}}{(2\cdot k!)^{k'+1}}\left[\prod_{i=k'+1}^{k}i!\right]^{2}\det\left(\begin{array}{ccc}
\frac{E_{k}}{k!} & \cdots &\frac{E_{k'+1}}{(k'+1)!}\\ 
\vdots & \ddots &\vdots \\
\frac{E_{k'+1}}{(k'+1)!} & \cdots &\frac{E_{1}}{\left(1\right)!}
\end{array}\right)\\
&=\frac{1}{(2\cdot k!)^{k'+1}}\left[\prod_{i=k'+1}^{k}i!\right]^{2}\frac{1}{2^{(k'+1)^{2}}}\det\left(\begin{array}{ccc}
\frac{\mathcal{E}_{k}}{k!} & \cdots &\frac{\mathcal{E}_{k'+1}}{(k'+1)!}\\ 
\vdots & \ddots &\vdots \\
\frac{\mathcal{E}_{k'+1}}{(k'+1)!} & \cdots &\frac{\mathcal{E}_{1}}{\left(1\right)!}
\end{array}\right)\\
&=\frac{1}{2^{(k'+1)(k'+2)}(k!)^{k'+1}}\left[\prod_{i=k'+1}^{k}i!\right]^{2} 2^{k'^{2}}\frac{k'!}{k!}\prod_{j=1}^{k'}\frac{(j-1)!^{2}}{(2j-1)!^{2}}\\
&=\frac{1}{2^{k+1}}\prod_{a=1}^{k^{\prime}}a^{k^{\prime}+1-2a}(2a+1)^{k^{\prime}-2a}.
\end{align*}

When $k=2k'+2$ is even, we have
\begin{align*}
\det B_{k}&=\frac{(-1)^{(k'+1)(k'+4)/2}}{(2\cdot k!)^{k'+1}}\left[\prod_{i=k'+1}^{k-1}i!\left(i+1\right)!\right] \det\left(\begin{array}{ccc}
\frac{E_{k-1}}{(k-1)!} & \cdots &\frac{E_{k'+1}}{(k'+1)!}\\ 
\vdots & \ddots &\vdots \\
\frac{E_{k'+1}}{(k'+1)!} & \cdots &\frac{E_{1}}{\left(1\right)!}
\end{array}\right)\\
&=\frac{(-1)^{(k'+1)(k'+4)/2}}{(2\cdot k!)^{k'+1}}\left[\prod_{i=k'+1}^{k-1}i!\left(i+1\right)!\right] \\
&\hspace{100pt}\cdot\frac{(\sqrt{-1})^{(k'+1)(k'+2)}}{2^{(k'+1)^2}}\det\left(\begin{array}{ccc}
\frac{\mathcal{E}_{k-1}}{(k-1)!} & \cdots &\frac{\mathcal{E}_{k'+1}}{(k'+1)!}\\ 
\vdots & \ddots &\vdots \\
\frac{\mathcal{E}_{k'+1}}{(k'+1)!} & \cdots &\frac{\mathcal{E}_{1}}{\left(1\right)!}
\end{array}\right)\\
&=\frac{(-1)^{(k'+1)(k'+3)}}{2^{(k'+1)(k'+2)}(k!)^{k'+1}}\left[\prod_{i=k'+1}^{k-1}i!\left(i+1\right)!\right] 2^{k'^{2}}\frac{k'!}{(k-1)!}\prod_{j=1}^{k'}\frac{(j-1)!^{2}}{(2j-1)!^{2}}\\
&=\frac{(-1)^{(k^{\prime}+1)(k^{\prime}+3)}}{2^k}\prod_{a=1}^{k^{\prime}}(a+1)^{k^{\prime}-2a-1}(2a+1)^{k^{\prime}+1-2a}.
\qedhere
\end{align*}
\end{proof}

Finally, before we conclude the basis of Betti homologies, we need to introduce the period pairings here. However, the details of the pairings will be given in the next section. By \cite[Corollary 2.11]{MR4578001}, there exist two perfect pairings 
\[
\begin{tikzcd}
{H_{1}^{\mathrm{rd}}\big(\mathbb{G}_{\mathrm{m}},(\sqrt{z}\Sym^{k}{\mathrm{Kl}}_{2})^{\nabla}\big)_{\mathbb{C}}\times H^{1}_{\mathrm{dR}}\big(\mathbb{G}_{\mathrm{m}},\sqrt{z}\Sym^{k}{\mathrm{Kl}}_{2}\big)_{\mathbb{C}}} \arrow[rr, "{\left\langle\ ,\ \right\rangle_{\mathrm{per}}}"] &  & \mathbb{C}
\end{tikzcd}
\]
\[
\begin{tikzcd}
{H^{1}_{\mathrm{dR,c}}\big(\mathbb{G}_{\mathrm{m}},\sqrt{z}\Sym^{k}{\mathrm{Kl}}_{2}\big)_{\mathbb{C}}\times H_{1}^{\mathrm{mod}}\big(\mathbb{G}_{\mathrm{m}},(\sqrt{z}\Sym^{k}{\mathrm{Kl}}_{2})^{\nabla}\big)_{\mathbb{C}}} \arrow[rr, "{\left\langle\ ,\ \right\rangle_{\mathrm{per,c}}}"] &  & \mathbb{C}
\end{tikzcd}
\]
Here, the notation $V_{\mathbb{C}}$ means $V\otimes_{\mathbb{Q}}\mathbb{C}$. For the next corollary, we just need to use the fact that these pairings are perfect. In the next section, we will compute these two pairings explicitly.

\begin{corollary}\label{Betti_basis_assume_dim}
The natural map
\[
\begin{tikzcd}
{H^{\mathrm{rd}}_{1}\big(\mathbb{G}_{\mathrm{m}},(\sqrt{z}{\Sym^{k}{\mathrm{Kl}}_2)^{\nabla}}\big)} \arrow[rr] \arrow[rr] &  & {H^{\mathrm{mod}}_{1}\big(\mathbb{G}_{\mathrm{m}},(\sqrt{z}{\Sym^{k}{\mathrm{Kl}}_2)^{\nabla}}\big)}                 
\end{tikzcd}
\]
sending $\delta_{b}$ to $\gamma_{b}$ is an isomorphism when $k\equiv 0,1,3\bmod 4$ and has a one-dimensional kernel when $k\equiv 2\bmod 4$. Moreover, we find the following.
\begin{enumerate}
\item $H^{\mathrm{rd}}_{1}\big(\mathbb{G}_{\mathrm{m}},(\sqrt{z}\Sym^{k}{\mathrm{Kl}}_{2})^{\nabla}\big)$ has basis $\{\delta_{b}\}_{b=0}^{k^{\prime}}$.
\item $H^{\mathrm{mod}}_{1}\big(\mathbb{G}_{\mathrm{m}},(\sqrt{z}\Sym^{k}{\mathrm{Kl}}_{2})^{\nabla}\big)$ has basis
\[
\begin{cases}
\left\{\gamma_{a}\right\}_{a=0}^{k^{\prime}} & {\text{ if }} k {\text{ is odd;}}\\
\left\{\gamma_{a}\right\}_{a=1}^{k/2} & {\text{ if }} k {\text{ is even.}}
\end{cases}
\]
\item $H^{\mathrm{mid}}_{1}\big(\mathbb{G}_{\mathrm{m}},(\sqrt{z}\Sym^{k}{\mathrm{Kl}}_{2})^{\nabla}\big)$ has basis
\[
\begin{cases}
\left\{\gamma_{a}\right\}_{a=0}^{k^{\prime}} & {\text{ if }}k\equiv 0,1,3\bmod{4};\\
\left\{\gamma_{a}\right\}_{a=1}^{k^{\prime}} & {\text{ if }}k\equiv 2\bmod{4}.\\
\end{cases}
\]
\end{enumerate}
\end{corollary}

\begin{proof}
From the perfect period pairings, the dimension of rapid decay homology and moderate decay homology are both $k^{\prime}+1$ by Proposition \ref{dimthm}. Then, by the Fact \ref{fact_linear_algebar_fact} and the non-vanishing determinant of $B_{k}$ in Proposition \ref{Betti_determinant}, we conclude $1.$ and $2.$ This also shows the natural map which sends $\delta_{b}$ to $\gamma_{b}$ for $b=0,\ldots,k'$ is an isomorphism when $k\equiv 1,3\bmod 4$. When $k\equiv 2\bmod{4}$, Lemma \ref{kernel_cycle} describes the one-dimensional kernel of the natural map. Moreover, $B_{k}$ has full rank $k/2$ when $k$ is even by the relation \eqref{Betti_relation}. Hence, we conclude that the natural map is an isomorphism when $k\equiv 0\bmod{4}$.
\end{proof}

\section{Twisted moments as periods}\label{section_period_pairing}

In this section, we compute the period pairing of the basis of de Rham cohomology and Betti homology in Corollary \ref{de_rham_basis} and Corollary \ref{Betti_basis_assume_dim}. Also, we interpret these periods as the Bessel moments and regularized Bessel moments.

\subsection{Bessel moments and regularized Bessel moments}\label{section_Bessel_moments}

The Bessel moments are defined by
\begin{align*}
{\mathrm{IKM}}_{k}(a,b)=\int_{0}^{\infty}I_{0}^{a}(t)K_{0}^{k-a}(t)t^{b}\mathrm{d}t.
\end{align*}
provided the convergence of the integral, that is, for non-negative integers $k,a,b$ satisfying $a\leq k^{\prime}$, $b\geq 0$ or $a=\frac{k}{2}$, $0\leq b<k^{\prime}$. The justification is given in the following lemma. Moreover, if $a=\frac{k}{2}$ and $b\geq k^{\prime}$, by analyzing the singular integral, we could define the regularized Bessel moments $\mathrm{IKM}_{k}^{\mathrm{reg}}\left(\frac{k}{2},b\right)$ by subtracting the singular part of the integral. The precise definition is also given in the following lemma.

\begin{lemma}\label{regularized_lemma} The integral expression of Bessel moments
\[
{\mathrm{IKM}}_{k}(a,b)=\int_{0}^{\infty}I_{0}^{a}(t)K_{0}^{k-a}(t)t^{b}{\mathrm{d}}t
\]
converges for non-negative integers $k,a,b$ satisfying $a\leq k^{\prime}$, $b\geq 0$ or $a=\frac{k}{2}$, $0\leq b<k^{\prime}$. Moreover, in the case that $k$ is even, $a=\frac{k}{2}$, and $b\geq k^{\prime}$ with $b$ even, the following two limits exist for $2j\geq k^{\prime}$:
\begin{align*}
{\mathrm{IKM}}^{\mathrm{reg}}_{k}\left(\frac{k}{2},2j\right)&:=\lim_{t\rightarrow\infty}\left(\int_{0}^{t}(I_{0}K_{0})^{2r+2}s^{2j}{\mathrm{d}}s-\sum_{m=0}^{j-r-1}\frac{\gamma_{k,j-r-1-m}t^{2m+1}}{2^{k-2j+2m}(2m+1)}\right)\ \ {\text{if }}k=4r+4,\\
{\mathrm{IKM}}^{\mathrm{reg}}_{k}\left(\frac{k}{2},2j\right)&:=\lim_{\substack{t\rightarrow\infty\\\varepsilon\rightarrow 0^{+}}}\left(\int_{\varepsilon}^{t}(I_{0}K_{0})^{2r+1}s^{2j}{\mathrm{d}}s-\frac{2\gamma_{k,j-r}}{2^{k-2j}}\int_{\varepsilon}^{t}\frac{\mathrm{d}s}{s}-\sum_{m=0}^{j-r-1}\frac{\gamma_{k,j-r-1-m}t^{2m+2}}{2^{k-2j+2m+1}(2m+2)}\right)\\
& \hspace{10cm}{\text{if }}k=4r+2.
\end{align*}
\end{lemma}
\begin{proof}
Near $0$, by \cite[\S 10.30(i)]{NIST:DLMF} we have the asymptotics
\begin{align}
I_0(t)&=1+O(t^{2});\label{expansion_of_I_0_at_0}\\
K_0(t)&=-\left(\gamma +\log\frac{t}{2}\right)+O(t^{2}\log t),\label{expansion_of_K_0_at_0}
\end{align}
where $\gamma$ is the Euler constant. Then, the integral $\int_{0}^{1}I_{0}^{a}(t)K_{0}^{k-a}(t)t^{b}{\mathrm{d}}t$ converges for all $0\leq a\leq\frac{k}{2}$ and any $b\geq 0$.

Near $\infty$, from \eqref{expansion_I_0_at_inf} and \eqref{expansion_K_0_at_inf}, when $0\leq a\leq k^{\prime}$, $I_{0}^{a}(t)K_{0}^{k-a}(t)$ decays exponentially and hence the integral $\int_{1}^{\infty}I_{0}^{a}(t)K_{0}^{k-a}(t)t^{b}{\mathrm{d}}t$ converges.

When $k$ is even and $a=\frac{k}{2}$, near $\infty$, by \eqref{expansion_of_(-AB)^half}, we have the asymptotic expansion
\[
\left(I_{0}K_{0}\right)^{k/2}t^{2j}=\frac{1}{2^{k/2}}\sum_{n=0}^{\infty}\gamma_{k,n}4^{n}t^{-2n-k/2+2j}.
\]
Taking integration, we have
\[
\int_{\varepsilon}^{t}(I_{0}K_{0})^{k/2}s^{2j}\mathrm{d}s = \frac{1}{2^{k/2}}\sum_{n=0}^{\infty}\gamma_{k,n}4^{n}\int_{\varepsilon}^{t} s^{-2n-k/2+2j}\mathrm{d}s.
\]
Using the fact that $\int_{1}^{\infty}t^{\alpha}{\mathrm{d}}t$ converges if and only if $\alpha <-1$ and $\int_{0}^{1}t^{\alpha}{\mathrm{d}}t$ converges if and only if $\alpha >-1$, the divergent part of the integral
$
\int_{\varepsilon}^{t}\left(I_{0}K_{0}\right)^{k/2}s^{2j}{\mathrm{d}}s
$ as $t\rightarrow\infty$, $\varepsilon\rightarrow 0^{+}$ is
\begin{align}
\sum_{m=0}^{j-r-1}\frac{\gamma_{k,j-r-1-m}}{2^{k-2j+2m}}\frac{t^{2m+1}}{\left(2m+1\right)} & {\text{ if }}k=4r+4\label{Bessel_moments_singular_term_4r+4}\\
\frac{2\gamma_{k,j-r}}{2^{k-2j}}\int_{\varepsilon}^{t}\frac{\mathrm{d}s}{s}+\sum_{m=0}^{j-r-1}\frac{\gamma_{k,j-r-1-m}}{2^{k-2j+2m+1}}\frac{t^{2m+2}}{(2m+2)}& {\text{ if }}k=4r+2\label{Bessel_moments_singular_term_4r+2}
\end{align}
Hence, after subtracting the divergent part of the integral, we conclude that the limits $\mathrm{IKM}_{k}^{\mathrm{reg}}(\frac{k}{2},2j)$ exist. 
\end{proof}

\begin{remark}
For $a=\frac{k}{2}$ and $b\geq k^{\prime}$ with odd $b$, the integral $\int_{0}^{\infty}I_{0}^{a}(t)K_{0}^{k-a}(t)t^{b}{\mathrm{d}}t$ also diverges. We may similarly define the regularized Bessel moments in this case.
See \cite[Definitions 6.1 and 6.4]{MR4578001}.
\end{remark}

\subsection{Period pairing and compactly supported period pairing}
By \cite[Corollary 2.11]{MR4578001}, there exist the following two perfect pairings. The period pairing is defined to be
\[
\begin{tikzcd}
{H_{1}^{\mathrm{rd}}\big(\mathbb{G}_{\mathrm{m}},(\sqrt{z}\Sym^{k}{\mathrm{Kl}}_{2})^{\nabla}\big)_{\mathbb{C}}\times H^{1}_{\mathrm{dR}}\big(\mathbb{G}_{\mathrm{m}},\sqrt{z}\Sym^{k}{\mathrm{Kl}}_{2}\big)_{\mathbb{C}}} \arrow[rr, "{\left\langle\ ,\ \right\rangle_{\mathrm{per}}}"] &  & \mathbb{C}
\end{tikzcd}
\]
by
\[
\left\langle\sigma\otimes\frac{1}{\sqrt{z}}e_{0}^{b}e_{1}^{k-b},\omega\right\rangle_{\mathrm{per}}=\int_{\sigma}\frac{1}{\sqrt{z}}\left\langle e_{0}^{b}e_{1}^{k-b},\omega\right\rangle_{\mathrm{top}}.
\]
Here, the notation $V_{\mathbb{C}}$ means $V\otimes_{\mathbb{Q}}\mathbb{C}$. There is a one-form $\omega$ occurs in $\langle e_{0}^{b}e_{1}^{k-b},\omega\rangle_{\mathrm{top}}$. This topological pairing means $\langle e_{0}^{b}e_{1}^{k-b},f\rangle_{\mathrm{top}}\mathrm{d}z$ whenever $\omega = f\mathrm{d}z$. That is, we take the pairing $\langle\ ,\ \rangle_{\mathrm{top}}$ only on the coefficients. Note that the period pairing induces on the middle part Betti homology $H_{1}^{\mathrm{mid}}(\mathbb{G}_{\mathrm{m}},(\sqrt{z}\Sym^{k}{\mathrm{Kl}}_{2})^{\nabla})_{\mathbb{C}}$ and middle part de Rham cohomology $H^{1}_{\mathrm{mid}}(\mathbb{G}_{\mathrm{m}},\sqrt{z}\Sym^{k}{\mathrm{Kl}}_{2})_{\mathbb{C}}$ by the restriction:
\[
\begin{tikzcd}
{H_{1}^{\mathrm{mid}}\big(\mathbb{G}_{\mathrm{m}},(\sqrt{z}\Sym^{k}{\mathrm{Kl}}_{2})^{\nabla}\big)_{\mathbb{C}}\times H^{1}_{\mathrm{mid}}\big(\mathbb{G}_{\mathrm{m}},\sqrt{z}\Sym^{k}{\mathrm{Kl}}_{2}\big)_{\mathbb{C}}} \arrow[rr, "{\left\langle\ ,\ \right\rangle_{\mathrm{per}}}"] &  & \mathbb{C}.
\end{tikzcd}
\]

Moreover, the compactly supported period pairing is defined to be
\[
\begin{tikzcd}
{H^{1}_{\mathrm{dR,c}}\big(\mathbb{G}_{\mathrm{m}},\sqrt{z}\Sym^{k}{\mathrm{Kl}}_{2}\big)_{\mathbb{C}}\times H_{1}^{\mathrm{mod}}\big(\mathbb{G}_{\mathrm{m}},(\sqrt{z}\Sym^{k}{\mathrm{Kl}}_{2})^{\nabla}\big)_{\mathbb{C}}} \arrow[rr, "{\left\langle\ ,\ \right\rangle_{\mathrm{per,c}}}"] &  & \mathbb{C}
\end{tikzcd}
\]
by
\begin{align*}
\left\langle(\xi,\eta,\omega),\sigma\otimes\frac{1}{\sqrt{z}}e_{0}^{b}e_{1}^{k-b}\right\rangle_{\mathrm{per,c}}&=\int_{\sigma}\frac{1}{\sqrt{z}}\left\langle\omega , e_{0}^{b}e_{1}^{k-b}\right\rangle_{\mathrm{top}}
-\frac{1}{\sqrt{z}}\left\langle\eta, e_{0}^{b}e_{1}^{k-b}\right\rangle_{\mathrm{top}}+\frac{1}{\sqrt{z}}\left\langle\xi,e_{0}^{b}e_{1}^{k-b}\right\rangle_{\mathrm{top}}.
\end{align*}

\begin{remark}
Note that the order of homology and cohomology in these two pairing are different. This is because we want to write down the matrix expression of quadratic relation \eqref{quadratic_matrix_relation} preventing the transpose notation.
\end{remark}

\begin{proposition}\label{period_pairing}
The period pairing of the rapid decay cycle $\delta_{b}$ in \eqref{rapid_decay_cycle_delta_b} and the de Rham cohomology class $\omega_{k,j}$ in Definition \ref{de_Rham_basis_elements} is given by
\[
\left\langle \delta_{b},v_{0}^{k}z^{j}\frac{\mathrm{d}z}{z}\right\rangle_{\mathrm{per}}=(\pi\sqrt{-1})^{b}(-1)^{k-b}2^{k-2j}{\mathrm{IKM}}_{k}(b,2j)
\]
for $0\leq b\leq k^{\prime}$ and $0\leq j\leq k^{\prime}$.
\end{proposition}
\begin{proof}
Denote $\varepsilon \sigma_{0}$ to be the scaling of the chain $\sigma_{0}$, that is, $\varepsilon\sigma_{0}$ is a chain of a circle of radius $\varepsilon$. Similarly, denote $\varepsilon\sigma_{+}$ to be the chain of the ray $[\varepsilon,\infty)$. Then, since $\sigma_{0}$ and $\sigma_{+}$ are homotopy to $\varepsilon\sigma_{0}$ and $\sigma_{+}$ respectively, we may replace $\sigma_{0}$ and $\sigma_{+}$ in $\delta_{b}$ by $\varepsilon\sigma_{0}$ and $\varepsilon\sigma_{+}$ respectively in the following computation. We compute
\begin{align*}
\left\langle \delta_{b},v_{0}^{k}z^{j}\frac{\mathrm{d}z}{z}\right\rangle_{\mathrm{per}}&=\int_{\varepsilon\left(\sigma_{+}-\frac{1}{2}\sigma_{0}+\sum_{n=1}^{k-b}d_{k-b}(n)\sigma_{0}^{2n}\right)}\frac{1}{\sqrt{z}}\left\langle e_{0}^{b}e_{1}^{k-b},v_{0}^{k}z^{j}\frac{\mathrm{d}z}{z}\right\rangle_{\mathrm{top}}\\
&=(-1)^{k-b}(\pi\sqrt{-1})^{b}\int_{\varepsilon\left(\sigma_{+}-\frac{1}{2}\sigma_{0}+\sum_{n=1}^{k-b}d_{k-b}(n)\sigma_{0}^{2n}\right)}\sqrt{z}(-A_{0})^{b}B_{0}^{k-b}z^{j-1}\mathrm{d}z\\
&=(\pi\sqrt{-1})^{b}(-1)^{k-b}2^{k}\int_{\varepsilon\left(\sigma_{+}-\frac{1}{2}\sigma_{0}+\sum_{n=1}^{k-b}d_{k-b}(n)\sigma_{0}^{2n}\right)}z^{j-1}\sqrt{z}I_{0}(2\sqrt{z})^{b}K_{0}(2\sqrt{z})^{k-b}\mathrm{d}z\\
&=(\pi\sqrt{-1})^{b}(-1)^{k-b}2^{k}\int_{\varepsilon}^{\infty}z^{j-1}\sqrt{z}I_{0}(2\sqrt{z})^{b}K_{0}(2\sqrt{z})^{k-b}\mathrm{d}z\\
&\hspace{20pt}-\frac{1}{2}(\pi\sqrt{-1})^{b}(-1)^{k-b}2^{k}\int_{\varepsilon\sigma_{0}}z^{j-1}\sqrt{z}I_{0}(2\sqrt{z})^{b}K_{0}(2\sqrt{z})^{k-b}\mathrm{d}z\\
&\hspace{20pt}+\sum_{n=1}^{k-b}d_{k-b}(n)(\pi\sqrt{-1})^{b}(-1)^{k-b}2^{k}\int_{\varepsilon\sigma_{0}^{2n}}z^{j-1}\sqrt{z}I_{0}(2\sqrt{z})^{b}K_{0}(2\sqrt{z})^{k-b}\mathrm{d}z.
\end{align*}

Changing the coordinate by $z=\frac{t^{2}}{4}$, the first term becomes
\[
(\pi\sqrt{-1})^{b}(-1)^{k-b}2^{k-2j}\int_{2\sqrt{\varepsilon}}^{\infty}I_0(t)^{b}K_0(t)^{k-b}t^{2j}\mathrm{d}t.
\]
When $\varepsilon\rightarrow 0^{+}$, this term tends to $(\pi\sqrt{-1})^{b}(-1)^{k-b}2^{k-2j}{\mathrm{IKM}}_{k}(b,2j)$. 

For the other two terms, $\int_{\varepsilon\sigma_{0}^{p}}z^{j-1}\sqrt{z}I_{0}(2\sqrt{z})^{b}K_{0}(2\sqrt{z})^{k-b}\mathrm{d}z$ tends to zero as $\varepsilon\rightarrow 0^{+}$ for the following reason. As $s\rightarrow 0^{+}$, we have the asymptotic expansions \eqref{expansion_of_I_0_at_0} and \eqref{expansion_of_K_0_at_0}. Then, as $\varepsilon\rightarrow 0^{+}$ for all $j\geq 0$, we have the estimate
\begin{align*}
\left|\int_{\varepsilon\sigma_{0}^{p}}z^{j-1}\sqrt{z}I_{0}(2\sqrt{z})^{b}K_{0}(2\sqrt{z})^{k-b}\mathrm{d}z\right|&\leq \int_{\varepsilon\sigma_{0}^{p}}\left|z^{j-1}\sqrt{z}I_{0}(2\sqrt{z})^{b}K_{0}(2\sqrt{z})^{k-b}\right|\left|\mathrm{d}z\right|\\
&\leq \int_{0}^{2\pi p}\varepsilon^{j-1}\sqrt{\varepsilon}\left|I_{0}\left(2\sqrt{\varepsilon e^{i\theta}}\right)^{b}K_{0}\left(2\sqrt{\varepsilon e^{i\theta}}\right)^{k-b}\right|\varepsilon{\mathrm{d}}\theta\\
&\leq \varepsilon^{j}\sqrt{\varepsilon}\int_{0}^{2\pi p}\left|\gamma + \log\sqrt{\varepsilon e^{i\theta}}\right|^{k-b}{\mathrm{d}}\theta\\
&=\varepsilon^{j}\sqrt{\varepsilon}\int_{0}^{2\pi p}\left|\gamma + \log\sqrt{\varepsilon }+\frac{1}{2}i\theta\right|^{k-b}{\mathrm{d}}\theta\rightarrow 0.
\qedhere
\end{align*}
\end{proof}

\begin{proposition}\label{compact_support_period}
The compactly supported period pairing of the compactly supported de Rham cohomology $\widetilde{\omega}_{k,j}$ in Definition \ref{compact_support_de_rham_bisis_elements} and moderate decay cycle $\gamma_{a}$ in \eqref{moderate_decay_cycle_gamma_a} is given by
\[
\left\langle\widetilde{\omega}_{k,j},\gamma_{a}\right\rangle_{\mathrm{per,c}}=2^{k-2j}(-1)^{k-a}(\pi\sqrt{-1})^{a}\cdot{\mathrm{IKM}},
\]
where $0\leq a\leq \lfloor k/2\rfloor, 0\leq j\leq k'$
with $j\neq r$ if $k\equiv 2\bmod{4}$,
and
\[ {\mathrm{IKM}}=\begin{cases}
{\mathrm{IKM}}^{\text{reg}}_{k}(a,2j)
& \text{if $4\mid k, a=k/2, r+1\leq j\leq  k^{\prime}$}, \\
{\mathrm{IKM}}_{k}(a,2j)-\gamma_{k,j-k'/2}2^{2j-k'}{\mathrm{IKM}}_{k}(a,k')
& \text{if $4\mid (k+2), 0\leq a\leq k', r+1\leq j\leq k'$}, \\
{\mathrm{IKM}}^{\text{reg}}_{k}(a,2j)-\gamma_{k,j-k'/2}2^{2j-k'}{\mathrm{IKM}}^{\text{reg}}_{k}(a,k')
& \text{if $4\mid (k+2), a=k'+1, r+1\leq j\leq k'$}, \\
{\mathrm{IKM}}_{k}(a,2j)
& \text{otherwise}.
\end{cases}\]
Moreover, when $k=4r+2$, we have
\[
\left\langle\widehat{m}_{2r+1},\gamma_{a}\right\rangle_{\mathrm{per,c}}=\delta_{a,2r+1}(\pi\sqrt{-1})^{a}2^{k}\frac{1}{\binom{k}{k/2}}.
\]

\end{proposition}
\begin{proof}
When $k\equiv 1,3\bmod{4}$. We compute the compactly supported period pairing
\begin{align*}
\left\langle\left(\xi_{j},\eta_{j},\omega_{k,j}\right), \gamma_{a}\right\rangle_{\mathrm{per,c}}&=\int_{\mathbb{R}_{+}}\frac{1}{\sqrt{z}}\left\langle v_{0}^{k},e_{0}^{a}e_{1}^{k-a}\right\rangle_{\mathrm{top}} z^{j}\frac{{\mathrm{d}}z}{z}-\frac{1}{\sqrt{z}}\left\langle -\sum_{c=0}^{k}\binom{k}{c}\frac{\eta_{j,c}}{\sqrt{z}}e_{0}^{k-c}\overline{e}_{1}^{c},e_{0}^{a}e_{1}^{k-a}\right\rangle_{\mathrm{top}}\\
&\hspace{20pt}+\frac{1}{\sqrt{z}}\left\langle \sum_{c=0}^{k}\xi_{j,c}v_{0}^{c}v_{1}^{k-c} ,e_{0}^{a}e_{1}^{k-a}\right\rangle_{\mathrm{top}}\\
&=\int_{\mathbb{R}_{+}}\frac{1}{\sqrt{z}}\left\langle\sum_{c=0}^{k}\binom{k}{c}\left(-A_{0}(z)\right)^{c}B_{0}(z)^{k-c}e_{0}^{k-c}\overline{e}_{1}^{c},e_{0}^{a}e_{1}^{k-a}\right\rangle_{\mathrm{top}}z^{j-1}{\mathrm{d}}z\\
&\hspace{20pt}+(\pi\sqrt{-1})^{a}(-1)^{a}\eta_{j,a}+\frac{1}{\sqrt{z}}\left\langle\sum_{c=0}^{k}\xi_{j,c}v_{0}^{c}v_{1}^{k-c},e_{0}^{a}e_{1}^{k-a}\right\rangle_{\mathrm{top}}\\
&=(-1)^{a}(\pi\sqrt{-1})^{a}\int_{\mathbb{R}_{+}}\left(-A_{0}(z)\right)^{a}B_{0}(z)^{k-a}\sqrt{z}\,z^{j-1}{\mathrm{d}}z\\
&\hspace{20pt}+(-1)^{a}(\pi\sqrt{-1})^{a}\eta_{j,a}+\frac{1}{\sqrt{z}}\left\langle\sum_{c=0}^{k}\xi_{j,c}v_{0}^{c}v_{1}^{k-c},e_{0}^{a}e_{1}^{k-a}\right\rangle_{\mathrm{top}}\\
&=(-1)^{a}(\pi\sqrt{-1})^{a}2^{k-2j}\int_{\mathbb{R}_{+}}I_0(s)^{a}K_0(s)^{k-a}s^{2j}{\mathrm{d}}s+(-1)^{a}(\pi\sqrt{-1})^{a}\eta_{j,a}\\
&\hspace{20pt}+\frac{2}{s}\left\langle\sum_{c=0}^{k}\xi_{j,c}v_{0}^{c}v_{1}^{k-c},e_{0}^{a}e_{1}^{k-a}\right\rangle_{\mathrm{top}}
\end{align*}
where the last equality follows by the change of variable $z=\frac{s^{2}}{4}$. The first term converges by Lemma \ref{regularized_lemma}. Since $k>2a$, by \eqref{expansion_eta_i_a}, the second term tends to zero as $s\rightarrow\infty$. The third term tends to zero as $s\rightarrow 0$ since all $\xi_{j,c}\in\mathbb{Q}\llbracket\frac{s^{2}}{4}\rrbracket$ and the topological pairing gives a factor $\frac{s^{2}}{4}$.

When $k\equiv 0\bmod{4}$, write $k=4r+4$. We compute the compactly supported period pairing
\begin{align*}
\left\langle\left(\xi_{j},\eta_{j},\omega_{k,j}\right), \gamma_{a}\right\rangle_{\mathrm{per,c}}&=\int_{\mathbb{R}_{+}}\frac{1}{\sqrt{z}}\left\langle v_{0}^{k},e_{0}^{a}e_{1}^{k-a}\right\rangle_{\mathrm{top}} z^{j}\frac{{\mathrm{d}}z}{z}-\frac{1}{\sqrt{z}}\left\langle -\sum_{c=0}^{k}\binom{k}{c}\frac{\eta_{j,c}}{\sqrt{z}}e_{0}^{k-c}\overline{e}_{1}^{c},e_{0}^{a}e_{1}^{k-a}\right\rangle_{\mathrm{top}}\\
&\hspace{20pt}+\frac{1}{\sqrt{z}}\left\langle\sum_{c=0}^{k}\xi_{j,c}v_{0}^{c}v_{1}^{k-c},e_{0}^{a}e_{1}^{k-a}\right\rangle_{\mathrm{top}}\\
&=2^{k-2j}(-1)^{a}(\pi\sqrt{-1})^{a}\int_{\mathbb{R}_{+}}I_0(s)^{a}K_0(s)^{k-a}s^{2j}{\mathrm{d}}s\\
&\hspace{20pt}+(-1)^{a}(\pi\sqrt{-1})^{a}\eta_{j,a}
+\frac{2}{s}\left\langle\sum_{c=0}^{k}\xi_{j,c}v_{0}^{c}v_{1}^{k-c},e_{0}^{a}e_{1}^{k-a}\right\rangle_{\mathrm{top}}
\end{align*}
where the last equality is the change of variable $z=\frac{s^{2}}{4}$. The third term tends to zero as $s\rightarrow 0$ since all $\xi_{j,c}\in\mathbb{Q}\llbracket\frac{s^{2}}{4}\rrbracket$ and the topological pairing gives a factor $\frac{s^{2}}{4}$. By the same argument above, when $a=0,1,\cdots,\frac{k-2}{2}=k^{\prime}$, that is, $k>2a$, we have that the first term converges and the second term tends to zero as $s\rightarrow\infty$.

Now, we turn to analyze the case that $a=\frac{k}{2}$. The pairing becomes
\[
\left\langle\left(\xi_{j},\eta_{j},\omega_{k,j}\right), \gamma_{a}\right\rangle_{\mathrm{per,c}}=2^{k-2j}(-\pi\sqrt{-1})^{a}\int_{0}^{s}I_0(t)^{a}K_0(t)^{k-a}s^{2j}{\mathrm{d}}s+(-\pi\sqrt{-1})^{a}\eta_{j,2r+2}
\]
This term converges as $s\rightarrow\infty$ for the following reason:

The singular part of the integral $\left(I_{0}K_{0}\right)^{2r+2}s^{2j}$ is given by \eqref{Bessel_moments_singular_term_4r+4} and $\eta_{j,2r+2}$ has expansion
\[
\eta_{j,2r+2}\sim\frac{2^{2r-2j+1}}{r-j+1/2}s^{2j-2r-1}\cdot G_{i}\sim 2^{2r-2j+1}\sum_{n=0}^{\infty}\frac{2^{2n}\gamma_{k,n}}{r-j+1/2+n}s^{2j-2r-1-2n}
\]
Thus, both of the singular terms cancel.

When $k\equiv 2\bmod{4}$, write $k=4r+2$. Recall from Definition \ref{compact_support_de_rham_bisis_elements}
the elements $\widetilde{\omega}_{k,j}$ and $\widehat{m}_{2r+1}$
in $H^{1}_{{\mathrm{dR,c}}}\big(\mathbb{G}_{\mathrm{m}},\sqrt{z}\Sym^{k}{\mathrm{Kl}}_{2}\big)$.
If we use the convention that $\gamma_{k,p}=0$ whenever $p<0$, we rewrite
\begin{align*}
\widetilde{\omega}_{k,i}&=\left(\xi_{i},\eta_{i},\omega_{k,i}\right)\\
&=\left(\sum_{a=0}^{k}\xi_{i,a}(z)v_{0}^{a}v_{1}^{k-a}-\gamma_{k,i-r}\sum_{a=0}^{k}\xi_{r,a}(z)v_{0}^{a}v_{1}^{k-a},\right.\\
&\hspace{20pt}-\sum_{\substack{a=0\\a\neq k/2}}^{k}\binom{k}{a}\frac{\eta_{i,a}-\gamma_{k,i-r}\eta_{r,a}}{\sqrt{z}}e_{0}^{k-a}\overline{e}_{1}^{a}-\binom{k}{k/2}\frac{\eta_{i,2r+1}}{\sqrt{z}}(e_{0}\overline{e}_{1})^{2r+1},\\
&\hspace{20pt}\left.v_{0}^{k}z^{i}\frac{\mathrm{d}z}{z}-\gamma_{k,i-r}v_{0}^{k}z^{r}\frac{\mathrm{d}z}{z}\right)
\end{align*}

In the pairing $\left\langle\widetilde{\omega}_{k,j},\gamma_{a}\right\rangle_{\mathrm{per,c}}$, the third term \[\frac{2}{s}\left\langle\sum_{c=0}^{k}\xi_{j,c}v_{0}^{c}v_{1}^{k-c},e_{0}^{a}e_{1}^{k-a}\right\rangle_{\mathrm{top}}\]
tends to zero as $s\rightarrow 0$ since all $\xi_{j,c}\in\mathbb{C}\llbracket\frac{s^{2}}{4}\rrbracket$ and the topological pairing gives a factor $\frac{s^{2}}{4}$.
The other two terms are equal to
\begin{align*}
&\left(\int_{\mathbb{R}_{+}}\frac{1}{\sqrt{z}}\left\langle v_{0}^{k},e_{0}^{a}e_{1}^{k-a}\right\rangle_{\mathrm{top}} z^{j}\frac{{\mathrm{d}}z}{z}-\gamma_{k,j-r}\int_{\mathbb{R}_{+}}\frac{1}{\sqrt{z}}\left\langle v_{0}^{k},e_{0}^{a}e_{1}^{k-a}\right\rangle_{\mathrm{top}} z^{r}\frac{{\mathrm{d}}z}{z}\right)\\
&-\frac{1}{\sqrt{z}}\left\langle -\sum_{\substack{b=0\\b\neq k/2}}^{k}\binom{k}{b}\frac{\eta_{j,b}-\gamma_{k,j-r}\eta_{r,b}}{\sqrt{z}}e_{0}^{k-b}\overline{e}_{1}^{b}-\binom{k}{k/2}\frac{\eta_{j,2r+1}}{\sqrt{z}}(e_{0}\overline{e}_{1})^{2r+1}, e_{0}^{a}e_{1}^{k-a}\right\rangle_{\mathrm{top}}\\
&=(-1)^{a}(\pi\sqrt{-1})^{a}\left(2^{k-2j}\int_{\mathbb{R}_{+}}I_0(s)^{a}K_0(s)^{k-a}s^{2j}{\mathrm{d}}s-2^{k-2r}\gamma_{k,j-r}\int_{\mathbb{R}_{+}}I_0(s)^{a}K_0(s)^{k-a}s^{2r}{\mathrm{d}}s\right)\\
&\hspace{20pt}+\frac{1}{z}\left\langle\sum_{\substack{b=0\\b\neq k/2}}^{k}\binom{k}{b}\left(\eta_{j,b}-\gamma_{k,j-r}\eta_{r,b}\right)e_{0}^{k-b}\overline{e}_{1}^{b}+\binom{k}{k/2}\eta_{j,2r+1}(e_{0}\overline{e}_{1})^{2r+1},e_{0}^{a}e_{1}^{k-a}\right\rangle_{\mathrm{top}}
\end{align*}

We analyze the convergence of these terms. When $2a<k$ or $j<r$, the integral $\int_{0}^{t}I_0(s)^{a}K_0(s)^{k-a}s^{2j}{\mathrm{d}}s$ converges as $t\rightarrow\infty$ by Lemma \ref{regularized_lemma}. The second term is equal to
\begin{align*}
&(-1)^{a}(\pi\sqrt{-1})^{a}\left(\eta_{j,a}-\gamma_{k,j-r}\eta_{r,a}\right).
\end{align*}
By the expansion of $\eta_{i,a}$:
\[
\eta_{i,a}\sim\frac{\sqrt{\pi}^{k-2a}}{k-2a}e^{-(k-2a)s}\left(\frac{4}{s^{2}}\right)^{k/4-i}\cdot G_{i,a},
\]
where $G_{i,a}\in 1+\frac{2}{s}\mathbb{Q}\llbracket\frac{2}{s}\rrbracket$, this term tends to $0$ as $s\rightarrow\infty$. 

When $a=\frac{k}{2}$ and $j\geq r$, the integral $\int_{0}^{t}I_0(s)^{a}K_0(s)^{k-a}s^{2j}{\mathrm{d}}s$ has the singular part \eqref{Bessel_moments_singular_term_4r+2}. The second term is equal to
\[ (-1)^{2r+1}(\pi\sqrt{-1})^{2r+1}\eta_{j,2r+1}
=(-1)^{2r+1}(\pi\sqrt{-1})^{2r+1}\frac{1}{r-j}\left(\frac{4}{s^{2}}\right)^{r-j}\cdot H_{i} \]
where $H_{i}\in 1+\frac{4}{s^{2}}\mathbb{Q}\llbracket\frac{4}{s^{2}}\rrbracket$. Thus, the singular part of this term is
\[
(-1)^{2r+1}(\pi\sqrt{-1})^{2r+1}\sum_{n=1}^{j-r}\frac{-\gamma_{k,j-r-n}}{n}\left(\frac{4}{s^{2}}\right)^{-n}.
\]
In consequence, the singular parts cancel.

Finally, for $a=0,1,\cdots,\frac{k}{2}$, we have
\begin{align*}
\left\langle\left(0,\frac{2^{k}}{\sqrt{z}}(e_{0}\overline{e}_{1})^{2r+1},0\right),\gamma_{a}\right\rangle_{\mathrm{per,c}}&=-\frac{2^{k}}{\sqrt{z}}\left\langle\frac{1}{\sqrt{z}}(e_{0}\overline{e}_{1})^{2r+1},e_{0}^{a}e_{1}^{k-a}\right\rangle_{\mathrm{top}}\\
&=\delta_{a,2r+1}(\pi\sqrt{-1})^{a}2^{k}\frac{1}{\binom{k}{k/2}}.
\qedhere
\end{align*}
\end{proof}

\begin{corollary}
The period matrix of the period pairing with respective to the bases $\left\{\delta_b\right\}_{b=0}^{k^{\prime}}$ of $H_{1}^{\mathrm{rd}}$ and $\left\{\omega_{k,j}\right\}_{j=0}^{k^{\prime}}$ of $H^{1}_{\mathrm{dR}}$ is $P=(P_{bj})$, where
\[
P_{bj} = \left\langle \delta_{b},v_{0}^{k}z^{j}\frac{\mathrm{d}z}{z}\right\rangle_{\mathrm{per}}=(\pi\sqrt{-1})^{b}(-1)^{k-b}2^{k-2j}{\mathrm{IKM}}_{k}(b,2j)
\]
for $0\leq b\leq k^{\prime}$ and $0\leq j\leq k^{\prime}$. Moreover, $P$ is invertible.
\end{corollary}

\begin{remark}[determinant of the period matrix]\label{period_matrix_determinant}
In fact, 
\begin{align*}
\det P&=(\pi\sqrt{-1})^{k^{\prime}(k^{\prime}+1)/2}(-1)^{(2k-k^{\prime})(k^{\prime}+1)/2}2^{(k-k^{\prime})(k^{\prime}+1)}\det\left({\mathrm{IKM}}_{k}(b,2j)\right),
\end{align*}
where
\begin{align*}
\det\left({\mathrm{IKM}}_{k}(b,2j)\right)=
\begin{cases}
\det\left(M_{k^{\prime}+1}\right) & {\text{if }}k{\text{ is odd;}}\\
\det\left(N_{k^{\prime}+1}\right) & {\text{if }}k{\text{ is even.}}\\
\end{cases}
\end{align*}
The definition of $M_{r}$ and $N_{r}$ are given in Appendix \ref{section_period_determinant} and their determinants are given in Corollary \ref{period_determinant} explicitly. 
\end{remark}

\subsection{\texorpdfstring{$\mathbb{Q}$}{}-linear and quadratic relations on Bessel moments}

We have now developed all the tools and computations to see the wonderful results in $\mathbb{Q}$-linear and quadratic relations on Bessel moments.

\begin{corollary}\label{linear_relation_Bessel_moments}
For $k=4r+4$,
\begin{align*}
\sum_{j=0}^{r}\binom{k/2}{2j}(-1)^{j}\pi^{2j}{\mathrm{IKM}}_{k}(2j,2i)
= \begin{cases}
(-1)^r\pi^{2r+2}{\mathrm{IKM}}_{k}(2r+2,2i) & \text{if $0\leq i\leq r$}, \\
(-1)^r\pi^{2r+2}{\mathrm{IKM}}_{k}^{\mathrm{reg}}(2r+2,2i)
	& \text{if $r+1\leq i\leq k^{\prime}$}.
\end{cases}
\end{align*}
For $k=4r+2$,
\begin{align}\label{4r+2_algebraic_relation}
\sum_{j=0}^{r}\binom{k/2}{2j}(-1)^{j}\pi^{2j}{\mathrm{IKM}}_{k}(2j,2i)
= \begin{cases}
0 & \text{if $0\leq i\leq r-1$}, \\
\gamma_{k,i-r}2^{2i-2r}
\displaystyle\sum_{j=0}^{r}\binom{k/2}{2j}(-1)^{j}\pi^{2j}{\mathrm{IKM}}_{k}(2j,2r)
	& \text{if $r+1\leq i\leq 2r$}.
\end{cases}
\end{align}
\end{corollary}

\begin{proof}
By Lemma \ref{kernel_cycle}, we know that
\begin{align*}
\sum_{j=0}^{k/4}\binom{k/2}{2j}\gamma_{2j}&=0 {\text{ if }}k\equiv 0\bmod{4};\\
\sum_{j=0}^{\left(k-2\right)/4}\binom{k/2}{2j}\gamma_{2j}&=0 {\text{ if }}k\equiv 2\bmod{4}.
\end{align*}
Then take the pairing with $\widetilde{\omega}_{k,i}$ in the compactly supported de Rham cohomology. Combining with the result of Proposition \ref{compact_support_period}, we obtain the desired algebraic relation.
\end{proof}

\begin{remark}
The above linear algebraic relations
for $i$ in the range $0\leq i\leq r$, under the name {\textit{sum rule identities}}, are previously proved by analytic method in \cite{MR3902501} (see \cite[(1.3)]{MR3902501} for $k\equiv 2\bmod{4}$ and \cite[(1.5)]{MR3902501} for $k\equiv 0\bmod{4}$).
\end{remark}

\begin{corollary}\label{cor_bessel_dim_upper_bound}
For any $k$ and any $0\leq a\leq k^{\prime}$, the dimension of the $\mathbb{Q}$-vector space generated by the Bessel moments has an upper bound:
\[
\dim {\mathrm{span}}_{\mathbb{Q}}\left\{{\mathrm{IKM}}_{k}(a,2j)\mid j\in\left\{0\right\}\cup \mathbb{N}\right\}\leq k^{\prime}+1.
\]
If $k$ is even, the dimension of the $\mathbb{Q}$-vector space generated by the regularized Bessel moments has an upper bound:
\[
\dim {\mathrm{span}}_{\mathbb{Q}}\left\{{\mathrm{IKM}}^{\mathrm{reg}}_{k}(k/2,2j)\mid j\in\left\{0\right\}\cup \mathbb{N}\right\}\leq k^{\prime}+1.
\]
Here when $0\leq j\leq \left\lfloor\frac{k-1}{4}\right\rfloor =r$, we do not need to regularize the Bessel moments, that is, ${\mathrm{IKM}}^{\mathrm{reg}}_{k}(k/2,2j)={\mathrm{IKM}}_{k}(k/2,2j)$ (see Lemma \ref{regularized_lemma}).
\end{corollary}
\begin{proof}
We know that the dimensions of $H^{1}_{\mathrm{dR}}\big(\mathbb{G}_{\mathrm{m}},\sqrt{z}\Sym^{k}{\mathrm{Kl}}_{2}\big)$ and $H^{1}_{\mathrm{dR,c}}\big(\mathbb{G}_{\mathrm{m}},\sqrt{z}\Sym^{k}{\mathrm{Kl}}_{2}\big)$ are $k^{\prime}+1$.

For each integer $s>k'$, since $\left\{v_{0}^{k}z^{j}\frac{{\mathrm{d}}z}{z}\right\}_{j=0,\cdots,k^{\prime}}$ form a basis of $H^{1}_{\mathrm{dR}}\big(\mathbb{G}_{\mathrm{m}},\sqrt{z}\Sym^{k}{\mathrm{Kl}}_{2}\big)$, we may express $v_{0}^{k}z^{s}\frac{{\mathrm{d}}z}{z}$ as the $\mathbb{Q}$-linear combination of the basis. Then after we take the period pairing between $v_{0}^{k}z^{s}\frac{{\mathrm{d}}z}{z}$ and the rapid decay cycle $\delta_{a}$ (see Proposition \ref{period_pairing}), the $\mathbb{Q}$-linear relation becomes a $\mathbb{Q}$-linear relation for the Bessel moments
\[
\left\{{\mathrm{IKM}}_{k}(a,2j)\mid j=0,\cdots,k^{\prime}\right\}\cup \left\{{\mathrm{IKM}}_{k}(a,2s)\right\}.
\]

If $k$ is even, similarly, when $s>k^{\prime}$, express $\widetilde{\omega}_{k,s}\in H^{1}_{\mathrm{dR,c}}\big(\mathbb{G}_{\mathrm{m}},\sqrt{z}\Sym^{k}{\mathrm{Kl}}_{2}\big)$ as the $\mathbb{Q}$-linear combination of the basis describe in Corollary \ref{de_rham_basis}. Then taking the compactly supported period pairing (see Proposition \ref{compact_support_period}), the $\mathbb{Q}$-linear equivalence become the $\mathbb{Q}$-linear relation for the regularized Bessel moments
\[
\left\{{\mathrm{IKM}}_{k}(a,2j)\mid j=0,\cdots,r-1\right\}\cup\left\{{\mathrm{IKM}}^{\mathrm{reg}}_{k}(a,2j)\mid j=r,\cdots,k^{\prime}\right\}\cup \left\{{\mathrm{IKM}}^{\mathrm{reg}}_{k}(a,2s)\right\}.
\qedhere
\]
\end{proof}

\begin{remark}
In \cite{MR2433887}, Borwein and Salvy provide a recurrence to find out the $\mathbb{Q}$-linear combination for Bessel moments by analyzing the symmetric power of the modified Bessel differential operator. Moreover, Zhou proves a similar result in \cite{MR4437425} for the $\mathbb{Q}$-linear dependence for Bessel moments ${\mathrm{IKM}}_{k}(a,2j-1)$. Our result is parallel to Zhou's result.
\end{remark}

\begin{proposition}\label{quadratic_relation}
With respect to the bases of $H^{{\mathrm{rd}}}_{1}$, $H^{{\mathrm{mod}}}_{1}$, $H_{\mathrm{dR}}^{1}$, and $H_{\mathrm{dR,c}}^{1}$ described in Corollaries \ref{de_rham_basis}, \ref{Betti_basis_assume_dim}, we form the pairing matrices:
\begin{enumerate}
\item $B$, the Betti intersection pairing matrix between $H^{{\mathrm{rd}}}_{1}$ and $H^{{\mathrm{mod}}}_{1}$ in Proposition \ref{Betti_intersection_pairing}.
\item $D$, the Poincar\'{e} pairing matrix between $H_{\mathrm{dR,c}}^{1}$ and $H_{\mathrm{dR}}^{1}$ in Proposition \ref{de_rham_paiting_matrix}.
\item $P$, the period pairing matrix between $H^{{\mathrm{rd}}}_{1}$ and $H_{\mathrm{dR}}^{1}$ in Proposition \ref{period_pairing}.
\item $P_{\mathrm{c}}$, the period pairing matrix between $H_{\mathrm{dR,c}}^{1}$ and $H^{{\mathrm{mod}}}_{1}$ in Proposition \ref{compact_support_period}.
\item $B_{\mathrm{mid}}$, the Betti pairing matrix on $H^{\mathrm{mid}}_{1}$.
\item $D_{\mathrm{mid}}$, the Poincar\'{e} pairing matrix on $H_{\mathrm{mid}}^{1}$.
\item $P_{\mathrm{mid}}$, the period pairing matrix between $H^{\mathrm{mid}}_{1}$ and $H_{\mathrm{mid}}^{1}$.\footnote{$B,D,P,P_{\mathrm{c}}$ are square matrices of size $k^{\prime}+1$ and that $B_{\mathrm{mid}},D_{\mathrm{mid}},P_{\mathrm{mid}}$ are of size $k^{\prime}+1-\delta_{4\mathbb{Z}+2,k}$. When $k\equiv 0,1,3\bmod{4}$, we have $B=B_{\mathrm{mid}}$, $D=D_{\mathrm{mid}}$, and $P_{\mathrm{mid}}=P=P_{\mathrm{c}}^{t}$.}
\end{enumerate}

Then we have the algebraic quadratic relations
\begin{align}
PD^{-1}P_{\mathrm{c}}&=(-1)^{k}(2\pi\sqrt{-1})^{k+1}B,\label{quadratic_matrix_relation}\\
P_{\mathrm{mid}}D^{-1}_{\mathrm{mid}}P_{\mathrm{mid}}^{t}&=(-1)^{k}(2\pi\sqrt{-1})^{k+1}B_{\mathrm{mid}}.\label{quadratic_matrix_relation_middle}
\end{align}

\end{proposition}
\begin{proof}
This quadratic relation is a general phenomenon on periods of meromorphic flat connection on complex manifolds. We refer to \cite[Corollaries 2.14, 2.16]{MR4601771} for more details.
\end{proof}

From this proposition, when $k\equiv 0,1,3\bmod{4}$, we see the Bessel moments have quadratic relation given by \eqref{quadratic_matrix_relation}. On the other hand, when $k\equiv 2\bmod{4}$, the relation involves some combination of Bessel moments and regularized Bessel moments in the matrix $P_{\mathrm{c}}$. In the following discussion, we provide another expression of this relation, and we will see the pure quadratic relation involving only Bessel moments.

When $k\equiv 2\bmod{4}$, write $k=4r+2$ and define two $(k^{\prime}+1)\times k^{\prime}$ matrices with rational coefficients:
\[ R_{k}=
\left(\begin{array}{cccc}
I_{r} & & 0 &\\ 
0 & -\gamma_{k,1} & \cdots & -\gamma_{k,k^{\prime}-r}\\
0 & & I_{r} &
\end{array}\right),
\quad
L_{k}=\left(\begin{array}{ccccccc}
0 & -\binom{k/2}{2} & 0 & -\binom{k/2}{4} & \cdots & 0 & -\binom{k/2}{k^{\prime}}\\ 
 & & & I_{k^{\prime}} &
\end{array}\right). \]
By the linear relations \eqref{4r+2_algebraic_relation} in Corollary \ref{linear_relation_Bessel_moments}, we have
\[
PR_{k}=L_{k}P_{\mathrm{mid}}.
\]
Also, $P_{\mathrm{mid}}$ is obtained by deleting the first row of $L_{k}P_{\mathrm{mid}}$. Set $\widetilde{B}=L_{k}B_{\mathrm{mid}}L_{k}^{t}$ and $\widetilde{D}=R_{k}D^{-1}_{\mathrm{mid}}R_{k}^{t}$ which are square matrices of size $k^{\prime}+1$ with rational coefficients. Then $B_{\mathrm{mid}}$ is obtained by deleting the first row and column from $\widetilde{B}$. Therefore, the quadratic relation \eqref{quadratic_matrix_relation_middle} (involving linear combinations of Bessel moments) now becomes
\[
P\widetilde{D}P^{t}=(-1)^{k}(2\pi\sqrt{-1})^{k+1}\widetilde{B}
\]
(involving pure Bessel moments).

\begin{remark}
The matrices $\widetilde{B}$ and $\widetilde{D}$ in the above expression are singular because of the linear relations \eqref{4r+2_algebraic_relation} in Corollary \ref{linear_relation_Bessel_moments}. This expression is equivalent to the middle part quadratic relation \eqref{quadratic_matrix_relation_middle} together with linear relations \eqref{4r+2_algebraic_relation}.
\end{remark}

\begin{proposition}
When $k=4r+2$, the middle part period matrix is a $k^{\prime}\times k^{\prime}$ matrix given by
\[
P_{\mathrm{mid}}=\left(\left\langle\delta_{b},\omega_{k,i}\right\rangle_{\mathrm{per}}\right)_{b=1,\cdots,k^{\prime},\ i=0,\cdots,\hat{r},\cdots,k^{\prime}}.
\]
The determinant of this matrix $P_{\mathrm{mid}}$ is given by
\[
\det P_{\mathrm{mid}}=\pi^{r(k+1)}\sqrt{-1}^{r(k^{\prime}-1)}\frac{2^{r(2r+1)}}{r!}\prod_{a=1}^{k^{\prime}}\frac{(2a+1)^{k^{\prime}+1-a}}{(a+1)^{a+1}}.
\]
\end{proposition}
\begin{proof}
The matrix $P_{\mathrm{mid}}$ appears in the upper left of the compactly supported period pairing matrix $P_{\mathrm{c}}$. Just take determinant on \eqref{quadratic_matrix_relation} and then use the results of Propositions \ref{de_rham_paiting_matrix}, \ref{Betti_determinant}, and Remark \ref{period_matrix_determinant}.
\end{proof}

\begin{appendices}
\section{The Bessel operator and determinants of Bessel moments}\label{section_bessel_operator_determinants}

\subsection{Symmetric power of the modified Bessel differential operator}\label{section_symmetric_power_of_differential_operator}
Consider the Weyl algebra $\mathbb{Q}\langle t,\partial_{t}\rangle$ consisting of ordinary differential operators. Write $\theta=t\partial_{t}$. The modified Bessel differential operator is an element in the subalgebra
$\mathbb{Q}\langle t^{2},\theta\rangle$
given by $L_{2}=\theta^{2}-t^{2}$.
The corresponding solutions are the modified Bessel functions $I_{0}(t)$ and $K_{0}(t)$.
The $n$-th symmetric power $L_{n+1}\in\mathbb{Q}\left\langle\theta,t^{2}\right\rangle$
of $L_{2}$ has order $n+1$ and the corresponding solutions are $I_{0}^{a}(t)K_{0}^{n-a}(t)$ for $0\leq a\leq n$.
By \cite{MR1809982, MR2433887},
the operator $L_{n+1}=L_{n+1,n}$ can be obtained by the recurrence relation as follows:
\begin{align}\label{symmetric_of_bessel_recurrence}
\begin{split}
L_{0,n}&=1, \\
L_{1,n}&=\theta, \\
L_{k+1,n}&=\theta L_{k,n}-t^{2}k\left(n+1-k\right)L_{k-1,n},\  1\leq k\leq n .
\end{split}
\end{align}
Here we provide two more concrete results about the operator $L_{n+1}$.

Put the degree on $\mathbb{Q}\langle t,\theta\rangle$ as $\deg t=\deg \theta=1$.
The associated graded ring $\operatorname{gr}\mathbb{Q}\langle t,\theta\rangle = \mathbb{Q}[\overline{t},\overline{\theta}]$
is a polynomial ring
where $\overline{t}$ and $\overline{\theta}$ are the images of $t$ and $\theta$, respectively.

\begin{proposition}\label{leading-term}
The image of $L_{n+1}$ in $\mathbb{Q}[\overline{t},\overline{\theta}]$ is the polynomial
\begin{equation}\label{result}
\overline{L}_{n+1}(\overline{t},\overline{\theta}) = \begin{cases}
\prod_{i=1}^r \left(\overline{\theta}^2 - (2i-1)^2\overline{t}^2\right)
& \text{if $n+1 = 2r$ is even}, \\
\overline{\theta}\prod_{i=1}^r \left(\overline{\theta}^2 - (2i)^2\overline{t}^2\right)
& \text{if $n+1 = 2r+1$ is odd}.
\end{cases}
\end{equation}
\end{proposition}

\begin{proof}
Taking the images in $\mathbb{Q}[\overline{t},\overline{\theta}]$
of the relation \eqref{symmetric_of_bessel_recurrence},
we obtain $\overline{L}_{n+1}=\overline{L}_{n+1,n}$
satisfying
\begin{align*}
\overline{L}_{0,n}&=1,
\quad
\overline{L}_{1,n}=\overline{\theta}, \\
\overline{L}_{k+1,n}&=\overline{\theta}\,\overline{L}_{k,n}-t^{2}k\left(n+1-k\right)\overline{L}_{k-1,n},\  1\leq k\leq n .
\end{align*}

The formula \eqref{result} is then a consequence of the following combinatorics lemma.
\end{proof}

\begin{lemma}\label{keylemma}
For any $m\in\mathbb{N}$,
set the recurrence for $\lambda_{n,m}(x), n\in\mathbb{N}\cup\{0\}$,
\begin{align*}
\lambda_{0,m}&=1,
\quad
\lambda_{1,m}=x, \\
\lambda_{k+1,m}&=x\lambda_{k,m}-k\left(m+1-k\right)\lambda_{k-1,m},\quad k\geq 1.
\end{align*}
Then we have
\begin{align}\label{modresult}
\lambda_{m+1,m}(x)=\begin{cases}\prod_{i=1}^{r}\left(x^{2}-(2i-1)^{2}\right), & m+1=2r,\\ x\prod_{i=1}^{r}\left(x^{2}-(2i)^{2}\right), & m+1=2r+1.\end{cases}
\end{align}
\end{lemma}
\begin{proof}
Notice that $\lambda_{i,m}$ is a monic integral polynomial of degree $i$ for any $m$.
Consider the formal generating function{\footnote{This generating function satisfies the differential equation $-y^{4}f^{\prime\prime}(y)-(2y^{3}-my^{3})f^{\prime}(y)+(1-xy+my^{2})f(y)=1.$}}:
\[
f_{m,x}(y)=\sum_{i=0}^{\infty}\lambda_{i,m}(x)y^{i}.
\]
An induction on $i$ immediately yields the relation
$f_{m,x-1}\left(y\right)+f_{m,x+1}\left(y\right)=2f_{m-1,x}\left(y\right)$
for any $m$ and $x$.\footnote{Equality also holds when viewed as the solution of the corresponding differential equations.} In other words,
$\lambda_{i,m}(x-1)+\lambda_{i,m}(x+1)=2\lambda_{i,m-1}(x)$
for all $i$.
Therefore we obtain
\[ \lambda_{m+1,m}(x-1)+\lambda_{m+1,m}(x+1)=2\lambda_{m+1,m-1}(x)=2x\lambda_{m,m-1}(x) \]
by the recurrence.
Thus, since $\lambda_{m+1,m}(x)$ is a monic polynomial of degree $m+1$, it is uniquely determined by the above functional equation when the polynomial $\lambda_{m,m-1}(x)$ is given. Hence, by the induction, it suffices to show that
\begin{align*}
\prod_{i=1}^{r}\left((x-1)^{2}-(2i-1)^{2}\right)+\prod_{i=1}^{r}\left((x+1)^{2}-(2i-1)^{2}\right)&=2x^{2}\prod_{i=1}^{r-1}\left(x^{2}-(2i)^{2}\right);\\
(x-1)\prod_{i=1}^{r}\left((x-1)^{2}-(2i)^{2}\right)+(x+1)\prod_{i=1}^{r}\left((x+1)^{2}-(2i)^{2}\right)&=2x\prod_{i=1}^{r}\left(x^{2}-(2i-1)^{2}\right),
\end{align*}
which are straightforward to verify.
\end{proof}

\begin{proposition}\label{dimennsion_parameter}
Write $L_{n+1}$ into the form $\sum_{i} t^{i}P_{i}\left(\theta\right)$, where $P_{i}(x)\in\mathbb{Q}\left[x\right]$. Define the integers $a,b$ by
\[
a=\max\left\{i\mid P_{i}\neq 0\right\};\quad b=\min\left\{i\mid P_{i}\neq 0\right\}.
\]
Then we have $a=2\left\lfloor\frac{n+1}{2}\right\rfloor$ and $b=0$.
\end{proposition}

\begin{proof}
By the recurrence \eqref{symmetric_of_bessel_recurrence}, if we set $\deg t=1$ and $\deg \theta=0$, we easily see that $L_{j}$ has degree $2\left\lfloor\frac{j}{2}\right\rfloor$ by the interchanging relation $\theta t=t+t\theta$. Thus, we have $a=2\left\lfloor\frac{n+1}{2}\right\rfloor$. On the other hand, if we set $\deg t=0$ and $\deg \theta=1$, we see that the leading term of $L_{k+1}$ is given by $\theta^{k+1}$. Therefore, we conclude that $b=0$ by Proposition \ref{leading-term}.
\end{proof}

\subsection{Two-scale Bessel moments}\label{section_period_determinant}

From now on, we take for granted the properties of modified Bessel functions $I_0(t), K_0(t)$ in the treatise \cite{MR1349110}.

Recall the Bessel moments ${\mathrm{IKM}}_{k}(a,b)$ given in section \ref{section_Bessel_moments}. For $r\in\mathbb{Z}_{\geq 1}$,
define the two $r\times r$ matrices
\[ M_r = \begin{pmatrix}
{\mathrm{IKM}}_{2r-1}(i-1,2j-2)
\end{pmatrix}_{1\leq i,j\leq r},
\quad
N_r = \begin{pmatrix}
{\mathrm{IKM}}_{2r}(i-1,2j-2)
\end{pmatrix}_{1\leq i,j\leq r}. \]
We aim to determine the two scalars $\det M_r$, $\det N_r$
adapting the inductive methods explored by Zhou \cite{MR3824579}.

For the initial values,
we have \cite[\S 13.21, Eq.(8)]{MR1349110}
\begin{equation}\label{eq:M_1}
M_1 = \int_0^\infty K_0(t)\,{\mathrm{d}} t = \frac{\pi}{2},
\end{equation}
and, by \cite[\S 13.72]{MR1349110},
one has
\begin{align*}
N_1 = \int_0^\infty K_0^2(t)\,{\mathrm{d}} t
&= \frac{1}{2}\int_0^\infty\int_{-\infty}^\infty\int_{-\infty}^\infty
	e^{-2t\cosh x\cosh y} \,{\mathrm{d}} x{\mathrm{d}} y{\mathrm{d}} t \\
&= \frac{1}{4}\int_{-\infty}^\infty\int_{-\infty}^\infty
	\frac{{\mathrm{d}} x{\mathrm{d}} y}{\cosh x\cosh y} \\
&= \frac{\pi^2}{4}.
\end{align*}

\medskip
For $r\in\mathbb{Z}_{\geq 0}$,
let $\omega_{2r+1}(x)$ be the Wronskian
of the $(2r+1)$ functions $f_i(x)$
\begin{equation}\label{eq:f_i}
f_i(x) = \begin{cases}
\int_0^\infty I_0(xt)I^{i-1}_0(t)K_0^{2r-i+1}(t) \,{\mathrm{d}} t, & 1\leq i\leq r, \\
\int_0^\infty K_0(xt)I^{i-r-1}_0(t)K_0^{3r-i+1}(t) \,{\mathrm{d}} t, & r < i \leq 2r+1.
\end{cases}
\end{equation}
The functions $f_i$ are well-defined and analytic on the interval $(0,2)$
and hence so is $\omega_{2r+1}$.
In particular,
$\omega_1(x) = \int_0^\infty K_0(xt) \,{\mathrm{d}} t = \frac{\pi}{2x}$
by \eqref{eq:M_1}.

For $r\in\mathbb{Z}_{\geq 1}$,
let $\omega_{2r}(x)$ be the Wronskian of the $2r$ functions $g_i(x)$
where
\[ g_i(x) = \begin{cases}
\int_0^\infty I_0(xt)I^{i-1}_0(t)K_0^{2r-i}(t) \,{\mathrm{d}} t, & 1\leq i\leq r, \\
\int_0^\infty K_0(xt)I^{i-r-1}_0(t)K_0^{3r-i}(t) \,{\mathrm{d}} t, & r < i \leq 2r.
\end{cases}\]
All entries in the Wronskian matrix
are well-defined analytic functions on the interval $(0,1)$
and so is $\omega_{2r}(x)$.

\begin{proposition}\label{prop:joint_pmatrix_minus}The determinant $\omega_{k}(x)$ and its evaluation at $x=1$ are given by the following formulae:
\begin{enumerate}
\item
For $r\in\mathbb{Z}_{\geq 1}$,
\begin{align*}
\omega_{2r+1}(x) &= \frac{(-1)^{\frac{r(r+1)}{2}}}{2}
\left[ \frac{1}{x^2}
\prod_{i=1}^r \frac{(2i)^2}{(2i)^2-x^2}\right]^{\frac{2r+1}{2}}
\Gamma\Big(\frac{r+1}{2}\Big)^2 \big(\det N_r\big)^2, \\
\omega_{2r+1}(1) &= (-1)^{r(r+1)/2} \det M_r \cdot \det M_{r+1}.
\end{align*}
\item
For $r\in\mathbb{Z}_{\geq 2}$,
\begin{align*}
\omega_{2r}(x) &= (-1)^{\frac{r(r+1)}{2}}
\left[ \frac{1}{x}
\prod_{i=1}^r \frac{(2i-1)^2}{(2i-1)^2-x^2}\right]^r
\big(\det M_r\big)^2, \\
\lim_{x\to 1^-} 2^r(1-x)^r\omega_{2r}(x) &=
(-1)^{r(r+1)/2} (r-1)! \det N_{r-1} \cdot \det N_r.
\end{align*}
\end{enumerate}
\end{proposition}

The above proposition leads to the recursive formulae
\begin{align}
\nonumber
\det M_r \cdot \det M_{r+1}
&= \frac{1}{2}
\left[\frac{2^rr!\sqrt{2r+1}}{(2r+1)!!}\right]^{2r+1}
\Gamma\Big(\frac{r+1}{2}\Big)^2 \big(\det N_r\big)^2
& (r\geq 1), \\
\label{eq:N-M}
\det N_{r-1} \cdot \det N_r
&= \frac{2^r}{(r-1)!}
\left[\frac{(2r-1)!!\sqrt{2r}}{2^rr!}\right]^{2r}\big(\det M_r\big)^2
& (r\geq 2).
\end{align}
With the initial data $M_1 = \frac{\pi}{2}, N_1 = \frac{\pi^2}{4}$ and the relation
\[
\Gamma\left(\frac{r}{2}\right)\Gamma\left(\frac{r+1}{2}\right)=\frac{\left(r-1\right)!}{2^{r-1}}\sqrt{\pi},
\]
one immediately obtains the following results by induction.

\begin{corollary}\label{period_determinant}
For positive integers $r$, we have
\begin{align*}
\det M_{r}&=\sqrt{\pi}^{r(r+1)}\sqrt{2}^{r(r-3)}\prod_{a=1}^{r-1}\frac{a^{r-a}}{\sqrt{2a+1}^{2a+1}}, \\
\det N_{r}&=\frac{1}{\Gamma\left(\frac{r+1}{2}\right)}\frac{\sqrt{\pi}^{(r+1)^{2}}}{\sqrt{2}^{r(r+3)}}\prod_{a=1}^{r-1}\frac{(2a+1)^{r-a}}{(a+1)^{a+1}}.
\end{align*}
In particular,
the two scalars $\sqrt{(2r-1)!!}\,\pi^{-m_r}\det M_r$
and $\pi^{-n_r}\det N_r$
are positive rational numbers, where $m_{r}=\frac{r(r+1)}{2}$ and $n_{r}=\left\lfloor\frac{(r+1)^{2}}{2}\right\rfloor$.
\end{corollary}

\subsection{The Vanhove operators}
The adjoint $L_{n+1}^*$ of $L_{n+1}$ is derived under the convolution
$(t,\partial_t) \mapsto (t,-\partial_t)$
(so $\theta \mapsto -(\theta+1)$)
and hence the leading term
of the signed adjoint $\Lambda_{n+1} = (-1)^{n+1}L_{n+1}^*$
equals $\overline{L}_{n+1}(\overline{\theta},\overline{t})$ by Proposition \ref{leading-term}.
For $F(xt) = I_0(xt), K_0(xt)$
and $G(t) = I_0^a(t)K_0^{n-a}(t)$,
we have, by integration by parts,
\[ \int_0^\infty (\Lambda_{n+1}F(xt)) G(t) \,{\mathrm{d}} t
= (-1)^{n+1}\int_0^\infty F(xt) (L_{n+1}G(t)) \,{\mathrm{d}} t = 0. \]

The \textit{Vanhove operator}
$V_{n+1} \in \mathbb{Q}\left\langle \partial_x,x^{\pm 1} \right\rangle$
is of order $(n+1)$ such that $V_{n+1}F(xt) = \Lambda_{n+1}F(xt)$.
So one has $V_{n+1}f_i = 0$ for $f_i(x)$ in \eqref{eq:f_i}
and consequently $\omega_{n+1}(x)$
satisfies a first order linear differential equation (See \eqref{odeequ} below).

\begin{lemma}\label{lemma:Vanhove_op}
Let $\lambda_{n+1}(x) = \overline{L}_{n+1}(1,x^{-1}) \in \mathbb{Q}[x^{-1}]$
of order $2\left\lfloor\frac{n+1}{2}\right\rfloor$ with respect to $x^{-1}$.
Let $\theta_x = x\partial_x$.
One has
\begin{align*}
V_{n+1}
&= \lambda_{n+1}(x)\theta_x^{n+1}
+ (n+1)\left[ \lambda_{n+1}(x) + \frac{x\lambda_{n+1}'(x)}{2} \right]\theta_x^n
+ \delta_1 \\
&= x^{n+1}\lambda_{n+1}(x)\partial_x^{n+1}
+ \frac{n+1}{2}x^n\left[ (n+2)\lambda_{n+1}(x) + x\lambda_{n+1}'(x) \right]\partial_x^n
+ \delta_2
\end{align*}
where $\delta_1,\delta_2$ are of order at most $(n-1)$
with respect to $\partial_x$ in $\mathbb{Q}\langle\partial_x,x^{\pm 1}\rangle$.
\end{lemma}

\begin{proof}

By Vanhove \cite{MR3330287},
there exists $\widetilde{L}_{n-1} \in \mathbb{Q}\langle\partial_x,x^{\pm 1}\rangle$ of order $(n-1)$
such that
\[ t\widetilde{L}_{n-1}F(xt) = \Lambda_{n+1}\frac{F(xt)}{t}. \]
The operator $\widetilde{L}_{n-1}$ is of the form
(\cite[Eq.\,(4.29)]{MR3824579})
\[ \widetilde{L}_{n-1} = x^2\lambda(x)\theta_x^{n-1}
+ x^2\left[ 2(n-1)\lambda(x) + \frac{n-1}{2}x\lambda'(x) \right]\theta_x^{n-2}
+ \widetilde\delta \]
where $\widetilde\delta$ is of order at most $(n-3)$
with respect to $\partial_x$ in $\mathbb{Q}\langle\partial_x,x^{\pm 1}\rangle$\footnote{Comparing $\widetilde{L}_{n-1}(\theta_x)$
with Zhou's Vanhove operator $\widetilde{L}_{n-1}(\theta_u)$,
we set his variable $u = x^2$
and multiply $\widetilde{L}_{n-1}(\theta_u)$ by $2^{n-1}$.}.

Set
\[ \Delta_n(\theta_t)
= \Lambda_{n+1}(\theta_t) - \Lambda_{n+1}(\theta_t -1). \]
Since
$\theta_t \frac{1}{t} = \frac{1}{t}(\theta_t -1)$ in $\mathbb{Q}\left\langle\partial_t,t\right\rangle$,
we have
\begin{align*}
\Lambda_{n+1}(\theta_t)F(xt)
&= t\Lambda_{n+1}(\theta_t)\frac{F(xt)}{t} + \Delta_n(\theta_t)F(xt) \\
&= \Big[ t^2\widetilde{L}_{n-1}(\theta_x) + \Delta_n(\theta_t) \Big]F(xt).
\end{align*}
Since
$t^2F(xt) = \frac{1}{x^2}\theta_x^2F(xt)$,
we have
\begin{align*}
t^2\widetilde{L}_{n-1}(\theta_x)F(xt)
&= \widetilde{L}_{n-1}(\theta_x)\frac{1}{x^2}\theta_x^2F(xt) \\
&= \frac{1}{x^2}\widetilde{L}_{n-1}(\theta_x-2)\theta_x^2F(xt)
\end{align*}
and the differential operator reads
\[ \lambda(x)\theta_x^{n+1} + \frac{n-1}{2}x\lambda'(x)\theta_x^n + \delta_3 \]
where $\delta_3$ is of order at most $(n-1)$
with respect to $\partial_x$ in $\mathbb{Q}\langle\partial_x,x^{\pm 1}\rangle$.

On the other hand, since
$\theta_tF(xt) = \theta_xF(xt)$ and by Proposition \ref{leading-term},
we have
\[ \Delta_n(\theta_t)F(xt) =\left[\Lambda_{n+1}\left(\theta\right)-\Lambda_{n+1}\left(\theta-1\right)\right]F\left(xt\right)=
\left[ \left( (n+1)\lambda(x) + x\lambda'(x)\right)\theta_x^n
+ \delta_4 \right] F(xt) \]
where $\delta_4$ is of order at most $(n-1)$
with respect to $\partial_x$ in $\mathbb{Q}\langle\partial_x,x^{-1}\rangle$.
Therefore the leading two terms of $V_{n+1}$ are determined.
\end{proof}

\subsubsection*{Rationality of $\omega_{n+1}(x)$}
Lemma \ref{lemma:Vanhove_op} yields
\begin{equation}\label{odeequ}
\omega_{n+1}'(x) = -\frac{n+1}{2x}
\left[ (n+2) + \frac{x\lambda_{n+1}'(x)}{\lambda_{n+1}(x)} \right] \omega_{n+1}(x).
\end{equation}
Since $\omega_{n+1}(x)$ takes real values on $(0,1)$,
one obtains
\[ \omega_{n+1}(x) = C_{n+1}
\left[ (-1)^{\left\lfloor\frac{n+1}{2}\right\rfloor}x^{n+2}\lambda_{n+1}(x) \right]^{-\frac{n+1}{2}} \]
for some real constant $C_{n+1}$
for each $n\in\mathbb{Z}_{\geq 0}$.
We shall determine $C_{n+1}$
by investigating the limiting behavior of $\omega_{n+1}(x)$
as $x \to 0^+$.

\subsection{Singularities of $\omega_{n+1}(x)$}

For $F(xt) = I_0(xt)$ or $K_0(xt)$,
we have
\[ \partial_xF(xt) = tF'(xt),
\quad
\partial_x^2F(xt) = -\frac{t}{x} F'(xt) + t^2F(xt). \]
So $\omega_{n+1}(x)$ coincides with the determinant of the matrix
$\Omega_{n+1}(x)$ of size $(n+1)$ whose $(i,j)$-entry is
\[\begin{cases}
\int_0^\infty I_0(xt) I_0^{j-1}(t) K_0^{n-j+1}(t) t^{i-1} \,{\mathrm{d}} t,
& 1\leq j\leq \left\lfloor\frac{n+1}{2}\right\rfloor, i =1,3,\cdots, 2\left\lfloor\frac{n}{2}\right\rfloor +1, \\
\int_0^\infty tI_0'(xt) I_0^{j-1}(t) K_0^{n-j+1}(t) t^{i-2} \,{\mathrm{d}} t,
& 1\leq j\leq \left\lfloor\frac{n+1}{2}\right\rfloor, i =2,4,\cdots, 2\left\lfloor\frac{n+1}{2}\right\rfloor, \\
\int_0^\infty K_0(xt) I_0^{j-r-1}(t) K_0^{n-j+r+1}(t) t^{i-1} \,{\mathrm{d}} t,
& \left\lfloor\frac{n+1}{2}\right\rfloor< j\leq n+1, i =1,3,\cdots, 2\left\lfloor\frac{n}{2}\right\rfloor+1, \\
\int_0^\infty tK_0'(xt) I_0^{j-r-1}(t) K_0^{n-j+r+1}(t) t^{i-2} \,{\mathrm{d}} t,
& \left\lfloor\frac{n+1}{2}\right\rfloor< j\leq n+1, i =2,4,\cdots, 2\left\lfloor\frac{n+1}{2}\right\rfloor.
\end{cases} \]

\subsubsection*{Properties of $I_0(t)$ and $K_0(t)$}
We collect some properties of the modified Bessel functions
$I_0(t)$ and $K_0(t)$
in order to obtain information of $\omega_{n+1}(x)$
as $x\to 0^+, 1^-$.

The function $I_0(t)$ is entire and even;
it is real and increasing on the half line $[1,\infty)$.
The function $K_0(t)$ has a logarithmic pole at $x=0$;
it is real and decreasing on $(0,\infty)$.
On the half plane $\Re(t)>0$,
we have the asymptotic approximations
\[ I_0(t) = \frac{e^t}{\sqrt{2\pi t}}\left( 1+O\Big(\frac{1}{t}\Big) \right),
\quad
K_0(t) = \sqrt{\frac{\pi}{2t}}e^{-t}\left( 1+O\Big(\frac{1}{t}\Big) \right) \]
as $t\to\infty$.
In particular,
for a positive integer $a$,
\[ [I_0(t)K_0(t)]^a - \frac{1}{(2t)^a} = O\Big(\frac{1}{t^{a+1}}\Big) \]
as $t \to \infty$ along the real line.
One has the boundedness
\[ \sup_{t>0}\frac{|tK'_0(t) + 1|}{t(1+|\log t|)} < \infty. \]
For $c\in\mathbb{Z}_{\geq 0}$,
one has the evaluation \cite[\S 13.21, Eq.(8)]{MR1349110}
\[ \int_0^\infty K_0(t)t^c\,{\mathrm{d}} t
= 2^{c-1}\Gamma\Big(\frac{c+1}{2}\Big)^2. \]

\subsubsection*{Integrations}
With the data collected above,
we list some consequences
for the integrals that appear in the entries of the matrix $\Omega_{n+1}(x)$.

For $0\leq a<b$ and $c\geq 0$,
one obtains
\begin{align}
\label{eq:simple_Kxt}
\int_0^\infty K_0(xt)I_0^a(t)K_0^b(t)t^c \,{\mathrm{d}} t
&= O(\log x), \\
\label{eq:simple_I'xt}
\int_0^\infty I_0'(xt)I_0^a(t)K_0^b(t)t^c \,{\mathrm{d}} t
&= O(x)
\end{align}
and
\begin{align}
\nonumber
\int_0^\infty tK'_0(xt)I_0^a(t)K_0^b(t)t^c \,{\mathrm{d}} t
&= \frac{-1}{x}\left[ \int_0^\infty I_0^a(t)K_0^b(t)t^c \,{\mathrm{d}} t
 -\int_0^\infty \big(xtK_0'(xt)+1\big) I_0^a(t)K_0^b(t)t^c \,{\mathrm{d}} t \right] \\
\label{eq:simple_K'xt}
&= \frac{-1}{x}\int_0^\infty I_0^a(t)K_0^b(t)t^c \,{\mathrm{d}} t + O(\log x)
\end{align}
as $x \to 0^+$.
For $0\leq c\leq a$ and as $x\to 0^+$, we thus have
\begin{align}
\label{eq:KxtIK_1}
\int_0^\infty K_0(xt)I_0^a(t)K_0^a(t)t^c \,{\mathrm{d}} t
&= O\Big(\int_0^\infty K_0(xt) \,{\mathrm{d}} t\Big)
= O\Big(\frac{1}{x}\Big), \\
\label{eq:K'xtIK_1}
\int_0^\infty tK_0'(xt)I_0^a(t)K_0^a(t)t^c \,{\mathrm{d}} t
&= O\Big(\int_0^\infty tK_0'(xt) \,{\mathrm{d}} t\Big)
= O\Big(\frac{1}{x^2}\Big).
\end{align}
If $0\leq a <c$ and as $x\to 0^+$, then
\begin{align}
\nonumber
\int_0^\infty K_0(xt)I_0^a(t)K_0^a(t)t^c {\mathrm{d}} t
&= \int_0^\infty \frac{K_0(xt)t^{c-a}}{2^a} {\mathrm{d}} t
+ \int_0^\infty K_0(xt)\Big[ I_0^a(t)K_0^a(t) - \frac{1}{(2t)^a}\Big]t^c {\mathrm{d}} t \\
\label{eq:KxtIK_2}
&= \frac{2^{c-2a-1}}{x^{c-a+1}} \Gamma\Big(\frac{c-a+1}{2}\Big)^2
+ O\Big(\frac{1}{x^{c-a}}\Big), \\
\nonumber
\int_0^\infty tK_0'(xt)I_0^a(t)K_0^a(t)t^c {\mathrm{d}} t
&= \int_0^\infty \frac{tK_0'(xt)t^{c-a}}{2^a} {\mathrm{d}} t
+ \int_0^\infty tK_0'(xt)\Big[ I_0^a(t)K_0^a(t) - \frac{1}{(2t)^a}\Big]t^c {\mathrm{d}} t \\
\nonumber
&= \frac{c-a+1}{2^ax} \int_0^\infty K_0(xt)t^{c-a} \,{\mathrm{d}} t
+ O\Big(\frac{1}{x^2}\Big) \\
\label{eq:K'xtIK_2}
&= O\Big(\frac{1}{x^{c-a+2}}\Big).
\end{align}

\medskip
On the real line,
we have \cite[Lemma 4.5]{MR3824579}
\[ \lim_{x\to 1^-} \frac{\int_0^\infty I_0(xt)K_0(t) \,{\mathrm{d}} t}{-\log(1-x)}
= \frac{1}{2} \]
and for $c\in\mathbb{Z}_{\geq 0}$,
\[ \lim_{x\to 1^-}(1-x)^{c+1}\int_0^\infty I_0(xt)K_0(t)t^{c+1} \,{\mathrm{d}} t
= \frac{c!}{2}
= \lim_{x\to 1^-}(1-x)^{c+1}\int_0^\infty tI_0'(xt)K_0(t)t^c \,{\mathrm{d}} t. \]
Therefore for $a\geq 1, a> c$ and $x\to 1^-$,
one has
\begin{align}
\nonumber
\int_0^\infty I_0(xt)I_0^{a-1}(t)K_0^a(t)t^c \,{\mathrm{d}} t
&= \int_0^\infty I_0(xt)K_0(t)\big[I_0^{a-1}(t)K_0^{a-1}(t)t^c\big] \,{\mathrm{d}} t \\
\nonumber
&= O\Big(\int_0^\infty I_0(xt)K_0(t) \,{\mathrm{d}} t \Big) \\
\label{eq:IxtIK_1}
&= O\big(\log (1-x)\big), \\
\nonumber
\int_0^\infty tI_0'(xt)I_0^{a-1}(t)K_0^a(t)t^c \,{\mathrm{d}} t
&= \int_0^\infty tI_0'(xt)K_0(t)\big[I_0^{a-1}(t)K_0^{a-1}(t)t^c\big] \,{\mathrm{d}} t \\
\nonumber
&= O\Big(\int_0^\infty tI'_0(xt)K_0(t) \,{\mathrm{d}} t \Big) \\
\label{eq:I'xtIK_1}
&= O\Big(\frac{1}{1-x}\Big).
\end{align}
If $c\geq a\geq 1$ and $x\to 1^-$, then
\begin{align}
\nonumber
\int_0^\infty I_0(xt)I_0^{a-1}(t)K_0^a(t)t^c \,{\mathrm{d}} t
&\begin{aligned}[t]
= &\int_0^\infty \frac{I_0(xt)K_0(t)t^{c-a+1}}{2^{a-1}} \,{\mathrm{d}} t \\
&\hspace{20pt}+ \int_0^\infty I_0(xt)K_0(t)
	\Big[ I_0^{a-1}(t)K_0^{a-1}(t) - \frac{1}{(2t)^{a-1}}\Big]t^c \,{\mathrm{d}} t
\end{aligned} \\
\label{eq:IxtIK_2}
&= \frac{(c-a)!}{2^a(1-x)^{c-a+1}}
+ o\Big(\frac{1}{(1-x)^{c-a+1}}\Big), \\
\nonumber
\int_0^\infty tI'_0(xt)I_0^{a-1}(t)K_0^a(t)t^c \,{\mathrm{d}} t
&\begin{aligned}[t]
= &\int_0^\infty \frac{tI'_0(xt)K_0(t)t^{c-a+1}}{2^{a-1}} \,{\mathrm{d}} t \\
&\hspace{20pt}+ \int_0^\infty tI'_0(xt)K_0(t)
	\Big[ I_0^{a-1}(t)K_0^{a-1}(t) - \frac{1}{(2t)^{a-1}}\Big]t^c \,{\mathrm{d}} t
\end{aligned} \\
\label{eq:I'xtIK_2}
&= \frac{(c-a+1)!}{2^a(1-x)^{c-a+2}}
+ o\Big(\frac{1}{(1-x)^{c-a+2}}\Big).
\end{align}
Notice that the error terms in the above two formulas are
of class small $o$;
it is needed in the investigation of the limit of $\omega_{2r}$
as $x\to 1^-$ below.

\subsubsection*{Evaluation of $\omega_{2r+1}(x)$ at $x=1$}
All entries of $\Omega_{2r+1}(x)$ can be evaluated at $x=1$.
We move the $2i$-th row to row $i$ in $\Omega_{2r+1}(1)$
for $1\leq i\leq r$ and then
subtract the $(r+j+1)$-st column from the $j$-th column
of the resulting matrix for $1\leq j\leq r$.
By \eqref{eq:Wronskian_IK} on the upper-left block,
we obtain
\[ \omega_{2r+1}(1) = (-1)^{\frac{r(r+1)}{2}}
\det \begin{pmatrix}
M_r & * \\
0 & M_{r+1}
\end{pmatrix}. \]

\subsubsection*{Behavior of $\omega_{2r+1}(x)$ as $x \to 0^+$}
Fix $r\geq 1$.
We move row $(2i-1)$ of $\Omega_{2r+1}(x)$
to row $i$ for $1\leq i \leq r$,
which creates a sign $(-1)^{r(r-1)/2}$
to the determinant $\omega_{2r+1}(x)$.
As $x\to 0^+$,
the resulting matrix decomposes into
$(r,r,1)\times (r,r,1)$ blocks of the form
\[\begin{pmatrix}
N_r + o(1) & O(\log x) & O\big(\frac{1}{x^r}\big) \\
O(x) & \frac{-1}{x}N_r + O(\log x) & O\big(\frac{1}{x^r}\big) \\
O(1) & O(\log x) &
\frac{1}{2x^{r+1}}\Gamma\big(\frac{r+1}{2}\big)^2
	+ O\big(\frac{1}{x^r}\big)
\end{pmatrix}\]
by direct evaluation and \eqref{eq:simple_I'xt}
in the left three blocks,
\eqref{eq:simple_Kxt} and \eqref{eq:simple_K'xt}
in the middle,
and
\eqref{eq:KxtIK_1}, \eqref{eq:K'xtIK_1},
\eqref{eq:KxtIK_2} and \eqref{eq:K'xtIK_2}
in the last column.
The leading term of $\omega_{2r+1}(x)$,
which is of order $x^{-(2r+1)}$,
comes from the diagonal blocks
and one gets
\[ \lim_{x\to 0^+} x^{2r+1}\omega_{2r+1}(x)
= (-1)^{\frac{r(r+1)}{2}} \frac{1}{2}
	\Gamma\Big(\frac{r+1}{2}\Big)^2 {\det}^2 N_r. \]

\subsubsection*{Behavior of $\omega_{2r}(x)$ as $x \to 1^-$}
Fix $r\geq 2$.
We move $2i$-th row of $\Omega_{2r}(x)$
to row $i$ for $i=1,2,\cdots, (r-1)$
and $r$-th column to the last,
which adds a sign $(-1)^{r(r+1)/2}$
to the determinant $\omega_{2r}(x)$.
We subtract $j$-th column by $(r+j)$-th
for $j = 1,2,\cdots, (r-1)$.
As $x\to 1^-$,
the resulting matrix decomposes into
$(r-1,r,1)\times (r-1,r,1)$ blocks of the form
\[\begin{pmatrix}
N_{r-1} + o(1) & O(1) & O\big(\frac{1}{(1-x)^{r-1}}\big) \\
0 & N_r + o(1) & O\big(\frac{1}{(1-x)^{r-1}}\big) \\
O(1) & O(1) & \frac{(r-1)!}{2^r(1-x)^r} + o\big(\frac{1}{(1-x)^r}\big)
\end{pmatrix}\]
by \eqref{eq:Wronskian_IK} and direct evaluation
in the left three blocks,
direct evaluation
in the middle,
and
\eqref{eq:IxtIK_1}, \eqref{eq:I'xtIK_1},
\eqref{eq:IxtIK_2} and \eqref{eq:I'xtIK_2}
in the last column.
The leading term of $\omega_{2r}(x)$,
which is of order $(1-x)^{-r}$,
comes from the diagonal blocks.
It yields
\[ \lim_{x\to 1^-} (1-x)^r\omega_{2r}(x)
= (-1)^{\frac{r(r+1)}{2}} \frac{(r-1)!}{2^r} \det N_{r-1} \det N_r. \]

\subsubsection*{Behavior of $\omega_{2r}(x)$ as $x \to 0^+$}
Fix $r\geq 2$.
We move row $(2i-1)$ of $\Omega_{2r}(x)$
to row $i$ for $1\leq i \leq r$,
which adds a sign $(-1)^{r(r-1)/2}$
to the determinant $\omega_{2r}(x)$.
As $x\to 0^+$,
the resulting matrix decomposes into
four blocks of equal size of the form
\[\begin{pmatrix}
M_r + o(1) & O(\log x) \\
O(x) & \frac{-1}{x}M_r + O(\log x)
\end{pmatrix}\]
by direct evaluation and \eqref{eq:simple_I'xt}
in the left two blocks
and \eqref{eq:simple_Kxt} and \eqref{eq:simple_K'xt}
in the right.
This leads to
\[ \lim_{x\to 0^+} x^r\omega_{2r}(x)
= (-1)^{\frac{r(r+1)}{2}} {\det}^2 M_r. \]

\begin{remark}
Proposition \ref{prop:joint_pmatrix_minus} indeed
holds for $\omega_2(x)$ by the same analysis
if we set $\det N_0=1$;
it is consistent with the relation \eqref{eq:N-M} for $r=1$.
\end{remark}

\end{appendices}


\bibliography{cite}

\end{document}